\documentclass[12pt]{amsart}
\usepackage{graphicx}
\usepackage{color}
\usepackage{verbatim}
\newtheorem{theorem}{Theorem}[section]
\newtheorem{corollary}[theorem]{Corollary}
\newtheorem{remark}[theorem]{Remark}

\newtheorem{lemma}[theorem]{Lemma}
\newtheorem{proposition}[theorem]{Proposition}
\newtheorem{definition}{Definition}[section]

\numberwithin{equation}{section}
\author{Xiaoli Han, Jiayu Li, Jun Sun, Jie Zhou}

\address{Xiaoli Han, Department of Mathematical Sciences, Tsinghua University, Beijing 100084, P. R. of China.}
\email{hanxiaoli@mail.tsinghua.edu.cn}

\address{Jiayu Li, School of Mathematical Sciences, University of Science and Technology of China, Hefei 230026, P. R. China}
\email{jiayuli@ustc.edu.cn}

\address{Jun Sun, School of Mathematics and Statistics, Wuhan University, Wuhan 430072, P. R. of China.}
\email{sunjun@whu.edu.cn}

\address{Jie Zhou, School of Mathematical Sciences, Captital Normal University, Beijing 100048, P. R. of China.}
\email{zhoujiemath@cnu.edu.cn}

\thanks {The first author is supported by National Key R$\&$D Program of China 2022YFA1005400. The second author is supported by NSFC No. 12531002, 12431004, 11721101. The third author is supported by NSFC No. 12671074, 12531002, 12271039. The fourth author is supported by NSFC No. 12301077, 12531001.}

\begin{document}

	\title[A Bernstein type theorem for a two dimensional translating soliton]
	{A Bernstein type theorem for a two dimensional translating soliton in ${\mathbb R}^4$ with finite total curvature}

\begin{abstract}
		  In this paper, we prove that any two dimensional complete translating soliton in ${\mathbb R}^4$ with finite total curvature is a plane. It is in fact a consequence of our main result.
\end{abstract}
\maketitle

	{\bf Mathematics Subject Classification (2020):} 53C42 (primary), 53E10 (secondary).

	\section{Introduction}
	
	\allowdisplaybreaks

\vspace{.1in}

Given an initial immersion $\Sigma_0=F_0(\Sigma)$,
we consider a one-parameter family of smooth maps
$F_t=F(\cdot, t):\Sigma\rightarrow M$,
	with the corresponding images
	$\Sigma_t=F_t(\Sigma)$
	satisfying the mean curvature flow equation (MCF):
	\begin{equation}\label{meaneqn}
		\left\{\begin{array}{lll}
			\displaystyle \frac{d}{dt}F(x,t)={\bf H}(x,t)\\
			\displaystyle \,\,\,\,\,\,F(x,0)=F_0(x).
		\end{array}\right.
	\end{equation}
	Here ${\bf H}(x,t)$ is the mean curvature vector of $\Sigma_t$ at $F(x,t)$ in $M$.

The mean curvature flow is the negative gradient flow of the area functional. If the flow exists globally and converges as time tends to  infinity, then the limit must be a minimal submanifold. It plays an important role in the study of the existence of minimal subanifolds in Riemannian manifolds, as well as in the study of the topology of submanifolds.

In general, the mean curvature flow develops singularities whenever the maximum norm of the second fundamental form blows up (\cite{Huisken1}).
Using the blow-up rate of the second fundamental form, Huisken classified the singularities of the mean curvature flow into two types (\cite{Huisken2}). Using the monotonicity formula introduced in the same work (\cite{Huisken2}), he proved that the smooth blow-up limits of Type I singularities are self-shrinking. Colding-Minicozzi (\cite{CM1}) further demonstrated that self-shrinkers serve as canonical singularity models not only for Type I singularities but also for generic singularities arising in mean curvature flow. They also established the uniqueness of cylindrical singularities of mean curvature flows in ${\mathbb R}^{n+1}$ and investigated the singular set of flows with generic singularities (\cite{CM2}, \cite{CM3}). Subsequently, they extended their analysis to singularities of mean curvature flows of arbitrary codimension (\cite{CM4}). Recently, Bamler-Kleiner proved the multiplicity one conjecture for mean curvature flows in ${\mathbb R}^3$, which states that every tangent flow of a 2-dimensional mean curvature flow in ${\mathbb R}^3$ is a smooth self-similar shrinker of one multiplicity (\cite{BK}).

By standard blow-up analysis (\cite{HuS2}), we see that any blow-up limit of a Type II singularity of the mean curvature flow yields an eternal solution to the mean curvature flow with a bounded second fundamental form in ${\mathbb C}^2$ .
	
	\vspace{.1in}
	
As a special model of Type II singularity, translating solitons play an important role in the study of mean curvature flow.
	
	\begin{definition}
		A submanifold $\Sigma^n$ in ${\mathbb R}^{n+k}$ is a \textbf{translating soliton} for the mean curvature flow if there is a constant unit vector ${\bf T}\in {\mathbb R}^{n+k}$, such that
		\begin{align*}
			\Sigma^n_t=\Sigma^n+t{\bf T}
		\end{align*}
		is a solution to the mean curvature flow. Equivalently, it holds
		\begin{equation}\label{e-TS}
			{\bf T}^{\perp}={\bf H}
		\end{equation}
		on $\Sigma^n$ everywhere, where ${\bf H}$ is the mean curvature vector of $\Sigma^n$ in ${\mathbb R}^{n+k}$.
	\end{definition}
	
    Hamilton proved that any strictly convex eternal solution to the mean curvature flow in ${\mathbb R}^{n+1}$ where the mean curvature attains its maximum at some point in space-time must be a translating soliton (\cite{Hamilton}). Using this result together with refined convexity estimates, Huisken-Sinestrari proved that the blow-up limits of Type II singularities for mean convex mean curvature flow are strictly convex translating solitons or products of strictly convex translating solitons with Euclidean subspaces (\cite{HuS1,HuS2}). Wang showed that a convex translating soliton to the mean curvature flow must be rotationally symmetric if it arises as the limit flow of a mean convex mean curvature flow in ${\mathbb R}^3$ (\cite{WangXJ}). Jian and his collaborators also studied properties of rotationally symmetric convex translating solitons (\cite{GHH,JLC,JJLW}, etc.). Bao and Shi, as well as Ma and Miquel, established several Bernstein type theorems for translating solitons in ${\mathbb R}^{n+1}$ (\cite{BS,MM}). Lynch and Tinaglia proved that any complete embedded translating soliton in ${\mathbb R}^3$ with finite total curvature is a vertical plane (\cite{LT}; see also \cite{Khan}). Xin obtained various rigidity results for higher codimension translating solitons (\cite{Xin}). Neves-Tian established a rigidity result for Lagrangian translating solitons in ${\mathbb R}^4$ (\cite{NT}). They showed that under certain natural assumptions-specifically, if the first Betti number of the Lagrangian translating soliton is finite, and either $\int_{\Sigma}|{\bf H}|^2d\mu$ is finite or $\Sigma$ is static and almost calibrated, then the translating soliton must be a plane.

    	\vspace{.1in}
	
Our main result in this paper is as follows:

    		\begin{theorem}\label{maintheorem}
		Let $\Sigma$ be a properly immersed  translating soliton in ${\mathbb R}^4$ satisfying
\begin{enumerate}
\item ${\rm {Area}}(\Sigma\cap B_R)\le \Lambda R^2$ for all $R>0$,
\item the genus of $\Sigma$ is finite,
\item there exists $R_k\to \infty$ such that
$$\lim_{k\to \infty}\int_{(A_{\frac{R_k}{2},4R_k})\cap \Sigma}|{\bf A}|^2d\mu=0,$$
where $A_{\frac{R_k}{2},4R_k}=B_{4R_k}(0)\backslash B_{\frac{R_k}{2}}(0)$,
\item $\lim_{|x|\to \infty} \int_{B_1(x)\cap \Sigma}|{\bf A}|^2d\mu\to 0.$
\end{enumerate} Then $\Sigma$ is a plane.
	\end{theorem}

By Corollary \ref{cor-finite-total-curvature}, it is easy to see that

\begin{theorem}\label{maintheorem-1}
Any 2 dimensional complete translating soliton in ${\mathbb R}^4$ with finite total curvature is a plane.
	\end{theorem}

    The following corollary is clear.
\begin{corollary}\label{thm-3-dim}
Any complete translating soliton in ${\mathbb R}^3$ with finite total curvature is a plane.
	\end{corollary}

\begin{remark}\label{rmk-total curvature}
 In \cite{LT}, Lynch and Tinaglia proved Corollary \ref{thm-3-dim}
 under the additional assumption that the surface is embedded.
 \end{remark}

By Corollary \ref{cor-finite-total-curvature}, we see that the total curvature of a surface $\Sigma$ is finite if and only if $\Sigma$ has quadratic area growth, finite genus and mean curvature vector squared integrable. So we can restate our Theorem \ref{maintheorem-1}.

\begin{theorem}\label{maintheorem-H}
		Let $\Sigma$ be a complete 2 dimensional translating soliton in ${\mathbb R}^4$ satisfying
        \begin{enumerate}
		\item  ${\rm Area}(\Sigma\cap B_R(x))\leq \Lambda R^2$ for all $R>0$ and $x\in {\mathbb R}^4$;
		
		\item  the genus of $\Sigma$ is finite,

        \item $\int_{\Sigma}|{\bf H}|^2 d\mu<\infty$,
        \end{enumerate}
	 for some positive constant $\Lambda$. Then $\Sigma$ is a plane.
	\end{theorem}

Because the assumption that the mean curvature is squared integrable implies that the surface has a lower bound for the intrinsic ball (Proposition \ref{prop-iso-proper}), we see that it has finitely many ends. If we assume in addition that the first Betti number is finite, then genus must be finite. Hence, Theorem \ref{maintheorem-H} improves the theorem in \cite{NT}.

	\vspace{.1in}

In the proof of the theorem, we use the properties of the K\"ahler angle $\alpha$ of the translating soliton with respect to a compatible complex structure in the ambient space ${\mathbb R}^4$ (see Section 2). As is well known, such compatible complex structures are parametrized  by the $2$-sphere.

A key ingredient in the proof of the main theorem is Theorem \ref{prop-alpha-decay}, in which we establish, for each end of $\Sigma$, a decay estimate for $\alpha$ associated with some compatible complex structure.
To prove Theorem \ref{prop-alpha-decay}, we carefully examine the level sets of $\cos \alpha$. For this purpose, we make use of the equation derived in \cite{CL1} (see also \cite{HL2}). However, since equation (\ref{e-cosalpha}) is nonlinear, much more additional care is required in the analysis.

\vspace{.1in}

In Section 5, we verify that the translating solitons constructed by Joyce-Lee-Tsui \cite{JLT} are symplectic with respect to the compatible complex structure $J_3$ on ${\mathbb C}^2$ given by (\ref{e-J-1}). Indeed, one can make $\cos\alpha$ arbitrarily close to $1$.  Moreover, these examples are shown to have uniformly bounded second fundamental form, genus 0, first Betti number 0 and extrinsic quadratic area growth. In other words, there do exist nontrivial translating solitons with K\"ahler angle arbitrarily small, for which
\begin{equation*}
      \lim_{r\to\infty}\int_{\Sigma\cap\{r<|{\bf x}|<2r\}}|{\bf H}|^2d\mu>0.
\end{equation*}
In particular, the assumptions (1), (2) and (4) of Theorem \ref{maintheorem} are satisfied while the assumption (3) fails. Since the level sets of $\cos\alpha$ are lines in this case, the function does not admit a limit at infinity. This absence of a limiting value is precisely what our proof relies upon.

\vspace{.1in}
	
	The subsequent sections are organized as follows: in Section 2, we provide some preliminary results which will be used later; in Section 3, we obtain the key decay estimate for the K\"ahler angle; in Section 4, we prove the main theorem of this paper; in Section 5, we examine the examples of a family of symplectic translating solitons.

    \vspace{.2in}

	\section{Preliminaries}
	
	\vspace{.1in}
	
In this section, we first collect several useful formulas for later use. Consider a translating soliton $\Sigma$ for the mean curvature flow. Throughout this part, we fix the vector ${\bf T}=(1, 0, 0, 0)$. Here we list some notations of balls which will be used in the paper.  By an immersed surface $\Sigma\subset \mathbb{R}^n$, we mean there exists an immersion map $f:\Sigma\to \mathbb{R}^n$ and the induced metric $g=f^*g_{\mathbb{R}^n}$ is given to $\Sigma$.  For any $x\in \Sigma$, $\hat{B}_r(x)$ will denote the intrinsic ball of radius $r$ centered at $x$. For Any $x\in \Sigma$, we will also use the notation $B_r(x)$ to denote the ball in $\mathbb{R}^n$ of radius $r$ centered at $f(x)$. By
    $\Sigma\cap B_r(x)$, we always mean $\Sigma\cap f^{-1}(B_r(f(x_))$. By $\Sigma^r_x$, we mean the connected component of $\Sigma\cap f^{-1}(B_r(f(x)))$ containing $x$.

It is known that the complex structure on ${\mathbb R}^4$ is given by
\begin{equation}\label{e-J}
{\mathcal J}=\left\{xJ_1+yJ_2+zJ_3|x^2+y^2+z^2=1\right\},
\end{equation}
where
\begin{equation}\label{e-J-10}
J_1=\left(\begin{array}{cccc}
    0 & -1 & 0 & 0 \\
    1 & 0  & 0 & 0 \\
    0 & 0  & 0 & -1 \\
    0 & 0  & 1 & 0 \\
\end{array}
\right),
J_2=\left(\begin{array}{cccc}
    0 & 0 & -1 & 0 \\
    0 & 0  & 0 & 1 \\
    1 & 0  & 0 & 0 \\
    0 & -1  & 0 & 0 \\
\end{array}
\right),
J_3=\left(\begin{array}{cccc}
    0 & 0  & 0 & -1 \\
    0 & 0  & -1 & 0 \\
    0 & 1  & 0 & 0 \\
    1 & 0  & 0 & 0 \\
\end{array}
\right).
\end{equation}

Let $\Sigma$ be a real oriented surface immersed in ${\mathbb R}^4$. For each $J\in {\mathcal J}$,  the K\"ahler angle $\alpha_J$  of $\Sigma$  with respect to $J$ is defined by (\cite{CW})
\begin{equation}\label{e-cosalpha-J}
\cos\alpha_J:=\langle Je_1,e_2\rangle,
\end{equation}
where $\{e_1,e_2\}$ is an oriented local orthonormal frame of $T\Sigma$. As $\{e_1, J_1e_1, J_2e_1, J_3e_1\}$ forms an orthonormal basis of ${\mathbb R}^4$, one readily sees that
\begin{equation}\label{e-complex}
\cos\alpha_{J_1}^2+\cos\alpha_{J_2}^2+\cos\alpha_{J_3}^2=1.
\end{equation}

	\vspace{.1in}

We recall the elliptic equations satisfied by geometric quantities on a translating soliton $\Sigma$ in ${\mathbb R}^4$.
	
	\begin{proposition} (\cite{HL2})\label{pro1}
		On a 2 dimensional translating soliton $\Sigma$ in ${\mathbb R}^4$, the following identity holds:

		\begin{equation}\label{e-x1}
			\Delta\langle {\bf T},x\rangle=\Delta x_1=|{\bf H}|^2,
		\end{equation}
		where $x$ is the position vector in ${\mathbb R}^4$; for any $J\in {\mathcal J}$,
		\begin{equation}\label{e-cosalpha}
			\Delta\cos\alpha_J+\langle\nabla x_1,\nabla\cos\alpha_J\rangle=-|\overline{\nabla}J_{\Sigma}|^2\cos\alpha_J,
		\end{equation}
		where $|\overline{\nabla}J_{\Sigma}|^2=|h^1_{2i}+h^2_{1i}|^2+|h^1_{1i}-h^2_{2i}|^2$ which is independent of the choice of oriented orthonormal frame;
	
		\begin{equation}\label{e-H2}
			\Delta|{\bf H}|^2+\langle\nabla x_1,\nabla|{\bf H}|^2\rangle\geq2|\nabla{\bf H}|^2-2|{\bf A}|^2|{\bf H}|^2;
		\end{equation}

        		\begin{equation}\label{e-A2}
			\Delta|{\bf A}|^2+\langle\nabla x_1,\nabla|{\bf A}|^2\rangle\geq2|\nabla{\bf H}|^2-3|{\bf A}|^4.
		\end{equation}
	\end{proposition}

	\begin{corollary}
		On a translating soliton $\Sigma$ in ${\mathbb R}^4$, we have for any $J\in {\mathcal J}$
		\begin{equation}\label{e-sinalpha}
			\Delta\sin^2\frac{\alpha_J}{2}+\langle\nabla x_1,\nabla\sin^2\frac{\alpha_J}{2}\rangle=\frac{1}{2}|\overline{\nabla}J_{\Sigma}|^2\cos\alpha_J.
		\end{equation}
		Furthermore, at the point where $\alpha_J\neq0$, we have
		\begin{equation}\label{e-alpha}
			\Delta\alpha_J+\langle\nabla x_1,\nabla\alpha_J\rangle=\frac{\cos\alpha_J}{\sin\alpha_J}(|\overline{\nabla}J_{\Sigma}|^2-|\nabla\alpha_J|^2).
		\end{equation}
	\end{corollary}

	Let ${\bf V}$ denote the tangential part of ${\bf T}$ on a translating soliton $\Sigma$, i.e.,
	\begin{equation*}
		{\bf T}={\bf H}+{\bf V},
	\end{equation*}
	then we have ${\bf V}=\nabla x_1$ and
	
	\begin{proposition}\label{prop-A-H}(\cite{HL2})
		Suppose that $\Sigma$ is a translating soliton in ${\mathbb R}^4$. Then the following identity holds on $\Sigma$:
		\begin{equation*}
			|{\bf A}|^2=|{\bf H}|^2+2\frac{|\nabla{\bf H}|^2}{|{\bf V}|^2}+\frac{{\bf V}\cdot \nabla|{\bf H}|^2}{|{\bf V}|^2},
		\end{equation*}
		at the point where ${\bf V}$ does not vanish.
	\end{proposition}

	\begin{corollary}\label{cor-plane}
		Suppose that $\Sigma$ is a translating soliton in ${\mathbb R}^4$. If $\Sigma$  is minimal, then it is a plane.
	\end{corollary}

	Fix any $J\in {\mathcal J}$. In a local orthonormal frame $\{e_1,e_2,v_3,v_4\}$, where $\{e_1,e_2\}$ spans $T\Sigma$ and $\{v_3,v_4\}$ spans $N\Sigma$, denote ${\bf W}=\partial_2\alpha_J v_3+\partial_1\alpha_J v_4$. It is easy to see $|{\bf W}|^2=|\nabla\alpha_J|^2$.
	
	\begin{lemma}\label{lemma-HJW} (\cite{HL3})
		On a surface $\Sigma$ in ${\mathbb R}^4$, we have
		\begin{equation}\label{e-HJW}
			|\overline{\nabla}J_{\Sigma}|^2=|{\bf H}|^2+2{\bf H}\cdot{\bf W}+2|{\bf W}|^2.
		\end{equation}
	\end{lemma}

	\begin{corollary}\label{cor-HJA}
		On a surface $\Sigma$ in ${\mathbb R}^4$, we have
		\begin{equation}\label{e-HJ}
		 2|{\bf A}|^2\geq	|\overline{\nabla}J_{\Sigma}|^2\geq \frac{1}{2}|{\bf H}|^2,
		\end{equation}
		and for any $J\in {\mathcal J}$
		\begin{equation}\label{e-J-nablaalpha}
			|\overline{\nabla}J_{\Sigma}|^2\geq |\nabla\alpha_J|^2.
		\end{equation}
	\end{corollary}
	
	\vspace{.1in}
	The following proposition asserts that the assumption of non-concentration of  $L^m$ mean curvature at infinity  implies a uniform isoperimetric inequality at infinity. In addition, the condition of extrinsically locally finite volume yields properness and the finiteness of ends.
\begin{proposition}\label{prop-iso-proper} There exists $\epsilon_0=\epsilon_0(m,n)>0$ such that the following holds.

	Assume  $\Sigma$ is an $m$-dimensional complete immersed submanifold in ${\mathbb R}^n$.
    \begin{enumerate}
    \item If $A\subset \Sigma$ is a set of finite perimeter such that  $$\int_{A}|{\bf H}|^md\mathrm{vol}\le \epsilon_0,$$
then the isoperimetric inequality holds
\begin{align}\label{isoperimetric inequality}
\left(\mathcal{H}^{m}(A)\right)^{1-\frac{1}{m}}\le D_1 \mathcal{H}^{m-1}(\partial A)
\end{align}
for some  $D_1=D_1(m,n)<\infty$.
\item  If there exists $r_1>0$ such that
        \begin{align}\label{finite total mean curvature}
        \lim_{R\to \infty}\sup_{x\in \Sigma\backslash \hat{B}_{R}(p_0)}\int_{\hat{B}_{r_1}(x)}|{\bf H}|^md\mathrm{vol}\le \epsilon_0,
        \end{align}
 then there exists a compact subset $K\subset \Sigma$ such that  for any $\hat{B}_r(x)\subset \Sigma\backslash K$ with $r\le r_1$, there holds
\begin{align}\label{volume lower bound}
\mathrm{Vol}(\hat{B}_r(x))\ge \kappa r^m,
\end{align}
where $\kappa=\frac{1}{m^mD_1^m}$.

\item In addition to \eqref{finite total mean curvature},  if we also  assume
\begin{align*}
{\rm{Vol}}(\Sigma\cap B_R(0))<+\infty \  \ {\rm for \ all} \ R>0,
\end{align*}
then $\Sigma$ is properly immersed in $\mathbb{R}^n$ and only has  finitely many ends.
\end{enumerate}

	\end{proposition}
    \begin{proof} Applying the Michael-Simon Sobolev inequality (\cite{MS}, \cite{B}), we have for any  smooth  positive function $f$ on any compact subset $\hat{\Sigma}\subset\Sigma$, there holds
    \begin{align*}
    \int_{\hat{\Sigma}}\sqrt{|\nabla f|^2+f^2|{\bf H}|^2}+\int_{\partial\hat{\Sigma}}f\ge C(m,n)\|f\|_{L^\frac{m}{m-1}(\hat{\Sigma})}.
    \end{align*}
    By the smoothness of \(\Sigma\) , the smooth functions are dense in $BV(\Sigma)$. Therefore, given any set $A\subset \Sigma$ with finite perimeter, substituting $f=\chi_A$  into the Michael-Simon Sobolev inequality yields
    \begin{align*}
    \|{\bf H}\|_{L^1(A)}+\mathcal{H}^{m-1}(\partial A)\ge C(m,n)\left(\mathcal{H}^m(A)\right)^{1-\frac{1}{m}}.
    \end{align*}
    Take $\epsilon_0=\left(\frac{C(m,n)}{2}\right)^{m}$.
    Since $\int_{A}|{\bf H}|^md\rm{vol}\le \epsilon_0$, we know  $\|{\bf H}\|_{L^1(A)}\le \frac{C(m,n)}{2}\left(\mathcal{H}^m(A)\right)^{1-\frac{1}{m}}$ and hence
    \begin{align*}
    \mathcal{H}^{m-1}(\partial A)\ge \frac{C(m,n)}{2}\left(\mathcal{H}^m(A)\right)^{1-\frac{1}{m}}.
    \end{align*}
    Thus  \eqref{isoperimetric inequality} follows by taking $D_1=\frac{2}{C(m,n)}$.
    Especially,\eqref{finite total mean curvature} implies there exists a compact subset $K\subset \Sigma$ such that  for any $\hat{B}_r(x)\subset \Sigma\backslash K$ with $r\le r_1$, the coarea formula implies  $\phi(s)=\mathcal{H}^m(\hat{B}_s(x))$ satisfies
    \begin{align*}
    \phi'(s)=\mathcal{H}^{m-1}(\partial \hat{B}_r(x)) \ge \frac{1}{D_1}\phi^{1-\frac{1}{m}}(s).
    \end{align*}
    Combining with $\phi(0)=0$, we get
    $$\phi(r)\ge \frac{1}{m^mD_1^m}r^m,\quad \forall r\le r_1.$$
    The volume lower bound \eqref{volume lower bound} follows by taking $\kappa=\frac{1}{m^mD_1^m}$.

 Now, we prove the properness  by contradiction. Otherwise,  there exists $R>0$ such that $\overline{B_{R}(0)}\cap \Sigma$ is not compact.  Since $\Sigma$ is complete, we get a  sequence of $x_i\in\Sigma\cap \overline{B_R(0)}$ such that
 \begin{align*}
 d_\Sigma(x_i,K)\to \infty \quad \text{ and } d_\Sigma(x_i,x_j)\ge 2r_1.
 \end{align*}
 Without loss of generality, we may assume that $\hat{B}_{r_1}(x_i)\subset \Sigma\backslash K$ for each $i$. Thus
 $\{\hat{B}_{r_1}(x_i)\}_{i=1}^\infty$ are disjoint family of balls in $B_{2r_1+R}(0)\cap \Sigma$, which further implies the contradiction
 \begin{align*}
 \infty>{\rm{Vol}}(\Sigma\cap B_{2r_1+R}(0))\ge \sum_{i=1}^\infty {\rm{Vol}}(\hat{B}_{r_1}(x_i))\ge \Sigma_{i=1}^\infty \kappa r_1^m=\infty.
 \end{align*}
 If $\Sigma$ has infinite ends in $\mathbb{R}^n$, then there exists $B_R(0)\supset K$ such that $\Sigma\backslash B_R(0)$ contains infinite many non-compact connected components $\{E_i\}$. By taking $y_i\in E_i$ such that $d(y_i,B_R(0))=r_1$,  we get infinite many disjoint balls $\{\hat{B}_{r_0}(y_i)\subset E_i\}$ in $B_{R+2}(0)\cap \Sigma$, which again implies the contradiction
 \begin{align*}
 \infty>{\rm{Vol}}(\Sigma\cap B_{2r_1+R}(0))\ge \sum_{i=1}^\infty {\rm{Vol}}(\hat{B}_{r_1}(y_i))=\infty.
 \end{align*}
    \end{proof}

\begin{remark}
Note that the uniformly symplectic assumption $\cos\alpha \ge \delta > 0$ yields an isoperimetric inequality (\cite{CHLS}) on $\Sigma$.
If we assume further that $\mathrm{Area}(\Sigma\cap B_R(0))<\infty, \forall R<\infty$, then $\Sigma$ is properly immersed in $\mathbb{R}^4$ and only has  finitely many ends.  In particular, any translating soliton arising as blow up limit of sympelctic mean curvature flow must be properly immersed and has  finitely many ends. The same conclusion also holds for almost Lagrangian mean curvature flow.
    \end{remark}

\begin{corollary}\label{cor-finite-total-curvature} Assume $\Sigma\subset \mathbb{R}^n$ is a complete immersed surface. Then $\Sigma$ has finite total curvature, i.e.,
\begin{align*}
\int_{\Sigma}|{\bf A}|^2d\mu<\infty.
\end{align*}
if and only if
\begin{enumerate}
\item ${\rm {Area}}(\Sigma\cap B_R)\le \Lambda R^2$ for all $R>0$,
\item the genus of $\Sigma$ is finite,
\item $\int_{\Sigma}|{\bf H}|^2d\mu<\infty,$
\end{enumerate}
for some constant $\Lambda<+\infty$. Moreover, if one of the two equivalent statements holds, then $\Sigma$ is properly immersed in $\mathbb{R}^n$ and only has finitely many ends.
\end{corollary}

\begin{proof} ($\Longleftarrow$)
 By Proposition \ref{prop-iso-proper}, $\Sigma$ is properly immersed in ${\mathbb R}^n$. Hence we may apply a modified version of Ilmanen's local Gauss-Bonnet formula (\cite{Il1}) to obtain, for any $r<R$
			\begin{eqnarray*}
				\frac{1}{2}\int_{\Sigma\cap B_{r}(0)}|{\bf A}|^2
				&\leq &\int_{\Sigma\cap B_{R}(0)}|{\bf H}|^2+8\pi g(\Sigma\cap B_R(0))-8\pi c'(\Sigma\cap B_R(0))\\
				& &+\frac{48\pi \Lambda_1 R^2}{(R-r)^2},
			\end{eqnarray*}
			where
			\begin{equation*}
				\Lambda_1=\sup_{S\in [r,R]}\frac{{\rm Area}(\Sigma\cap B_{S}(0))}{\pi S^2}\leq \Lambda
			\end{equation*}
			and $c'(\Sigma\cap B_R(0))$ is the number of components of $\Sigma\cap B_R(0)$ that meet $B_r$. Note that $g(\Sigma\cap B_R(0))$ is uniformly bounded by the genus $g(\Sigma)$. Hence
			\begin{eqnarray*}
				\frac{1}{2}\int_{\Sigma\cap B_{r}(0)}|{\bf A}|^2
				\leq \int_{\Sigma\cap B_{R}(0)}|{\bf H}|^2+8\pi g(\Sigma)+\frac{48\pi \Lambda R^2}{(R-r)^2}.
			\end{eqnarray*}
			Letting $R\to\infty$ first and then $r\to\infty$ , we obtain
			\begin{equation*}
				\int_{\Sigma}|{\bf A}|^2\leq 2\int_{\Sigma}|{\bf H}|^2+ 16\pi g(\Sigma)+96\pi \Lambda<+\infty.
			\end{equation*}

        \vspace{.1in}

        ($\Longrightarrow$) If $\Sigma$ has finite total curvature, then (3) follows immediately. By Huber's theorem, $\Sigma$ is conformal to a compact Riemann surface with finitely many points removed. Consequently, $\Sigma$ is properly immersed in $\mathbb{R}^n$ with finite genus and  quadratic area growth \cite[Corollary 4.2.5]{MS1}.
\end{proof}
Recall that (\cite{SY}) the isoperimetric inequality  is equivalent to the Sobolev inequality (note that $n=2$ in our case)
\begin{equation}\label{e-soblev}
    \left(\int_{\Omega}v^2\right)^{\frac{1}{2}}\leq D_1\int_{\Omega}|\nabla v|
\end{equation}
for any $v\in W_0^{1,1}(\Omega)$. A direct calculation using the H\"older inequality yields, for all $q\geq 2$, that
\begin{equation}\label{e-soblev-2}
    \left(\int_{\Omega}v^q\right)^{\frac{1}{q}}\leq \frac{q}{2}D_1|\Omega|^{\frac{1}{q}}\left(\int_{\Omega}|\nabla v|^2\right)^{\frac{1}{2}}
\end{equation}
for any $v\in W_0^{1,q}(\Omega)$ and $v\geq 0$. Here $|\Omega|$ denote the area of $\Omega$. Actually, we have the following improved Sobolev inequality:

\begin{proposition}\label{prop-sob}
    Let $\Omega$ be a bounded domain on $\Sigma$ with smooth boundary such that the isoperimetric inequality
\begin{equation}
			\left({\mathcal H}^2(A)\right)^{\frac{1}{2}}\leq D_1{\mathcal H}^1(\partial A)
		\end{equation}
        holds for any set $A\subset \Omega$  of finite perimeter.
	 Then for any  $v\in W_0^{1,q}(\Omega)$, $v\geq 0$ and $q\geq 2$, we have
    \begin{equation}\label{e-soblev-3}
    \left(\int_{\Omega}v^q\right)^{\frac{1}{q}}\leq Cq^{\frac{1}{2}}|\Omega|^{\frac{1}{q}}\left(\int_{\Omega}|\nabla v|^2\right)^{\frac{1}{2}}
\end{equation}
for some constant $C$ depending only on $D_1$.
\end{proposition}
\begin{proof}
The result follows from the  P\'olya-Szeg\"o rearrangement inequality (See \cite{MonSe}, \cite{AB1} for exmaple) . We include a sketch of the proof to clarify the power $\frac{1}{2}$ on $q$, which is useful for later application. By approximating (see \cite[Lemma 3.6]{MonSe}, \cite[Lemma 4]{AB1}), we assume  $v$ is a nonnegative compactly supported  Lipschitz function such that  $\mathcal{H}^2(\{x\in \Omega| |\nabla v|(x)=0, v(x)>0\})=0$.
Denote the distribution function and rearrangement of $v$ by
\begin{align*}
\lambda_v(t)=\mathcal{H}^2(\Omega_t),  \text{ where }  \Omega_t=\{x\in \Omega| |v|>t\} \quad \text{ and } \quad v^*(s)=\inf\{t| \lambda_v(t)\le s\}, s\ge 0.
\end{align*}
Then  $v^*(s)=0$ for $s\ge M=:\mathcal{H}^2(\Omega)$ and $v^*|_{[0,M]}$ is strictly decreasing and absolutely continuous \cite[Lemma 5]{AB1}.

Using the isoperimetric inequality, by the same proof of  \cite[Theorem 1]{AB1}, there holds
the P\'olya-Szeg\"o rearrangement inequality
\begin{align*}
D_1^{-2}\int_0^M(\frac{dv^*}{ds})^2sds\le
\int_{\Omega}|\nabla v|^2d\mathcal{H}^2.
\end{align*}
For convenience, we define $s=s(\tau)=Me^{-\tau}$ and $y(\tau)=v^*(s(\tau))$, then the above estimate becomes
\begin{align*}
D_1^{-2}\int_0^\infty(y'(\tau))^2d\tau=D_1^{-2}\int_0^M(\frac{dv^*}{ds})^2sds\le \int_{\Omega}|\nabla v|^2d\mathcal{H}^2.
\end{align*}
Note that $y(0)=v^*(M)=0$. We know
\begin{align*}
y(\tau)=\int_0^\tau y'\le \|y'\|_{L^2}\sqrt{\tau}\le D_1\|\nabla v\|_{L^2(\Omega)}\sqrt{\tau}.
\end{align*}
So, we get
\begin{align*}
\int_{\Omega}v^q&=q\int_0^{\infty}t^{q-1}\lambda_v(t)dt\\
&=q\int_0^{\infty}t^{q-1}\lambda_{v^*}(t)dt=\int_{[0,M]}v^{*q}ds\\
&=\int_{0}^\infty( y(\tau))^{q}Me^{-\tau} d\tau\\
&\le \int_{0}^\infty\left(D_1\|\nabla v\|_{L^2(\Omega)}\sqrt{\tau}\right)^{q}Me^{-\tau} d\tau,
\end{align*}
and hence
\begin{align*}
\|v\|_{L^q(\Omega)}\le D_1|\Omega|^{\frac{1}{q}}\|\nabla v\|_{L^2(\Omega)}\left(\int_{0}^\infty\tau^{\frac{q}{2}}e^{-\tau}d\tau\right)^{\frac{1}{q}}= D_1|\Omega|^{\frac{1}{q}}\|\nabla v\|_{L^2(\Omega)}\left(\Gamma(\frac{q}{2}+1)\right)^{\frac{1}{q}}.
\end{align*}

By the Stirling formula $\Gamma(z+1)\sim \sqrt{2\pi z}\cdot(\frac{z}{e})^z$ as $z\to \infty$, there holds
\begin{align*}
\left(\Gamma(\frac{q}{2}+1)\right)^{\frac{1}{q}}\sim\left[\sqrt{\pi q}(\frac{q}{2e})^{\frac{q}{2}}\right]^{\frac{1}{q}}\sim(2e)^{-\frac{1}{2}}q^{\frac{1}{2}} \quad \text{as } q\to \infty.
\end{align*}
As a result, we get
\begin{align*}
\|v\|_{L^q(\Omega)}\le C_1D_1|\Omega|^{\frac{1}{q}}q^{\frac{1}{2}}\|\nabla v\|_{L^2(\Omega)},
\end{align*}
where $C_1$ is an absolute constant independent of $\Omega, q$ or $v$.
\end{proof}
	
	\vspace{.1in}

Finally we consider a graphical surface $\Sigma$ in $\mathbb{C}^2$ defined by
	\begin{equation*}
		(u,v,f(u,v),g(u,v)) \ \ \ \ {\rm with} \  (u,v)\in U\subset {\mathbb R}^2,
	\end{equation*}
	for some differentiable functions $f,g$. By direct computations, we see that the K\"ahler angle of $\Sigma$ is given by (\cite{HL3})
	\begin{equation}\label{graph-cos-1}
		\cos\alpha_{J_1}=\frac{1+f_ug_v-f_vg_u}{\sqrt{1+f_u^2+f_v^2+g_u^2+g_v^2+(f_ug_v-f_vg_u)^2}}.
	\end{equation}

 	\vspace{.2in}

\section{The decay estimate for the K\"ahler angle}

    \vspace{.1in}

    \vspace{.1in}
This section is devoted to deriving the decay estimate for the Kähler angle.

\subsection{Local structure of surfaces with small total curvature}
In preparation, we first analyze the geometry of properly immersed surfaces in the annulus $A_{R,2R}$ with small total curvature. We start by classifying the model case: proper immersions of open surfaces into a flat annulus.

\begin{lemma}\label{classification} Assume $f:(M,g)\to (D_2\backslash D_1,g_0)$ is a proper Riemannian immersion from an open surface to a flat annulus and $\mathrm{Area}_g(M)<\infty$. Then,  there exists an integer $1\le d=\frac{\mathrm{Area}(M)}{3\pi}$ such that
\begin{enumerate}
\item $(M,g)$ is conformal to $D_{r}\backslash D_1$ for $r=2^{\frac{1}{d}}$
\item under the conformal parameterization $\varphi: D_r\backslash D_1\to M$, $f\circ \varphi(z)=(e^{i\theta_0}z)^d$ or $f\circ \varphi(z)=(e^{i\theta_0}\bar{z})^d$.
\end{enumerate}
\end{lemma}
\begin{proof}
Since a proper immersion between open manifolds of the same dimension is a covering map, we know $f$ must be a Riemannian covering map. Consequently, $M$ is homeomorphic to either an annulus or a disk. Moreover, in the case $M$ homeomorphic to a disk, the  multiplicity of the covering is infinite, which contradicts $\mathrm{Area}_g(M)<\infty$.  As a result, $M$ is homeomorphic to an annulus. By the conformal classification of topological annuli, we know there exists a conformal parameterization
\begin{align*}
\varphi:A_{r',R}:=D_{R}\backslash D_{r'}\to M, 0\le r'< R\le \infty.
\end{align*}
We next show $r'>0$ and $R<\infty$.  In fact, if $r'=0$ or $R=\infty$, then, we know $A_{r',R}$ is parabolic and is conformally covered by $\mathbb{C}$. Noting that $D_2\backslash D_1$ is hyperbolic, which is conformally covered by $\mathbb{D}$, we can lift the conformal immersion $f:M\to A_{1,2}$ to be a conformal immersion $\tilde{f}: \mathbb{C}\to \mathbb{D}$, and get a contradiction by the Liouville theorem.  As a result, we know $0<r'<R<\infty$. Taking $r=\frac{R}{r'}$, we can assume $A_{r',R}=A_{1,r}$ without loss of generality since they are conformal. So, under the conformal parameterization $\varphi: A_{1,r}\to M$, $f\circ \varphi: A_{1,r}\to A_{1,2}$ is a conformal covering map. Assume $d$ is the multiplicity of the covering map $f\circ \varphi$ and $\pi(w)=w^{d}: A_{1,2^{\frac{1}{d}}}\to A_{1,2}$ is the standard covering map. Then, $f\circ \varphi$ lifts to a conformal isomorphism $\hat{f}: A_{1,r}\to A_{1,2^{\frac{1}{d}}}$.  So, $A_{1,r}$ and $A_{1,2^{\frac{1}{d}}}$ has the same conformal module, which implies $r=2^{\frac{1}{d}}$ and $\hat{f}$ is a rotation. That is, $\hat{f}(z)=e^{i\theta_0}z$ or $\hat{f}(z)=e^{i\theta_0}\bar{z}$. Noting $f\circ \varphi=\pi\circ \hat{f}$, we get the conclusion

$$\text{either }\quad  f\circ \varphi(z)=(e^{i\theta_0}z)^d \quad \text { or} \quad  f\circ \varphi(z)=(e^{i\theta_0}\bar{z})^d.$$
Finally, since $f:M\to A_{1,2}$ is a Riemannian covering with multiplicity $d$, we know
$$3\pi d= d\cdot \mathrm{Area}(D_2\backslash D_1)=\mathrm{Area}(M).$$
\end{proof}

Then, we use a compactness argument to transfer the above rigidity theorem to surfaces in the annulus with volume bounds and  small total curvature.
\begin{proposition}\label{structure of properly immersed small annulus}
 For any $\Lambda>0$, there exists $C, \epsilon_0>0$ such that the following holds.

Assume $\Sigma\subset \mathbb{R}^n$ is an immersed compact surface (with boundary) such that  for some fixed $R>0$, there holds
\begin{align}\label{proper condition}
\partial \Sigma\cap \mathrm{interior}(B_{4R}\backslash  B_{\frac{R}{2}})=\emptyset, \quad  \Sigma\cap\partial B_{\frac{3R}{2}}\neq\emptyset,
\end{align}
and
\begin{align}\label{small total curvature in annulus}
\int_{\Sigma\cap (B_{4R}\backslash B_{\frac{R}{2}})}|{\bf A}|^2d\mu\le \epsilon_0
\quad \text{ and }
 \quad
\mathrm{Area}(\Sigma\cap (B_{4R}\backslash B_{\frac{R}{2}}))\le \Lambda R^2.
\end{align}
Then, there exists an integer $1\le d\le \frac{\Lambda}{3\pi}$ and   a conformal parametrization $f: D_{(2R)^{\frac{1}{d}}}\backslash D_{R^{\frac{1}{d}}}\to \Omega_R$ for some domain  $\Omega_R\subset \Sigma$ with \begin{align}\label{Hausdorff distance estimate}
d_{H}(\left(\Sigma\cap (B_{2R}\backslash B_R)\right)_{x},\Omega_R)\le \psi(\epsilon_0)R
\end{align}
such that  $df\otimes df=e^{2u}g_{\mathbb{R}^2}$ and
\begin{align}\label{annulus estimate}
C^{-1}\le \frac{e^{2u}}{d^2|z|^{2d-2}}\le C
\quad \text{ and } \quad
1-\psi(\epsilon_0)\le \frac{|f(z)|}{|z|^d}\le 1+\psi(\epsilon_0), \forall z\in D_{(2R)^{\frac{1}{d}}}\backslash D_{(\frac{R}{2})^{\frac{1}{d}}}.
\end{align}
Here $\psi$ is a function such that   $\lim_{\epsilon\to 0}\psi( \epsilon)=0$, $x$ is a point on $\Sigma\cap \partial B_{\frac{3R}{2}}$ and $\left(\Sigma\cap (B_{2R}\backslash B_R)\right)_x$ means the connected component of $\Sigma\cap (B_{2R}\backslash B_R)$ containing $x$.
\end{proposition}
\begin{proof}
  Since the proposition is scaling invariant, we only need to consider the case $R=1$. Arguing by contradiction, it suffices to establish \eqref{annulus estimate} for a sequence $\Sigma_i$ satisfying the assumptions of the proposition with $$\int_{\Sigma_i\cap (B_4\backslash B_{\frac{1}{2}})}|{\bf A}_i|^2d\mu_i\le \epsilon_i\to 0.$$
  For any $x_i\in\Sigma_i\cap \partial B_{\frac{3}{2}}$,  consider the connected component $\left(\Sigma_i\cap(B_4\backslash B_{\frac{1}{2}}) \right)_{x_i}$ of $\Sigma_i\cap (B_4\backslash B_{\frac{1}{2}})$ containing $x_i$. By the compactness of surfaces with small total curvature and uniform bound on their areas \cite[Proposition 1.4 and Theorem 1,3]{SZ}, we know there exists a proper immersion
$$F_\infty:(\Sigma_\infty,g_\infty, p_\infty)\to B_{4}\backslash B_{\frac{1}{2}} $$
of a connected surface
such that after composition with an embedding $\varphi_i:\Sigma_\infty\to \Sigma_i$,
the sequence of  immersions  $F_i\circ \varphi_i:\Sigma_\infty\to \left(\Sigma_i\cap(B_{4}\backslash B_{\frac{1}{2}} )\right)_{x_i} \to B_{4}\backslash B_{\frac{1}{2}}$ converges to $F_\infty$ weakly in $W^{2,2}_{\mathrm{loc}}$ and strongly in $W^{1,p}_{\mathrm{loc}}$ for any $p<\infty$ and $F_\infty(p_\infty)=\lim_{i\to \infty} x_i\in \partial B_{\frac{3}{2}}$. Moreover, there exists $C$ such that
\begin{align}\label{bi-Lipschitz}
C^{-1}g_\infty\le \varphi_i^*g_i\le C g_\infty \quad \text{ for } i \quad \text{ large}.
\end{align}
By the weak lower semi-continuity\cite{Morrey}\cite{Langer}\cite[Appendix B]{SZ} and $\epsilon_i\to 0$, we know
\begin{align*}
\int_{\Sigma_\infty}|{\bf A}_\infty|^2d\mu_\infty\le \liminf_{i\to \infty}\int_{\Sigma_i\cap (B_4\backslash B_{\frac{1}{2}})}|{\bf A}_i|^2d\mu_i=0,
\end{align*}
which means $F_\infty(\Sigma_\infty)$ is totally geodesic in $\mathbb{R}^n$. By connectedness and  properness, we can assume  $F_\infty(\Sigma_\infty)=D_{4}\backslash D_{\frac{1}{2}}$. Moreover, the strongly convergence in $W^{1,p}_{\mathrm{loc}}$ implies
$$\mathrm{Area}_{g_\infty}(\Sigma_\infty)\le \Lambda.$$ So, by Lemma \ref{classification}, without loss of generality,  we may assume $\Sigma_\infty=D_{4^{\frac{1}{d}}}\backslash D_{(\frac{1}{2})^{\frac{1}{d}}}$ and $g_\infty=\hat{f}^{*}g_{\mathbb{R}^2}$, where $1\le d\le \frac{\Lambda}{3\pi}$ is an integer and $F_\infty=\hat{f}(z)=z^d$.    The conclusion follows from \eqref{bi-Lipschitz} and the $C_{\mathrm{loc}}^0$ convergence of $F_i\circ \varphi_i$ to $F_\infty$.
\end{proof}
\begin{lemma}\label{lem:topology and geometry of end}
 Let $\Sigma\subset \mathbb{R}^n$ be a  properly  immersed  surface with
\begin{enumerate}
\item ${\rm {Area}}(\Sigma\cap B_R)\le \Lambda R^2$ for all $R>0$,
\item the genus of $\Sigma$ is finite,
\item there exists $R_k\to \infty$ such that $\lim_{k\to \infty}\int_{(A_{\frac{R_k}{2},4R_k})\cap \Sigma}|{\bf A}|^2d\mu=0$. Here $A_{\frac{R_k}{2},4R_k}=B_{4R_k}(0)\backslash B_{\frac{R_k}{2}}(0)$.
\end{enumerate}
Then the number of ends of  $\Sigma$ is no more  than $2\Lambda$ and they are all annular. More precisely, there exists a large number $R_0>0$ such that $\Sigma\backslash B_{R_0}(0)$ contains finite many  non-compact connected   components
$\{\Sigma_j\}_{j=1}^{N\le 2\Lambda}$, which are all homeomorphic  to $\mathbb{D}^c=\{z\in \mathbb{C}||z|\ge 1\}$.
Moreover, for each $j\in \{1,2,\ldots, N\}$ and $k$ large, if $x_k\in \Sigma_j\cap \partial B_{\frac{3R_k}{2}}$, then,  there exist a positive integer $d= d_{j,k}\le \frac{\Lambda}{3\pi}$ and   a conformal parameterization $f: D_{(2R_k)^{\frac{1}{d}}}\backslash D_{R_k^{\frac{1}{d}}}\to \Omega_{R_k}$ for some domain $\Omega_{R_k}\subset \Sigma_j$  containing $x_k$ and satisfying
\begin{align*}
d_{H}(\left(\Sigma_j\cap (B_{2R_k}\backslash B_{R_k})\right)_{x},\Omega_{R_k})\le \psi(\epsilon_k)R_k
\end{align*}
such that  $df\otimes df=e^{2u}g_{\mathbb{R}^2}$ and  for any $ z\in D_{(2R_k)^{\frac{1}{d}}}\backslash D_{(R_k)^{\frac{1}{d}}}$, there holds
\begin{align} \label{conformal factor control}
C^{-1}\le \frac{e^{2u}}{d^2|z|^{2d-2}}\le C
\end{align}
\begin{align}\label{norm}
1-\psi(\epsilon_k)\le \frac{|f(z)|}{|z|^d}\le 1+\psi(\epsilon_k).
\end{align}
\end{lemma}

\begin{proof}
 Let $N(R)$ be the number of non-compact connected  components $\Sigma\backslash B_R(0)$, then $N(R)$ is non-decreasing as $R\to \infty$, since for any two points $x_1,x_2$ belonging to two different non-compact components $\Sigma_1,\Sigma_2$ of $\Sigma\backslash B_{R_1}(0)$ and $R_2>R_1$ and any $y_j\in \Sigma_j\backslash B_{R_2}, j=1,2$, $y_1$ and $y_2$ can not be connected by a curve in $\Sigma\backslash B_{R_2}$.   Below we show $N(R)\le 2\Lambda$ for large $R$. In fact, since $\Sigma$ immersed in $\mathbb{R}^n$ properly,  by Proposition \ref{structure of properly immersed small annulus} and the assumption
\begin{align*}
\epsilon_k:=\int_{(A_{\frac{R_k}{2},4R_k})\cap \Sigma}|{\bf A}|^2d\mu\to 0,
\end{align*}
we can take $k$ large such that each connected component $\Sigma_j$ of $\Sigma\backslash B_{R}$ contains an annulus $A_j\subset B_{(1+\psi(\epsilon_k))2R_k}\backslash B_{(1-\psi(\epsilon_k))R_k}$ with area estimate
\begin{align*}
\mathrm{Area}(A_j)\ge (1-\psi(\epsilon_k))d_{k,i}\cdot  3\pi R_k^2,
\end{align*}
where $d_{k,i}\ge 1$ is an integer.
Thus the area bound assumption $\mathrm{Area}(\Sigma\cap B_{r})\le \Lambda r^2, \forall r$ implies
\begin{align*}
N(R)\le \frac{4\pi \Lambda(1+\psi(\epsilon_k) )^2R_k^2}{(1-\psi(\epsilon_k))3\pi R_k^2 }.
\end{align*}
Letting $k\to \infty$, we get $N(R)\le 2\Lambda$.
 Since $N(R)$ is non-increasing, we know the number of ends of $\Sigma$ is
$$N=\lim_{R\to \infty}N(R)\le 2\Lambda. $$
Since $\Sigma$ is complete and  has finite  genus, we know $\Sigma$ is homeomorphic to a closed surface with $N$ points removed.  By the properness of the immersion,  there exists $R_0\gg 1$ such that $\Sigma\backslash B_{R_0}$ contains exactly  $N$ non-compact connected components  $\{\Sigma_j\}_{j=1}^N$. Each $\Sigma_j$ is homeomorphic to $\mathbb{D}^c$ and $\Sigma_j$ also immersed in $\mathbb{R}^n$ properly as a surface with compact boundary.  Moreover, for each $1\le j\le N$,  large $k$ and $x_k\in \Sigma_j\cap \partial B_{\frac{3R_k}{2}}$, the properness of the immersion again implies $\left(\Sigma_j\cap(A_{\frac{R_k}{2},4R_k})\right)_{x_k}$ satisfies \eqref{proper condition} and \eqref{small total curvature in annulus}  for $R=R_k$, Thus, the conclusion follows  from  Proposition \ref{structure of properly immersed small annulus}.

\end{proof}

In the following, for convenience, we suppress the index $1\le j\le N$ and adopt the notation $E=\Sigma_j$ to denote an end of $\Sigma$.

\begin{definition} Let $X$ be a connected topological space and $A,B, \Gamma\subset X$. We say  We say $\Gamma$ {\bf separates $A$ and $B$} in $X$ if $A, B \subset X\backslash \Gamma$ do not belong to the same connected component of  $X\backslash \Gamma$.  We say $\Gamma$ {\bf separates $A$ and the infinity} if $X$ is non-compact, $A, \Gamma$ are compact and $A\subset X\backslash \Gamma$ does not belongs to any non-compact connected components of $X\backslash \Gamma$.
\end{definition}
 \begin{lemma}\label{construction of annulus} Let $\Sigma\subset \mathbb{R}^{n\ge 3}$ be a properly immersed surface satisfying the same assumptions as in Lemma \ref{lem:topology and geometry of end} and $E$ be an end of $\Sigma$. Then, there exist $C,c>0$ such that  for $k$ large enough and any $\delta_k\in (0,\frac{1}{10})$,  there exist  disjoint annuli $\{\mathcal{A}_k\}_{k=1}^\infty$ in E and sub-annuli $\mathcal{\hat{A}}_k\subset \mathcal{A}_k\subset E\cap A_{R_k,2R_k}$ satisfying
\begin{enumerate}

\item $\partial \mathcal{A}_k=\Gamma_k^{\pm}$ both separate $\partial E$ and the infinity and $\mathcal{\hat{A}}_k$ separates $\Gamma_k^{\pm}$.
\item $d_E(\mathcal{\hat{A}}_k,\partial \mathcal{A}_k)\ge c\delta_kR_k$ and
\begin{align*}
\mathcal{H}^2(\mathcal{A}_k)\le C\delta_kR_k^2.
\end{align*}
\item For any non-contractible closed curve $\gamma$ in $\mathcal{A}_{k}$, there holds
\begin{align*}
\mathrm{Length}(\gamma)\ge c\pi R_k.
\end{align*}
\item As $k\to \infty$,
\begin{align}
   \int_{\mathcal{A}_{k}}|{\bf A}|^2d\mu\le C\delta_k \int_{(A_{\frac{R_k}{2},4R_k})\cap \Sigma}|{\bf A}|^2d\mu\to 0.
\end{align}
\end{enumerate}

\end{lemma}
\begin{proof}
Since $E$ is connected and complete non-compact surface with compact boundary, we can take a curve $\gamma:[0,\infty)\to E$ such that $\gamma(0)\in \partial E$ and $\lim_{t\to \infty}d_g(\gamma(t),\partial E)=\infty$.
For large $k$, by the Sard theorem,  without loss of generality, we can assume $\partial B_{\frac{3R_k}{2}}$ intersects with $E$ transversely and does not touch $\partial E$. As a result, $E\cap \partial B_{\frac{3R_k}{2}}$  consists of finite many circles $\{\Gamma_i^k\}_{i\le I_k}$.

We first show there exists at least one $i$ such that $\Gamma_i^k$ is non-contractible in $E\approx \mathbb{D}^c$.
In fact,  if  all these circles $\Gamma_i^k$ were contractible in $E\approx \mathbb{D}^c$, then
$$[\gamma]\cap [\Gamma_i^k]\equiv0 \quad   \mathrm{mod}\quad   2, \quad \forall i\le I_k,$$
hence $\gamma$ can be perturbed in a compact subset of $\mathring{E}$  to be a curve $\tilde{\gamma}$ which does not intersect with $\cup_i\Gamma_i^k=E\cap \partial B_{\frac{3R_k}{2}}$. So, $\tilde{\gamma}([0,\infty))\subset E$ also satisfies $\tilde{\gamma}(0)\in \partial E$ and $\lim_{t\to \infty}d_g(\tilde{\gamma}(t),\partial E)=\infty$.  Noting that  $E$ is immersed in $\mathbb{R}^n$ properly, we know $\lim_{t\to \infty}|\tilde{\gamma}(t)|=\infty$ and $|\tilde{\gamma}(0)|=R_0<R_k$. So, by continuity, there must exists $\gamma(t_k)\in \partial B_{\frac{3R_k}{2}}$, which contradicts to the fact $\tilde{\gamma}$ does not intersect with $\cup_i\Gamma_i^k$.

Now, choose a non-contractible circle $\Gamma_{i_k}^k\subset E\cap \partial B_{\frac{3R_k}{2}}$ and a point $x_k\in \Gamma_{i_k}^k$.  Then, $\Gamma_{i_k}^k$ seperates $\partial E$ and the infinity.  By Lemma \ref{lem:topology and geometry of end}, we know there exist a
conformal parameterization $f: D_{(2R_k)^{\frac{1}{d}}}\backslash D_{R_k^{\frac{1}{d}}}\to \Omega_{R_k}$ for some domain  $\Omega_{R_k}\subset E$ containing $x_k$ such that   \eqref{conformal factor control} and \eqref{norm} hold. Especially, for any $R_k^{\frac{1}{d}}<r\le s< (2R_k)^{\frac{1}{d}}$, $f(D_s\backslash D_r)$ is an annulus(or a circle when $r=s$) which retracts to to a non-contractible  curve isotropic to $\Gamma_{i_k}^k$ in $E$. Thus $f(D_s\backslash D_r)$ separates $\partial E$ and  the infinity.

For $\delta_k\le \frac{1}{10}$ , noting
\begin{align*}
\bigcup_{l=1}^{[\frac{3}{4\delta_k}]} \left(D_{\left(\frac{5R_k}{4}+l\delta_kR_k\right)^{\frac{1}{d}}}\backslash D_{ \left(\frac{5R_k}{4}+(l-1)\delta_kR_k\right)^{\frac{1}{d}}}\right)\subset D_{\left(\frac{7R_k}{4}\right)^{\frac{1}{d}}}\backslash D_{\left(\frac{5R_k}{4}\right)^{\frac{1}{d}}}\subset D_{(2R_k)^{\frac{1}{d}}}\backslash D_{R_k^{\frac{1}{d}}},
\end{align*}
and denoting $\mathcal{A}_{k,l}=f\left(D_{\left(\frac{5R_k}{4}+l\delta_kR_k\right)^{\frac{1}{d}}}\backslash D_{ \left(\frac{5R_k}{4}+(l-1)\delta_kR_k\right)^{\frac{1}{d}}}\right)\subset E\cap A_{R_k, 2R_k}$, then $\mathcal{A}_{k,l}$ all separates $\partial E$ and the infinity and
\begin{align*}
\sum_{l=1}^{[\frac{3}{4\delta_k}]}\int_{\mathcal{A}_{k,l}}|{\bf A}|^2d\mu\le \int_{E\cap A_{R_k,2R_k}}|{\bf A}|^2d\mu=\epsilon_k\to 0.
\end{align*}
So,  there exist a positive integer  $l_k\le[\frac{3}{4\delta_k}]$ such that
\begin{align*}
\int_{\mathcal{A}_{k,l_k}}|{\bf A}|^2d\mu\le \frac{8\delta_k\epsilon_k}{3}\to 0.
\end{align*}
Define $r_k=\frac{5R_k}{4}+(l_k-1)\delta_kR_k$, $s_k=\frac{5R_k}{4}+l_k\delta_kR_k$,   $\mathcal{A}_k=\mathcal{A}_{k,l_k}=f(D_{(s_k)^{\frac{1}{d}}}\backslash D_{(r_k)^{\frac{1}{d}}})$ and $\hat{\mathcal{A}}_k=f(D_{(\frac{r_k+3s_k}{4})^{\frac{1}{d}}}\backslash D_{(\frac{3r_k+s_k}{4})^{\frac{1}{d}}})$. Then,
$\mathcal{A}_k$ and $\mathcal{\hat{A}}_k$ are both annuli such that  $\partial \mathcal{A}_k=\Gamma_{k}^{-}\cup \Gamma_{k}^+$ are two circles separating  $\partial E$
 and the infinity of $E$ and   $\mathcal{\hat{A}}_k\subset \mathcal{A}_k$ separates $\Gamma_k^-$ and $\Gamma_k^+$ in $\mathcal{A}_k$.
So, item $(1),(4)$ of the Lemma hold. By taking a subsequence if necessary, $\{\mathcal{A}_{k}\}$ are disjoint.
  Moreover,  by \eqref{conformal factor control}, we know
 \begin{align*}
 \mathcal{H}^2(\mathcal{A}_k)\le \int_{D_{(s_k)^{\frac{1}{d}}}\backslash D_{(r_k)^{\frac{1}{d}}}} e^{2u}&\le 2\pi C \int_{(r_k)^{\frac{1}{d}}}^{(s_k)^{\frac{1}{d}}}d^2\cdot|z|^{2(d-1)} rdr\\
 &=\pi Cd(s_k^2-r_k^2)\\
 &\le \pi C \cdot \frac{\Lambda}{3\pi}(\frac{7R_k}{4}+\frac{7R_k}{4})\cdot \delta_k R_k=C \delta_k R_k^2.
 \end{align*}
For any $x\in \partial \mathcal{A}_k$ and $y\in \partial \mathcal{\hat{A}}_k$, denote $x=f(z_1)$ and $y=f(z_2)$. Then, $$\left||z_1^d|-|z_2^d|\right|\ge \frac{s_k-r_k}{4}=\frac{\delta_k R_k}{4}$$ and for any curve $\sigma$ in $E$ joining $x$ and $y$, we can take  $\gamma_1=f^{-1}(\sigma)$ joining $z_1$ and $z_2$ and calculate
\begin{align*}
\mathrm{Length}(\sigma)&=\int_{0}^1 e^{u(\gamma_1(t))}|\gamma_1'(t)|dt\\
&\ge \int_0^1 C^{-1}d\cdot|\gamma_1(t)|^{d-1}|\gamma_1'(t)|dt\\
&\ge C^{-1}\left|\int_0^1\frac{d}{dt}\gamma_1^d(t)dt\right|\\
&= C^{-1}|\gamma_1^d(1)-\gamma_1^d(0)|\\
&\ge C^{-1}\left||z_1^d|-|z_2^d\right|\ge \frac{\delta_k R_k}{4C}=:c\delta_kR_k.
\end{align*}
As a result, we know $d_E(\hat{A}_k, \partial \mathcal{A}_k)\ge c\delta_k R_k$. Similarly, for any closed curve $\gamma$ in $\mathcal{A}_k$ which is non-contractible and $\gamma_2=f^{-1}(\gamma)$, we know
\begin{align*}
\mathrm{Length}(\gamma)&\ge C^{-1}\int_0^1 d\cdot |\gamma_2(t)|^{d-1}|\gamma_2'(t)|dt\\
&\ge C^{-1}\left(\frac{5R_k}{4}\right)^{\frac{d-1}{d}}\int_0^1|\gamma_2'(t)|dt\\
&\ge C^{-1}\left(\frac{5R_k}{4}\right)^{\frac{d-1}{d}} 2\pi \left(\frac{5R_k}{4}\right)^{\frac{1}{d}}\\
&=c\pi R_k.
\end{align*}
 \end{proof}

\vspace{.1in}

\subsection{Decay estimate of $\cos\alpha_J$}
From now on, in addition to the assumptions from the previous subsection, we assume that $\Sigma$ is a translating soliton satisfying
  \begin{align*}
    \lim_{x\to \infty} \int_{B_1(x)\cap \Sigma}|{\bf  A}|^2d\mu=0.
    \end{align*}
 The second preliminary step of this section is to establish that the second fundamental form vanishes asymptotically.

\begin{proposition}\label{p-vanishing of second fundamental form}
    Assume $\Sigma\subset\mathbb{R}^n$ is a tranlating soliton such that
    \begin{align*}
    \lim_{x\to \infty} \int_{B_1(x)\cap \Sigma}|{\bf  A}|^2d\mu=0.
    \end{align*}
    Then   $\lim_{x\to \infty}|{\bf A}|(x)=0$.
\end{proposition}

\begin{proof}
By the assumption, we know for  any $\varepsilon>0$,  there exists $R_0>10$ such that for  any $x\in \Sigma\backslash B_{2R_0}(0)$, there holds $\int_{B_{1}(x)}|{\bf A}|^2d\mu\le \varepsilon$. Denote $\Sigma^1_x$ by the connected component of $\Sigma\cap B_1(x)$ passing through $x$. Fix $\delta>0$, by choosing $\varepsilon$ small enough, we know \cite[Corollary 5.2]{SZ}
\begin{align*}
\sup_{s\le 1}\frac{\mathcal{H}^2(\Sigma_x^1\cap B_s(x))}{\pi s^2}\le 1+\delta.
\end{align*}
Moreover, noting the translating soliton equation ${\bf H=T^{\bot}}$ implies $|{\bf H}|\le 1$, by the Allard regularity theorem\cite{Allard}, we know   an affine plane $T_{x}$ passing through $x$ such that $\Sigma_x^1\cap B(x,1)$ can be written as a graph $\{(z,u(z))|z\in D(x,1):=T_{x,1}\cap B(x,1)\}$ such that the
\begin{align*}
\|u\|_{C^{1,\alpha}(D(x,1))}\le \psi(\delta),
\end{align*}
where $u(x)=0$ for $x\in T_{x,1}\subset \mathbb{R}^n$ and $\lim_{\delta\to 0}\psi(\delta)=0$.  For simplicity, we denote $D_s:=D(x,s)$. Now, by the mean curvature equation
we get
\begin{align*}
\|u\|_{W^{2,p}(D_1)}\le C_p\left(\|u\|_{L^p(D_1)}+\|{\bf H}\|_{L^p(D_1)}\right).
\end{align*}
Noting that $|{\bf H}|\le 1$, we get (for any $p\ge 2$)
\begin{align*}
\|{\bf A}\|_{L^p(\Sigma_x^1\cap B_1(x))}&\le C_p\|D^2u\|_{L^p(D_1)}\le C_p\left(\psi(\delta)+\|{\bf H}\|_{L^2(\Sigma_x^1)}\right)\le C_p\psi(\delta).
\end{align*}
Since $\Sigma$ is a translating soliton, using the equation of the second fundamental form
\begin{align*}
\Delta |{\bf A}|^2+\langle \nabla x_1, \nabla |{\bf A}|^2\rangle\ge 2|\nabla {\bf H}|^2-3|{\bf A}|^4\ge -3|{\bf A}|^4,
\end{align*}
By Moser's iteration we get
\begin{align*}
\sup_{\Sigma_x^1\cap B_{\frac{1}{2}}(x)}|{\bf A}|^2\le C_p\|{\bf A}\|_{L^{4p}(\Sigma_x^1\cap B_1(x))}^4\le C_p\psi(\delta)
\end{align*}
for any $p>1$. Taking $p=2$, we get
\begin{align*}
|{\bf A}|(x)\le C\psi(\delta), \forall x\in \Sigma\backslash B_{R_0(\varepsilon)},
\end{align*}
where $\varepsilon=\varepsilon(\delta)$. Thus
\begin{align*}
\limsup_{x\to \infty}|{\bf A}|(x)\le \psi(\delta).
\end{align*}
Since $\delta$ is arbitrary and $\lim_{\delta\to0}\psi(\delta)=0$, we get the conclusion.
\end{proof}

\vspace{.1in}

\begin{remark}
There is an alternative parabolic approach which follows Khan's  argument \cite{Khan} by using Echer's $\varepsilon$-regularity theorem.  
\end{remark}

\begin{proof}[Alternative proof of Proposition \ref{p-vanishing of second fundamental form}]

For $\varepsilon_0>0$ given, there exists $R_0>2$ such that if $x_0\in \Sigma\backslash B_{R_0}(0)$, then  $\int_{B_1(x_0)}|{\bf A}|^2d\mu\le \varepsilon_0$. We take
$
\rho=\frac{1}{4}.
$
and consider the mean curvature flow solution \(\Sigma_t = \Sigma + t{\bf T}\) defined on the interval
\(t \in [-3\rho^2/4, \rho^2/4]\).
Notice that  $|t|\le \frac{3\rho^2}{4}$ and $|{\bf T}|=1$ implies
\begin{align*}
B_{\rho}(x_0)+t{\bf T}\subset B_{\rho+\frac{3\rho^2}{4}}(x_0).
\end{align*}
By the definition of $\rho$, we know $\rho+\frac{3\rho^2}{4}\le \frac{1}{2}$, hence
\begin{align*}
\int_{\Sigma_t\cap B_\rho(x_0)}|{\bf A}|^2d\mu=\int_{\Sigma\cap (B_\rho(x_0)+t{\bf T})}|{\bf A}|^2d\mu\le \int_{B_1(x_0)}|{\bf A}|^2d\mu\le \varepsilon_0.
\end{align*}

By Echer's  $\varepsilon$-regularity theorem(\cite[Theorem 2,1]{Ecker}) (Although Theorem 2.1 of \cite{Ecker} was proved for two-dimensional mean curvature flow in ${\mathbb R}^3$, it is easy to see that the same conclusion holds when the ambient manifold is ${\mathbb R}^n$ (see Theorem 3.1 of \cite{ChenWen}).)
\[
\sup_{[0, \rho^2/4]}
\sup_{\Sigma_t \cap B_{\rho/2}(x_0)} |{\bf A}|^2
\leq
c_0 \rho^{-4}
\int_{-3(\rho/2)^2}^{(\rho/2)^2}
\int_{\Sigma_t \cap B_\rho(x_0)} |{\bf A}|^2
\leq 16
c_0  \int_{B_1(x_0)}|{\bf A}|^2d\mu .
\]
Thus, $|{\bf A}|^2(x_0) \leq 16 c_0 \int_{B_1(x_0)}|{\bf A}|^2d\mu.$
\end{proof}

\vspace{.1in}

The above estimate shows ${\bf T}^{\bot}(x)={\bf H}(x)\to 0$ as $x\to \infty$, which means  $\bf T$ is close to a unit tangent  vector $e_1$ in $T_x\Sigma$ as $x\to \infty$. Thus to determine $T_x\Sigma$, we need to estimate $\cos\alpha_J=\langle Je_1,e_2\rangle$. In particular, if $\cos\alpha_J(x)$ is close to $1$, then $Je_1$ is close to the unit tangent vector $e_2\perp e_1$ in $T_x\Sigma$, which further implies $T_x\Sigma$ is close to the fixed plane ${\bf T}\wedge J{\bf T}$. To this end, we next derive an estimate for the function $\cos\alpha_{J}-s$ on a domain $\Omega\subset \mathcal{A}_k\subset E\cap A_{R_k,2R_k}$,  where $\partial \Omega$ is a closed contractible curve and $\cos\alpha_{J}\equiv s$ on $\partial \Omega$. Here $J$ denotes a fixed compatible complex structure on ${\mathbb R}^4$.

\begin{proposition}\label{prop-osi}
    Let $\Omega\subset \mathcal{A}_k$ be a domain described as above for $k$ sufficiently large, and $|\Omega|$ denote the area of $\Omega$. Suppose furthermore that $\cos\alpha_{J}\geq 0$. Then there is a constant $\hat{C}$ and $\xi_0\in (\frac{1}{2},1)$ such that, if $|\Omega|$ is sufficiently large, then we have
    \begin{align}\label{e-cont}
        \max_{\Omega}(\cos\alpha_{J}-s)\leq \hat C\left(\log(|\Omega|+2)\int_{\Omega}|\overline{\nabla}J_{\Sigma}|^2\right)^{\frac{\xi_0}{2}}.
    \end{align}
\end{proposition}

\begin{proof}
    Set $u=\cos\alpha_{J}-s$. Then by (\ref{e-cosalpha}), we have
    		\begin{equation}\label{e-cosalpha-J1}
			\Delta u+\langle\nabla x_1,\nabla u\rangle=-|\overline{\nabla}J_{\Sigma}|^2\cos\alpha_{J}.
		\end{equation}
    Since $\cos\alpha_{J}\geq 0$ and $u=0$ on $\partial\Omega$, the maximum principle implies that $u>0$ in $\Omega$. Denote $f(x)=|\overline{\nabla}J_{\Sigma}|^2\cos\alpha_{J}$. For $p\geq 1$, multiplying both sides of (\ref{e-cosalpha-J1}) by $u^p$ and integrating by parts yields
\begin{align*}
    \int_{\Omega}fu^p
    =&-\int_{\Omega}u^p\Delta u-\int_{\Omega}u^p\langle\nabla x_1,\nabla u\rangle\nonumber\\
    =&p\int_{\Omega}u^{p-1}|\nabla u|^2+\frac{1}{p+1}\int_{\Omega}u^{p+1}\Delta x_1\\
    =& \frac{4p}{(p+1)^2}\int_{\Omega}|\nabla u^{\frac{p+1}{2}}|^2+\frac{1}{p+1}\int_{\Omega}u^{p+1}|{\bf H}|^2\\
    \geq &\frac{4p}{(p+1)^2}\int_{\Omega}|\nabla u^{\frac{p+1}{2}}|^2,
\end{align*}
where we used (\ref{e-x1}) in the third equality. Since $\frac{p+1}{p}\leq 2$ for $p\geq 1$, we derive that
\begin{align*}
    \int_{\Omega}|\nabla u^{\frac{p+1}{2}}|^2
    \leq \frac{(p+1)^2}{4p}\int_{\Omega}fu^p\leq \frac{p+1}{2}\int_{\Omega}fu^p.
\end{align*}
Since  $\lim_{k\to \infty} \int_{\Sigma\cap A_{\frac{R_k}{2}, 4R_k}}|{\bf A}|^2d\mu=0$, for $k$ large,  by item $(1)$ of Proposition \ref{prop-iso-proper} and Proposition \ref{prop-sob},  we can apply the Sobolev inequality (\ref{e-soblev-3}) to obtain for any $q\geq 2$ that
\begin{align*}
    \left(\int_{\Omega}u^{\frac{(p+1)q}{2}}\right)^{\frac{2}{q}}
    \leq Cq|\Omega|^{\frac{2}{q}}\int_{\Omega}|\nabla u^{\frac{p+1}{2}}|^2
    \leq C(p+1)q|\Omega|^{\frac{2}{q}}\int_{\Omega}fu^p,
\end{align*}
where $C$ is a constant depending only on $D_1$ which may vary from line to line. By H\"older inequality, we have
\begin{align*}
    ||u||_{L^{\frac{(p+1)q}{2}}}
    =&\left(\int_{\Omega}u^{\frac{(p+1)q}{2}}\right)^{\frac{2}{(p+1)q}}\\
    \leq &\left(Cq|\Omega|^{\frac{2}{q}}\right)^{\frac{1}{p+1}}(p+1)^{\frac{1}{p+1}}\left(\int_{\Omega}fu^p\right)^{\frac{1}{p+1}}\\
    \leq & \left(Cq|\Omega|^{\frac{2}{q}}\right)^{\frac{1}{p+1}}(p+1)^{\frac{1}{p+1}}||f||_{L^{p+1}}^{\frac{1}{p+1}}||u||_{L^{p+1}}^{\frac{p}{p+1}}\\
    \leq & \left(Cq|\Omega|^{\frac{2}{q}}||f||_{L^{\infty}(\Omega)}\right)^{\frac{1}{p+1}}|\Omega|^{\frac{1}{(p+1)^2}}(p+1)^{\frac{1}{p+1}}||u||_{L^{p+1}}^{\frac{p}{p+1}}.
\end{align*}
Fix $\chi=\frac{q}{2}>1$ and set $p_l=p+1$, $p_{l+1}=\frac{(p+1)q}{2}$ in the above inequality. Then we have $p_l=p_0\chi^l$ and
\begin{align*}
    &||u||_{L^{p_{l+1}}}\\
    \leq & \left(Cp_0q|\Omega|^{\frac{2}{q}}||f||_{L^{\infty}(\Omega)}\right)^{p_0^{-1}\chi^{-l}}|\Omega|^{p_0^{-2}\chi^{-2l}}\chi^{lp_0^{-1}\chi^{-l}}||u||_{L^{p_l}}^{\frac{p_l-1}{p_l}}\\
    \leq & \left(Cp_0q|\Omega|^{\frac{2}{q}}||f||_{L^{\infty}(\Omega)}\right)^{p_0^{-1}\chi^{-l}}|\Omega|^{p_0^{-2}\chi^{-2l}}\chi^{kp_0^{-1}\chi^{-l}}\\
    & \cdot \left\{\left(Cp_0q|\Omega|^{\frac{2}{q}}||f||_{L^{\infty}(\Omega)}\right)^{p_0^{-1}\chi^{-l+1}}|\Omega|^{p_0^{-2}\chi^{-2(l-1)}}\chi^{(l-1)p_0^{-1}\chi^{-l+1}}||u||_{L^{p_{l-1}}}^{\frac{p_{l-1}-1}{p_{l-1}}}\right\}^{\frac{p_l-1}{p_l}}\\
    =&\left(Cq|\Omega|^{\frac{2}{q}}||f||_{L^{\infty}(\Omega)}\right)^{p_0^{-1}\left(\chi^{-l}+\chi^{-(l-1)}\left(1-\frac{1}{p_l}\right)\right)}|\Omega|^{p_0^{-2}\left(\chi^{-2l}+\chi^{-2(l-1)}\left(1-\frac{1}{p_l}\right)\right)}\\
    & \cdot \chi^{p_0^{-1}\left(l\chi^{-l}+(l-1)\chi^{-(l-1)}\left(1-\frac{1}{p_l}\right)\right)}||u||_{L^{p_{l-1}}}^{\left(1-\frac{1}{p_{l}}\right)\left(1-\frac{1}{p_{l-1}}\right)}\\
    \leq & \cdots\\
    \leq & \left(Cq|\Omega|^{\frac{2}{q}}||f||_{L^{\infty}(\Omega)}\right)^{p_0^{-1}\beta_l}|\Omega|^{p_0^{-2}\gamma_l}\chi^{p_0^{-1}\alpha_l}||u||_{L^{p_{0}}}^{\xi_l},
\end{align*}
where
\begin{align*}
    \beta_l=&\chi^{-l}+\chi^{-(l-1)}\left(1-\frac{1}{p_l}\right)+\cdots+\chi^{-1}\left(1-\frac{1}{p_{2}}\right)\cdots\left(1-\frac{1}{p_{l-1}}\right)\left(1-\frac{1}{p_l}\right)\\
    =&\sum_{i=1}^l\chi^{-i}\left(1-\frac{1}{p_{i+1}}\right)\cdots\left(1-\frac{1}{p_l}\right)\\
    \leq& \sum_{i=1}^l\chi^{-i}\leq \frac{\chi^{-1}}{1-\chi^{-1}}=\frac{1}{\chi-1},
\end{align*}
\begin{align*}
    \gamma_l=&\chi^{-2l}+\chi^{-2(l-1)}\left(1-\frac{1}{p_l}\right)+\cdots+\chi^{-2}\left(1-\frac{1}{p_{2}}\right)\cdots\left(1-\frac{1}{p_{l-1}}\right)\left(1-\frac{1}{p_l}\right)\\
    =&\sum_{i=1}^l\chi^{-2i}\left(1-\frac{1}{p_{i+1}}\right)\cdots\left(1-\frac{1}{p_l}\right)\\
    \leq& \sum_{i=1}^l\chi^{-2i}\leq \frac{\chi^{-2}}{1-\chi^{-2}}=\frac{1}{\chi^2-1},
\end{align*}
\begin{align*}
    \alpha_l=&l\chi^{-l}+(l-1)\chi^{-(l-1)}\left(1-\frac{1}{p_l}\right)+\cdots +\chi^{-1}\left(1-\frac{1}{p_{2}}\right)\cdots\left(1-\frac{1}{p_{l-1}}\right)\left(1-\frac{1}{p_l}\right)\\
    =&\sum_{i=1}^li\chi^{-i}\left(1-\frac{1}{p_{i+1}}\right)\cdots\left(1-\frac{1}{p_l}\right)\\
    \leq& \sum_{i=1}^li\chi^{-i}\leq\sum_{i=1}^{\infty}i\chi^{-i}=\frac{\chi^{-1}}{(1-\chi^{-1})^2}=\frac{\chi}{(\chi-1)^2},
\end{align*}
and
\begin{align*}
    \xi_l=&\left(1-\frac{1}{p_{0}}\right)\cdots\left(1-\frac{1}{p_{l-1}}\right)\left(1-\frac{1}{p_l}\right)
    =\prod_{i=0}^l\left(1-\frac{1}{p_{i}}\right)\to\xi_0
\end{align*}
for some $\xi_0\in (0,1)$ when $l\to\infty$.

By Proposition \ref{p-vanishing of second fundamental form} and $\Omega\subset A_{R_k, 2R_k}$, we know $\|f\|_{L^\infty(\Omega)}\le 1$ for $k$ large.
Therefore,
\begin{align*}
    &||u||_{L^{p_{l+1}}}
    \leq  \left(Cq|\Omega|^{\frac{2}{q}}\right)^{p_0^{-1}\beta_l}|\Omega|^{p_0^{-2}\gamma_l}\chi^{p_0^{-1}\alpha_l}||u||_{L^{p_{0}}}^{\xi_l},
\end{align*}
The Sobolev inequality (\ref{e-soblev-3}) implies that
    \begin{equation*}
    ||u||_{L^{p_{0}}}\leq Cp_0^{\frac{1}{2}}|\Omega|^{\frac{1}{p_0}}\left(\int_{\Omega}|\nabla u|^2\right)^{\frac{1}{2}}.
\end{equation*}
By \eqref{e-J-nablaalpha}, we know $|\nabla u|\le |\overline{\nabla}J_{\Sigma}|$. So,
\begin{align}\label{after iteration}
    &||u||_{L^{p_{l+1}}}
    \leq  \left(Cq|\Omega|^{\frac{2}{q}}\right)^{p_0^{-1}\beta_l}|\Omega|^{p_0^{-2}\gamma_l}\chi^{p_0^{-1}\alpha_l}\cdot\left(  Cp_0^{\frac{1}{2}}|\Omega|^{\frac{1}{p_0}}||\overline{\nabla} J_{\Sigma}||_{L^{2}}\right)^{\xi_l}
\end{align}

If $|\Omega|\geq 1$, letting $l\to \infty$, then
\begin{align*}
    ||u||_{L^{\infty}(\Omega)}&\leq \left(Cq|\Omega|^{\frac{2}{q}}\right)^{\frac{1}{p_0(\chi-1)}}|\Omega|^{\frac{1}{p_0^2(\chi^2-1)}}\chi^{\frac{\chi}{p_0(\chi-1)^2}}\cdot\left(  Cp_0^{\frac{1}{2}}|\Omega|^{\frac{1}{p_0}}||\overline{\nabla} J_{\Sigma}||_{L^{2}}\right)^{\xi_0}\\
    &=\left(Cq|\Omega|^{\frac{2}{q}}\right)^{\frac{1}{p_0(\frac{q}{2}-1)}}|\Omega|^{\frac{1}{p_0^2((\frac{q}{2})^2-1)}}(\frac{q}{2})^{\frac{\frac{q}{2}}{p_0(\frac{q}{2}-1)^2}} \cdot C^{\xi_0}|\Omega|^{\frac{\xi_0}{p_0}}\left(p_0\int_{\Omega}|\overline{\nabla}J_{\Sigma}|^2\right)^{\frac{\xi_0}{2}}
\end{align*}
Now suppose that $|\Omega|$ is sufficiently large, and we choose $\frac{q}{2}=p_0=\log |\Omega|$ to obtain
\begin{align*}
    ||u||_{L^{\infty}(\Omega)}
    \leq &\left(2Ce\log |\Omega|\cdot \right)^{\frac{1}{\log |\Omega|(\log |\Omega|-1)}}e^{\frac{1}{\log |\Omega|((\log |\Omega|)^2-1)}}\\
    & \cdot(\log |\Omega|)^{\frac{1}{(\log |\Omega|-1)^2}} C^{\xi_0}e^{\xi_0}\left(\log |\Omega|\int_{\Omega}|\overline{\nabla}J_{\Sigma}|^2\right)^{\frac{\xi_0}{2}}.
\end{align*}
Notice that in the present choice of $q$ and $p_0$, we have $p_k=p_0\chi^k=(\log |\Omega|)^{k+1}$ so that
\begin{align*}
    \xi_0=\prod_{i=0}^{\infty}\left(1-\frac{1}{p_{i}}\right)=\prod_{l=1}^{\infty}\left(1-\frac{1}{(\log |\Omega|)^l}\right).
\end{align*}
It is easy to see that $\xi_0\in (\frac{1}{2},1)$ when $\log |\Omega|$ is sufficiently large. Noting that the function $(Ct)^{\frac{1}{t(t-1)}}e^\frac{1}{t(t^2-1)}\cdot t^{\frac{1}{(t-1)^2}}$ is bounded on $[2,\infty]$, we know there exists $M_0\ge 1$ such that when $|\Omega|\ge M_0$, there holds
\begin{align*}
\|u\|_{L^\infty(\Omega)}\le \hat{C}\left(\log|\Omega|\int_{\Omega}|\overline{\nabla}J_{\Sigma}|^2\right)^{\frac{\xi_0}{2}}
\end{align*}
for some $\xi_0\in(\frac{1}{2},1)$.

Finally, if $|\Omega|\le M_0$,  then we can take $q=4$ and $p_0$ large such that $\xi_0\in (\frac{1}{2},1)$.  Letting $l\to \infty$  in \eqref{after iteration}, we also  get

$$\|u\|_{L^\infty(\Omega)}\le C(M_0,p_0)\|\overline{\nabla} J_\Sigma\|_{L^{2}}^{\xi_0}\le \hat{C} \left(\log(|\Omega|+2)\int_{\Omega}|\overline{\nabla}J_{\Sigma}|^2\right)^{\frac{\xi_0}{2}}.
$$
\end{proof}

\vspace{.1in}

    Similarly to Lemma 4.4 in (\cite{NT}), we have the following lemma which we will use later.

\begin{lemma}\label{lem-f-2}
For each fixed $\mu>0$ and $\beta<\frac{2\mu}{\sqrt{1+16\mu^2}}$, there is a positive constant $R_1=R_1(\mu,\beta)$, such that for each $x^0\in {\mathbb R}^4$, the function
\begin{equation*}
			f_{2}(x)=|x-x^0|^{-\beta}\exp\left(-\mu |x-x^0|\right)
\end{equation*}
satisfies
		\begin{equation*}
			\Delta f_{2}\leq f_{2}\left(\mu^2+\frac{|{\bf H}|^2}{4}\right) \ \ {\rm for \ all} \ |x-x_0|\geq R_1.
		\end{equation*}
\end{lemma}

\vspace{.1in}

\vspace{.1in}

The principal theorem in this section is stated as follows:

	\begin{theorem}\label{prop-alpha-decay}
	Let $\Sigma$ be a complete properly immersed  translating soliton in ${\mathbb R}^4$ satisfying
\begin{enumerate}
\item ${\rm {Area}}(\Sigma\cap B_R)\le \Lambda R^2$ for all $R>0$,
\item the genus of $\Sigma$ is finite,
\item there exists $R_k\to \infty$ such that $\lim_{k\to \infty}\int_{(A_{\frac{R_k}{2},4R_k})\cap \Sigma}|{\bf A}|^2d\mu=0$,
\item $\lim_{|x|\to \infty} \int_{B_1(x)\cap \Sigma}|{\bf A}|^2d\mu\to 0.$
\end{enumerate}
Then for each $\beta<\frac{1}{3}$, there are constants $B$ and $R_2$ such that, on each
end $\Sigma_j$ of $\Sigma$, $1\le j\le N$, there is a compatible complex structure $I_{j}\in {\mathcal J}$ such that
		\begin{equation}\label{e-alpha-I-j}
			\sin^2\frac{\alpha_{I_j}}{2}\leq B|x|^{-\beta}\exp\left(-\frac{|x|}{2}-\frac{x_1}{2}\right) \ \ {\rm for \ all} \ x\in \Sigma_j \ {\rm with} \ |x|\geq R_2.
		\end{equation}
	\end{theorem}

\vspace{.1in}

In the next, we fix $j$ and denote  $E=\Sigma_j$.
The following lemma describes the nodal set of $\cos\alpha_J$.
\begin{lemma}\label{lem-structure of nodal set} Assume $\Sigma\subset\mathbb{R}^4$ is a properly immersed translating soliton satisfying the same assumptions of Theorem \ref{prop-alpha-decay}.  Denote $N_J=\{x\in E| \cos\alpha_J=0\}$ as the nodal set of $\cos\alpha_J$ on $E$. Assume $N_J\neq E$. Then,  there exists $R_J>0$ such that either $N_J\cap (E\backslash B_{R_J})=\emptyset$ or  $N_J\cap (E\backslash B_{R_J})$   consists  smooth curves properly embedded in $E\backslash B_{R_J}$.
 \end{lemma}
 \begin{proof}
 If $\cos\alpha_{J}$ does not change sign on $E\backslash K$ for some compact set $K$, there is nothing to prove.
So, we assume $\cos\alpha_{J}$ changes sign on $E\backslash K$ for any $K\subset\subset E$, which implies $N_J$ is non-compact. By Proposition \ref{p-vanishing of second fundamental form}, we can choose a compact set $K\subset E$ such that $\|{\bf A}\|_{L^\infty(E\backslash K)}\le \delta_0\ll 1$. By Lemma \ref{lem:topology and geometry of end},  we can assume $E\backslash K$ is homeomorphic to $\mathbb{D}^c$ without loss of generality.

Since $\cos\alpha_J$ satisfies the elliptic equation \eqref{e-cosalpha}, by the strong unique continuation theorem and  structure theorem for nodal sets of solutions of elliptic equations (\cite{Ber}, see also \cite[Theorem 6.1, Theorem 6.12 and Lemma 6.13]{CM} for detailed discussion),  we know $N_J\cap{(E\backslash K)}$ is a graph $G$ consists of smooth curves with boundary (may be empty) on $\partial (E\backslash K)$. The vertices of the graph are all critical points of $\cos\alpha_J$.  Moreover,  as a metric space, $G$ is  properly  embedded in $E$. Especially, each compact subset intersects at most finite many curves in $G$.  See Figure \ref{fig:nodal set}.
\begin{figure}[htbp]
\centering
\includegraphics[width=0.7\textwidth]{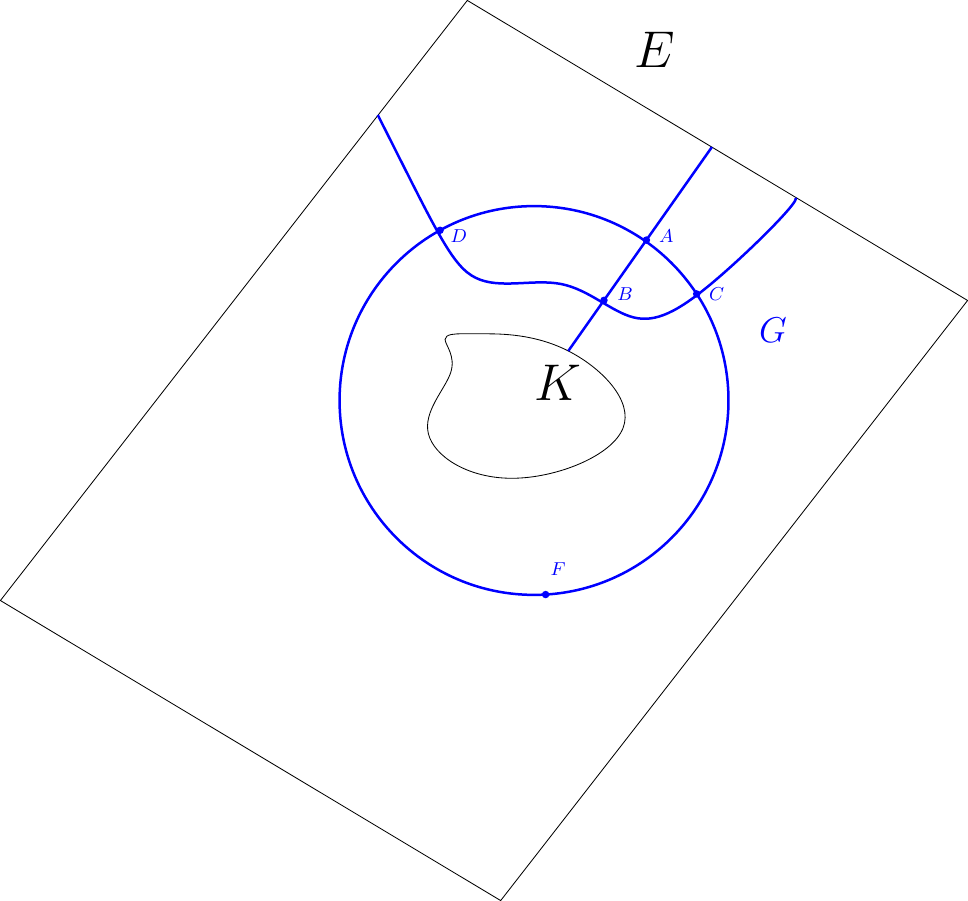}
\caption{Nodal set\quad \quad
$\widetilde{ABCA}$ is a  contractible circle in $E\backslash K$ \quad \quad \quad
$\widetilde{ACFDA}$ is a non-contractible circle in $E\backslash K$.}
\label{fig:nodal set}
\end{figure}

Let us discuss the number of  circles in the graph $G$.

\textbf{Claim 1:} {\it There is no circle in $G$ which is contractible in  $E\backslash K$.}

\vspace{.1in}
We check this by contradiction argument. Otherwise, there exists a circle $\Gamma$ which is contractible in $E\backslash K$. Since $G$ is properly embedded in $E\backslash K$, by choosing a smaller circle if necessary, we can assume $\Gamma$ bound a domain $\Omega_\Gamma\subset E\backslash K$ such that $\cos\alpha_J$ does not change sign in $\Omega_\Gamma$. Applying the strong maximum principle to the equation (\ref{e-cosalpha}), we see that $\Omega_\Gamma\cap N_J=\emptyset$.
Notice that by (\ref{e-x1}) and \eqref{e-cosalpha}, we have
\begin{equation}\label{e-e-x1}
\Delta e^{\frac{x_1}{2}}=e^{\frac{x_1}{2}}\left(\frac{1}{2}|{\bf H}|^2+\frac{|e_{1}^T|^2}{4}\right)=e^{\frac{x_1}{2}}\frac{1+|{\bf H}|^2}{4},
\end{equation}
and hence
\begin{eqnarray}\label{e-test}
  \Delta \left(e^{\frac{x_1}{2}}\cos\alpha_{J}\right)
  =e^{\frac{x_1}{2}}\left(-|\overline{\nabla}J_{\Sigma}|^2\cos\alpha_{J}+\frac{|{\bf H}|^2+1}{4}\cos\alpha_{J}\right).
\end{eqnarray}
By Proposition \ref{p-vanishing of second fundamental form}, we know  $\|{\bf A}\|_{L^\infty(E\backslash K)}\le \delta_0\ll 1$, hence  $\frac{|\bf {\bf H}|^2+1}{4}-|\overline{\nabla}J_\Sigma|^2\ge \frac{1}{8}>0$ on $E\backslash K\supset \Omega_\Gamma$.

Thus if $\cos\alpha_J\ge 0$ in  $\Omega_\Gamma$, then $\Delta \left(e^{\frac{x_1}{2}}\cos\alpha_{J_2}\right)\ge 0 $ in $\Omega_\Gamma$, which further implies  $e^{\frac{x_1}{2}}\cos\alpha_J\le 0$ in $\Omega_\Gamma$ by the maximal principle.  Similarly,  $\cos\alpha_J\le 0$ in $\Omega_\Gamma$ implies $e^{\frac{x_1}{2}}\cos\alpha_J\ge 0$ in $\Omega_\Gamma$.  As a result, $\cos\alpha_J\equiv 0$ in $\Omega_J$.   Moreover,  the  unique continuation theorem implies that $\cos\alpha_{J}\equiv 0$ on $\Sigma$. This contradicts to the assumption $N_J\neq E$.

\vspace{.1in}

\textbf{Claim 2:} {\it There is at most one circle in $G$ which is non-contractible in $E\backslash K$.}

\vspace{.1in}

We check this by contradiction argument again. Otherwise, there are two circles  $\Gamma_1$ and $\Gamma_2$ in $G$ which are both non-contractible in $E\backslash K$. Since there is no contractible circle by Claim 1 and $G$ is properly embedded in $E\backslash K$, by taking $\Gamma_1$ closer to $\Gamma_2$ if necessary, we can assume that  $\Gamma_1$ and $\Gamma_2$  bound a region $\Omega_{1,2}$ in the end $E$ such that $\cos\alpha_J$ does not change sign in $\Omega_{12}$ . By the same argument as that in Claim 1, we get a contradiction.

\vspace{.1in}

 As a result of Claim 1 and Claim 2, there exists at most one circle in $G$. Therefore, by taking $R_J$ larger such that $B_{R_J}$ contains this circle, $G\cap (E\backslash B_{R_J})$ only consists of smooth curves properly embedded in $E\backslash B_{R_J}$.

 \end{proof}

Roughly speaking, the following Lemma  shows that most of the level sets of $\cos\alpha_J$ are not long lines.

\begin{lemma}\label{almost finite length}
Assume $\Sigma\subset\mathbb{R}^4$ is a properly immersed translating soliton satisfying the same assumptions of Theorem \ref{prop-alpha-decay}. For $\mathcal{A}_k$ and $\hat{\mathcal{A}}_k$ constructed in Lemma \ref{construction of annulus}, let $F$  and $F_k$ be defined as
 \begin{align*}
 F&=\{s\in [-1,1]| s \text{ is a singular value  of } \cos\alpha_J \\
  & \quad \quad \quad \quad \quad \quad  \text{ or }  \cos\alpha_J^{-1}(s) \text{ contains a curve with infinite length }  \} \\
  F_k&=\{s\in [-1,1]|  \cos\alpha_J^{-1}(s) \text{ contains an arc joining } \hat{\mathcal{A}}_k \text{ to } \partial \mathcal{A}_k \}.
 \end{align*}
 Then, $\lim_{k\to \infty}\mathcal{L}^1(F_k)=0$ and $\mathcal{L}^1(F)=0$.
 \end{lemma}
 \begin{proof}  Split $F=F_R\cup F_S$, where $F_R$ consists those regular value in $F$ and $F_S$ consists the singular values in $F$. By Sard's theorem, $\mathcal{L}^1(F_S)=0$.
 For any $s\in F_R$, $\cos\alpha_J^{-1}(s)$ embedded in $E$ properly.   Since there exists a curve $\Gamma\subset \cos\alpha_J^{-1}(s)$ of infinite length,  we know there exists  $l>0$ such that for any $k\ge l$,  $\Gamma$ intersects both $\Gamma_k^{\pm}$ non-empty.  By item (1) of Lemma \ref{construction of annulus}, we know $\hat{\mathcal{A}}_k$ separates $\Gamma_k^+$ and $\Gamma_k^{-}$ in $\mathcal{A}_k$. Hence there exists an arc in $\Gamma$ joining $\hat{\mathcal{A}}_k$ to $\partial \mathcal{A}_k$. That is,
 \begin{align*}
 F_R\subset \cup_{l=1}^\infty\cap_{k=l}^\infty F_k.
 \end{align*}
 Moreover, by item (2) of Lemma \ref{construction of annulus}, for any  $s\in F_k$, we have
		\begin{equation}\label{e-lower}
			{\mathcal H}^1\left(\mathcal{A}_k\cap \cos\alpha_J^{-1}(s)\right)\geq c\delta_k R_k.
		\end{equation}
		Thus by the coarea formula and \eqref{e-J-nablaalpha}, we have
		\begin{eqnarray*}
		c\delta_k R_k\mathcal{L}^1(F_k)	
        \leq\int_{F_k}{\mathcal H}^1\left(\mathcal{A}_k\cap \cos\alpha_J^{-1}(s)\right)ds\leq\int_{ \mathcal{A}_k}|\sin\alpha_J\nabla\alpha_{J} |d\mu\leq \int_{\mathcal{A}_k}|\overline{\nabla}J_{\Sigma}|d\mu.
		\end{eqnarray*}
By item (2) and item (4) of Lemma \ref{construction of annulus}, we know
		\begin{align}\label{small measure of long level set}
\mathcal{L}^1(F_k)&\leq (c\delta_kR_k)^{-1}\int_{ \mathcal{A}_k}|\overline{\nabla}J_{\Sigma}|d\mu\\ &\leq (c\delta_k R_k)^{-1}\left(C\delta_kR_k^2\int_{ \mathcal{A}_k}|\overline{\nabla}J_{\Sigma}|^2d\mu\right)^{\frac{1}{2}}\nonumber\\
&\leq (c\delta_k )^{-1}\left(C\delta_k^2\int_{\Sigma\cap A_{\frac{R_k}{2},4R_k}}|{
\bf A
}|^2d\mu\right)^{\frac{1}{2}}\to 0\nonumber
		\end{align}
as $k\to\infty$. Especially, we have
\begin{align*}
\mathcal{L}^1(\cap_{k=l}^\infty F_k)\le \lim_{k\to \infty}\mathcal{L}^1(F_k)=0.
\end{align*}
As a result,
\begin{align}\label{no escape infinite}
\mathcal{L}^1(F)\le \mathcal{L}^1(F_S)+ \sum_{l=1}^\infty\mathcal{L}^1(\cap_{k=l}^\infty F_k)=0.
\end{align}
 \end{proof}

 By Lemma \ref{lem-structure of nodal set}, we know there exists a compact set $K$ such that
 \begin{enumerate}
 \item either $N_J\cap (E\backslash K)=\emptyset$,
 \item  or $N_J\cap(E\backslash K)$ consists of  smooth curves properly embedded in  $E\cap  K$.
 \end{enumerate}

 \vspace{.1in}

 Below we discuss the limit behavior of $\cos\alpha_J$ on $\mathcal{A}_k$ in the two cases.  For this, we introduce some notations.
\begin{definition}
 Assume $U\subset \Sigma$ is a smooth domain in the surface $\Sigma$ with compact closure and $f$ is a smooth function on $\Sigma$.
  We call $s$ a {\bf relative regular value of $f$ with respect to $U$} if $s$ is a regular value of  $f|_{\overline{U}}$. We call $s$ a {\bf relative singular value} if $s$ is not a relative regular value.
 For any subset $A\subset U$ and  $x\in A$, we define
 \begin{align*}
 \Gamma_{f,A,U}(x):=\text{ the connected component of } f^{-1}(f(x))\cap U \text{ passing through } x.
 \end{align*}
 Sometimes, we simply write $\Gamma_{f,U}(x):=\Gamma_{f,A,U}(x)$.
 \end{definition}

Since $\Sigma$ is a translating soliton, it is a critical point of the analytic functional $\int_{\Sigma}e^{\langle x, {\bf T}\rangle}d\mathcal{H}^2$. Hence, $\Sigma$ is analytic  and $\cos\alpha_J=\langle J e_1,e_2\rangle$ is analytic on $\Sigma$ \cite[Section 5.6-5.7]{Morrey}.

\begin{definition}For a graph $\Gamma=G(V,E)$,  its {\bf boundary} $\partial \Gamma$ is defined  by the vertexes only joint with exactly one edge.
 \end{definition}

 \begin{lemma}\label{lem-graph}
 For any non-constant analytic function $f$ on $\Sigma$, an open set $U\subset \Sigma$ and any $s\in f(U)$, the level set $f^{-1}(s)\cap U$ is a graph properly embedded in $U$ with no boundary.
 \end{lemma}
 \begin{proof}
 For any $x\in f^{-1}(s)\cap U$, since $f$ is analytic, by the Weierstrass preparation theorem \cite[Theorem 6.13]{KP} and Puiseux's parametrization theorem \cite[Theorem 4.27 and Theorem 4.28]{KP} of real analytic functions with two variables, we know there exists a neghborhood $V_x$ of $x$ such that $f^{-1}(s)\cap V_x$ is either isolated or  consists of finite many (at least two) analytic arcs with exactly one common end at $x$. See also Lojiasiewicz's structure theorem (\cite[Theorem 6.3.3]{KP}, \cite{Lo}) for more precise and general description.  Thus each connected component $\Gamma$ of $f^{-1}(s)\cap U$ is a connected graph without boundary.  Let $d$ be the  intrinsic length metric on $\Gamma$. For the properness, we argue by contradiction. Otherwise, there exists $x_i\in \Gamma$ such that $x_i\to x_0\in U$ but $d(x_i,x_1)\to \infty$. So, we know $f(x_0)=\lim_{i\to \infty}f(x_i)=s$.  By the local structure of $f^{-1}(s)\cap V_{x_0}$, we know $x_i$ belongs to one of the analytic arcs joining to $x_0$ for $i$ large. Thus there exists $i_0$ such that   $d(x_i,x_j)\le d(x_i,x_0)+d(x_j,x_0)\le 2$ for $i\ge i_0$. This leads to a contradiction
 \begin{align*}
 \infty=\lim_{i\to\infty}d(x_i,x_1)\le\lim_{i\to \infty}\left( d(x_i,x_{i_0})+d(x_{i_0},x_1)\right)\le 2+d(x_{i_0},x_1)<\infty.
 \end{align*}

 \end{proof}

\begin{lemma}\label{isolated singular value} For any non-constant analytic function $f$ on $\Sigma$ and any smooth domain $U\subset \Sigma$ with compact closure, the relative singular value of $f$ with respect to $U$ is finite.
 \end{lemma}
 \begin{proof}
 We argue by contradiction. Otherwise, there exists $s_i\to s$ such that all $s_i$ are different relative singular values of $f$ with respect to $U$. By the definition, there exists $x_i\in \overline{U}$ such that $Df(x_i)=0$.   Since $\overline{U}$ is compact, after passing to a subsequence, we can assume $x_i\to x\in \overline{U}$ and get $Df(x)=0$. Since $f$ is analytic and non-constant, we know there exists a neighborhood $V$ of $x$ such that the critical points of $f$ in $V$ is either  $\{x\}$ or  consists of a finite many arcs $c_i$ with common endpoint $x$.  The former case is impossible since $x_i\to x$ are all different critical points. In the later case, since $Df\equiv 0$ on each such arc $c_i$, we know $f|_{c_i}\equiv f(x)=s$. As a result, the critical value of $f$ in $V$ is exactly $s$, which also contradicts to the fact $x_i\to x$ with $f(x_i)=s_i\neq s$ and $Df(x_i)=0$.
 \end{proof}

 \begin{definition}
 If a smooth curve $\sigma_i: [t_i^-, t_i^+]\to \mathcal{A}_k$ satisfies
 \begin{align*}
 \frac{d f\circ\sigma_i}{dt}(t)<0, \forall t\in (t_i^{-} ,t_i^+),
 \end{align*}
 we call $\sigma_i$ a {\bf negative gradient curve of $f$} in $\mathcal{A}_k$.

 If $\{\sigma_i\}_{i=1}^{l}$ are finite many  negative gradient curves of $f$ in $\mathcal{A}_k$ such that for each $1\le i<l$,  there exists a curve
 $\eta_i:[\tau_i^{-},\tau_i^+]\to \mathcal{A}_k$ satisfying
 \begin{align}
f(\eta_i(\tau))&\ge f(\sigma_{i}(t_i^+)), \forall \tau\in [\tau_i^-,\tau_i^+],\\
 \eta_i(\tau_i^-)=\sigma_i(t_i^+) \text{ for } 1\le i\le l, \quad &\text{ and } \quad \eta_i(\tau_i^+)=\sigma_{i+1}(t_{i+1}^-), \text{ for } 1\le i<l.
 \end{align}
 then we call $\sigma=\cup_{i=1}^l (\sigma_i\cup \eta_i)$ a {\bf pseudo negative gradient curve of $f$} in $\mathcal{A}_k$.    Note that we can choose new parametrization of the curves such that $t_{i+1}^{-}=\tau_{i}^+$ and $t_i^+=\tau_i^-$ such that $\sigma$ is a continuous curve defined on an interval. In the following, we always assume this.

 We say that a pseudo negative gradient curve $\sigma=\cup_{i=1}^{l} \sigma_i\cup \eta_i$ {\bf joins $x$ and $y$} if $\sigma_1(t_1^-)=x$ and $\eta_l(\tau_l^+)=y$.

Moreover, we introduce the notation $c(\sigma)=\cup_{i=1}^l\{\sigma_i(t_i^-),\sigma_i(t_i^+)\}$ and call it the {\bf joint points} of $\sigma$.
 \end{definition}
 Below, we will always choose $f=\cos\alpha_J$.  The following lemma shows the oscillation estimate along pseudo negative gradient curves in $\mathcal{A}_k$. Especially, When $\mathcal{A}_k$ is constructed narraw enough, the oscillation will asymptotically vanish.
\begin{lemma}\label{osc along pseudo gradient curve}  Assume $\Sigma\subset\mathbb{R}^4$ is a properly immersed translating soliton satisfying the same assumptions of Theorem \ref{prop-alpha-decay}. For $\mathcal{A}_k$ constructed in Lemma \ref{construction of annulus} with  $\delta_k=o(\frac{1}{\log R_k})$, there exists $\hat{\delta}_k\to 0$ such that for any $x,y\in \mathcal{A}_k$ joining by a pseudo negative gradient curve of $f=\cos\alpha_J$ in $\mathcal{A}_k$, there holds
\begin{align*}
f(x)-f(y)\le \hat{\delta}_k.
\end{align*}
\end{lemma}
 \begin{proof}
 We argue by contradiction. Otherwise, there exists $\hat{\delta}>0$  and infinite many $k$ such that there are $x_k,y_k\in \mathcal{A}_k$ joining by pseudo negative gradient curves $\sigma^k=\cup_{i=1}^{l_k}(\sigma_i^k\cup\eta_i^k): I_k\to \mathcal{A}_k$  such that $f(x_k)-f(y_k)\ge \hat{\delta}$. Since $f\circ\sigma^k$ is  continuous on the interval $I_k$, its image is connected, which implies
$$[f(y_k),f(x_k)]\subset f(\sigma^k(I_k)).$$

Denote  $\sigma_i^k:I_i^k=[t_{i,k}^{-},t_{i,k}^+]\to \mathcal{A}_k, \quad  \eta_i^k:J_i^k=[\tau_{i,k}^{-},\tau_{i,k}^+]\to \mathcal{A}_k$ in the definition of $\sigma^k$. Then, by  definition of pseudo negative gradient curve, we know
\begin{align*}
[f(y_k),f(x_k)]\cap f(\eta_1^k(J_1^k))\subset f(\sigma_1^k(I_1^k)).
\end{align*}

If for any  $t\in (t_{2,k}^-,t_{2,k}^+)$ such that $\sigma^k(t)\ge \sigma_1^k(t_{1,k}^+)$, then the definition of pseudo negative gradient curve implies that for any $t\in [t_{1,k}^+,\tau_{2,k}^+]$, there holds $f(\sigma^k(t))\ge f(\sigma_1^k(t_{1,k}^+))$. In this case,  we replace $\eta_1^k$ by $\eta_1^k\cup \sigma_2^k\cup \eta_2^k$ and consider $t\in (t_3^-,t_3^+)$. Continue this process, we find a first $i_1\ge 1$ such that $f(\sigma^k(t))\ge f(\sigma_{1,k}(t_{1,k}^+)) $ for any $t\in[t_{1,k}^+, \tau_{i_1,k}^+]$ and  there exists $t\in (t_{i_1+1,k}^-, t_{i_1+1,k}^+)$ such that
$f(\sigma^k(t))<f(\sigma^k(t_{1,k}^+))$. We define
\begin{align*}
\tilde{t}_{2,k}^-:=\tilde{\tau}_{1,k}^+&:=\sup\{t\in (t_{i_1+1,k}^-, t_{i_1+1,k}^+)| f(\sigma^k(t))\ge f(\sigma^k(t_{1,k}^+))\}<t_{i_1+1,k}^+\\
\tilde{\eta}_{2,k}&:=\sigma^k|_{[\tau_{1,k}^-,\tilde{\tau}_{1,k}^+]} \quad \text{ and } \quad  \tilde{\sigma}_2^k:=\sigma^k|_{[\tilde{\tau}_{1,k}^+, t_{i_1+1,k}^+]}.
\end{align*}
By induction construction as above, we get a pseudo negative gradient curve
$\sigma^k=\cup_{j=1}^{\tilde{l}_k}(\tilde{\sigma}_{i_j}^k\cup\tilde{\eta}_{i_j}^k)$ of $f$ with the additional property (See Figure \ref{fig:pseudo netative gradient curve})
\begin{align}
f(\tilde{\sigma}_{i_{j+1}}^k(\tilde{t}_{j+1,k}^-))=f(\tilde{\eta}_{i_j}^k(\tilde{\tau}_{j,k}^+))=f(\tilde{\eta}_{i_j}^k(\tilde{\tau}_{j,k}^-))=f(\tilde{\sigma}_{i_j}^k(\tilde{t}_{j,k}^+)).
\end{align}
\begin{figure}[htbp]
\centering
\includegraphics[width=0.7\textwidth]{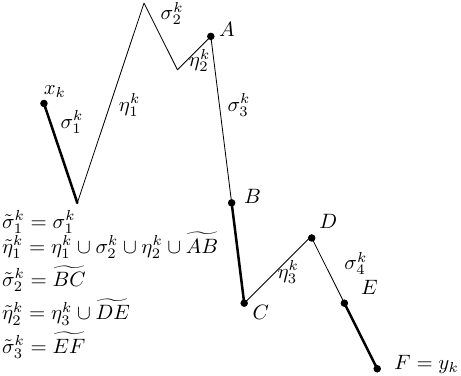}
\caption{ Rearrangement of the pseudo negative gradient curve}
\label{fig:pseudo netative gradient curve}
\end{figure}

To simplify the notations, without loss of generality,  we can assume the original pseudo negative gradient curve $\sigma^k=\cup_{i=1}^l(\sigma_i^k\cup \eta_i^k)$  satisfies
\begin{align*}
f(\sigma_{i+1}^k(t_{i+1,k}^-))=f(\eta_{i}^k(\tau_{i,k}^+))=f(\eta_i^k(\tau_{i,k}^-))=f(\sigma_i^k(t_{i,k}^+)).
\end{align*}
This further implies for any $t\in (t_{i,k}^-,t_{i,k}^+)$ and $s<t$, there holds
\begin{align}\label{pseudo monotonicity}
f(\sigma^k(s))>f(\sigma^k(t)).
\end{align}
Moreover, by the definition of pseudo negative gradient curve and induction argument, we get
\begin{align*}
[f(y_k),f(x_k)]\cap f(\eta_i^k(J_i^k))\subset \cup_{j\le i}f(\sigma_j^k(I_j^k)).
\end{align*}
As a result,
\begin{align*}
[f(y_k),f(x_k)]&=\cup_{i=1}^{l_k}\left([f(y_k),f(x_k)]\cap (f(\sigma_i^k(I_i^k))\cup f(\eta_i^k(J_i^k))\right)\\
&\subset [f(y_k),f(x_k)]\cap \cup_{i=1}^{l_k}(f(\sigma_i^k(I_i^k)))
\end{align*}

Let $c$ be the constant as in Lemma \ref{construction of annulus} and denote
\begin{align*}
S_k:=\{s\in f(\mathcal{A}_k)| \mathrm{either} \quad  s \text{ is a singular value of } \cos\alpha_J  \\\text{ or } \mathcal{H}^1(\cos\alpha_J^{-1}(s)\cap \mathcal{A}_k)\ge c\delta_k R_k \}.
\end{align*}
Then, by the same argument as in the proof of Lemma \ref{almost finite length}, we know  $\mathcal{L}^1(S_k)\to 0$.
Noting that the joint points $c(\sigma^k)$ are finite, we know
\begin{align*}
\mathcal{L}^1\left(f\left(\cup_{i=1}^k\sigma_i^k(I_i^k)\backslash c(\sigma^k)\right)\backslash S_k\right)\ge \hat{\delta}-0-\mathcal{L}^1(S_k)\to \hat{\delta}.
\end{align*}
So, there exists $t_k\in \cup_{i}(t_i^-,t_i^+)$ such that $s_k:=f(\sigma^k(t_k))\notin S_k$ and
$$\min\{|s_k-f(x_k)|,|s_k-f(y_k)|\}\ge \frac{\hat{\delta}}{3}.$$
Since $s_k\notin S_k$ implies $s_k$ is a regular value of $f$ with respect to $\mathcal{A}_k$, and $\mathcal{H}^1(f^{-1}(s_k)\cap \mathcal{A}_k)<c\delta_kR_k<cR_k$, by item (3) of Lemma \ref{construction of annulus}, we know
$\Gamma_{s_k}:=\Gamma_{f,\mathcal{A}_k}(\sigma^k(t_k))$ is a close contractible curve in $\mathcal{A}_k$ which bound a disk $\Omega_{s_k}$ in $\mathcal{A}_k$ and $\Gamma_{s_k}$ intersects transversely to $\sigma^k$ at $\sigma^k(t_k)$. So, there exists small $\varepsilon_k'>0$ such that $\Gamma_{s_k}$ separates $\sigma^k((t_k-\varepsilon_k',t_k))$ and $\sigma^k((t_k,t_k+\varepsilon_k'))$.   By the maximal principle, we know $f|_{\Omega_{s_k}}>s_k=f(\sigma^k(t_k))$, which combining with the strictly decreasing property of  $f\circ \sigma^k(\cdot)$ near $t_k$, we know $\sigma^k((t_k-\varepsilon_k',t_k))\subset \Omega_{s_k}$. Noting $\sigma^k([0,t_k))$ is connected and \eqref{pseudo monotonicity} implies $f(\sigma^k([0,t_k)))\cap \partial \Omega_{s_k}=\emptyset$, we know $\sigma^k([0,t_k))$ belongs to the connected component of $\mathcal{A}_k\backslash \partial \Omega_{s_k}$ passing through $\sigma^k(t_k-\varepsilon_k')$. Especially,  $x_k=\sigma^k(0)\in \Omega_{s_k}$.
 By using Proposition \ref{prop-osi} in $\Omega_{s_k}$ and item (2)(4) of Lemma \ref{construction of annulus}, we get the contradiction
 \begin{align*}
 0<\frac{\hat{\delta}}{3}\le \lim_{k\to\infty}|f(x_k)-s_k|&\le C\lim_{k\to \infty}(\log (C\delta_kR_k^2+2))\int_{ \mathcal{A}_{k}}|\overline{\nabla}J_{\Sigma}|^2d\mu\\
 & \le C\lim_{k\to \infty}(\log (C\delta_kR_k^2+2))\delta_k\int_{\Sigma\cap A_{\frac{R_k}{2},4R_k}}|{\bf A}|^2d\mu= 0,
 \end{align*}
where in the last equality we use $\delta_k=o(\frac{1}{\log R_k})$.
 \end{proof}

\begin{remark} From now on, we will fix $\delta_k=o(\frac{1}{\log R_k})$ (for example,take $\delta_k=\frac{1}{(\log R_k)^2}$) in the construction of $\mathcal{A}_k$ in  the rest of the section.
\end{remark}

In the following lemma, we  construct negative/positive gradient curve from points in the intersection of the upper/lower level set and a small neighborhood of a critical point to its level set.

\begin{lemma}\label{gradient flow to the level set} Assume $x$ is a critical point of an analytic function $f$ on a surface $\Sigma$ and $U\ni x$ is an open set with compact closure. Denote  $c=f(x)$ and
\begin{align*}
U^+=\{y\in U| f(y)>c\} \quad \text{ and } \quad U^-=\{y\in U|f(y)<c\}.
\end{align*}
For each $y\in U^+$, consider the negative gradient flow starting from $y$, that is, $\gamma:[0,T)\to \Sigma$ satisfying
\begin{align*}
\sigma(0)=y, \quad \text{ and } \quad \dot{\sigma}(t)=-\nabla f(\sigma(t)).
\end{align*}
Let
\begin{align*}
T_{\partial U}(y)=\sup \{t| \sigma(\tau)\in U, \forall \tau<t\} \quad \text{ and } \quad
T_{c}(y)=\sup\{t|f(\sigma(\tau))>c, \forall \tau<t\}
\end{align*}
 be the first time of touching $\partial U$ or the level set $f^{-1}(c)$ respectively.
Then,  there exists a smaller neighborhood $\hat{U}$ of $x$ such that
\begin{align}\label{touching time comparison}
T_c(y)\le T_{\partial U}(y),\quad  \forall y\in \hat{U}\cap U^+ .
\end{align}
Moreover, for any $y\in \hat{U}\cap U^+$, there exists a negative gradient curve $\sigma$ in $U$ joining $y$ to some point $y^-\in f^{-1}(c)\cap U$. For any $y\in \hat{U}\cap U^-$, there exists some point $y^+\in f^{-1}(c)\cap U$ and a negative gradient curve $\sigma$ joining $y^+$ to $y$.  See Figure \ref{fig:negative gradient}
 \end{lemma}
 \begin{figure}[htbp]
\centering
\includegraphics[width=0.5\textwidth]{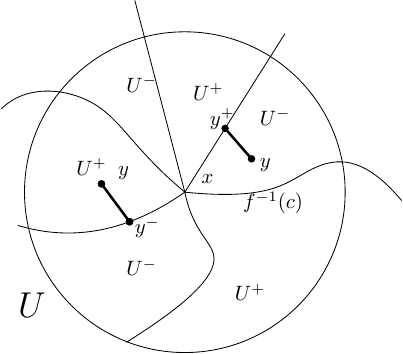}
\caption{Negative gradient flow near a singular point}
\label{fig:negative gradient}
\end{figure}
 \begin{proof} Since $f$ is analytic and $\nabla f(x)=0$, the Lojiasiewicz inequality\cite{Lo} implies there exist some neighborhood $x\in V\subset U$, some constants $C>0$ and $\theta\in (0,1)$ such that
 \begin{align}\label{Loj}
 |\nabla f(y)|\ge C |f(y)-f(x)|^{\theta}, \quad  \forall y\in V.
 \end{align}
 Especially, all the critical points of $f$ in $\overline{V}$ must be contained in the level set $f^{-1}(c)\cap\overline{V}$.
 Now, letting $$\hat{U}=\{y\in V| |f(y)-c|^{1-\theta}<\frac{C\cdot (1-\theta)}{3}d(x,\partial V) \quad \text{ and  } \quad d(y,x)< \frac{1}{3}d(x,\partial V)\},$$  we are going to show \eqref{touching time comparison}  by contradiction argument. More precisely, for any $y\in \hat{U}$,   letting
 $$t_1=\sup \{{t| \sigma(\tau)\in V, \quad } \forall \tau<t \}$$
be the first time for $\sigma$ touching $\partial V$, then $t_1\le T_{\partial U}$ and we will show
\begin{align}\label{touching time comparison new}
T_c(y)\le t_1, \quad \forall y\in \hat{U}.
\end{align}
Otherwise,  there exists some $y\in \hat{U}$ such that
 \begin{align*}
 T_c(y)> t_1.
 \end{align*}
 Then  $f(\sigma(t))\in (c,f(y)]$ and $\nabla f(\sigma(t))\neq 0$  for any $t<t_1 <\infty$, $\sigma(t_1)\in \partial V$.
Define
\begin{align*}
 s(t):=\int_0^t|\dot{\sigma}(\tau)|d\tau=\int_{0}^t|\nabla f(\sigma(\tau))|d\tau.
 \end{align*}
 Letting $\varphi(\tau)=f(\sigma(\tau))$, then $\frac{d\varphi}{d\tau}=-|\nabla f(\sigma(\tau))|^2$. So, $d\tau=-\frac{1}{|\nabla f|^2(\sigma(\tau))}d\varphi$ and \eqref{Loj} implies
 \begin{align*}
 s(t)=\int_{f(\sigma(t))}^{f(y)}\frac{d\varphi}{|\nabla f|(\sigma (\tau)  )}\le C^{-1}\int_{f(\sigma(t))}^{f(y)}\frac{d\varphi}{|f(\sigma(\tau))-c|^\theta}\le \frac{|f(y)-c|^{1-\theta}}{C\cdot (1-\theta)}.
 \end{align*}
 Letting $t\to t_1$, by the definition of $\hat{U}$,  we get the contradiction
 \begin{align*}
 0<d(x,\partial V)&\le d(x,y)+\lim_{t\to t_1} s(t)\\
 &\le \frac{1}{3}d(x,\partial V)+\frac{|f(y)-c|^{1-\theta}}{C\cdot (1-\theta)}\\
 &\le \frac{2}{3}d(x,\partial V).
 \end{align*}
 Now, for $y\in \hat{U}\cap U^+$,  if $T_c(y)=\infty$, then \eqref{touching time comparison new} implies $\sigma(t)$ stays in the compact subset  $\overline{V}$ for $t\in (0,\infty)$ and the Lojiasiewicz inequality implies $\lim_{t\to\infty}\sigma(t)=y^-$ for some critical point  $y^-\in \overline{V}$\cite[Page 70, Section 3.10,  Theorem 2]{Simon}, thus $y^-\in f^{-1}(c)$ and $\sigma:[0,\infty)\to U$ is a negative gradient curve joining $y$ to  $y^-$. If $T_c(y)<\infty$, then the definition of $T_c(y)$ implies $y^-:=\sigma(T_c(y))\in f^{-1}(c)$ and $\sigma:[0,T_c(y)]$ is a negative gradient curve joining $y$ to $y^-$.
 The conclusion for $y\in \hat{U}\cap U^-$ holds similarly.
 \end{proof}

 \begin{lemma} \label{boundary intersection} Let $\mathcal{A}$ be a closed set in a surface $\Sigma$ with connected interior $\mathring{\mathcal{A}}$ and $D\subset \mathcal{A}$ is closed subset such that
 \begin{align*}
 D\cap \mathring{\mathcal{A}}\neq \emptyset \quad \text{ and } \quad  \mathring{\mathcal{A}}\backslash D\neq \emptyset.
 \end{align*}
 Then, for any connected component $V$ of $\mathring{\mathcal{A}}\backslash D$, there holds
 \begin{align*}
 \partial V\cap \partial D\neq \emptyset.
 \end{align*}
 Moreover, there also holds
 $$\partial D\cap \mathring{\mathcal{A}}\neq \emptyset.$$
 \end{lemma}
 \begin{figure}[htbp]
\centering
\includegraphics[width=0.3\textwidth]{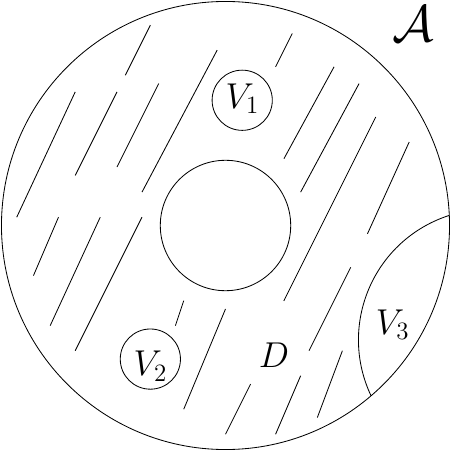}
\caption{ Picture for Lemma \ref{boundary intersection}}
\label{fig:domain}
\end{figure}
 \begin{proof}
 Denote the connected components decomposition of  $\mathring{\mathcal{A}}\backslash D$ by  $\sqcup_{\alpha\in \Lambda}V_\alpha$. Since $D$ is closed, we know each  $V_\alpha$ is open and there exits $\alpha_0\in \Lambda$ such that $V=V_{\alpha_0}$.
 Consider the boundary $\partial V=\overline{V}\backslash V$.  To show $\partial V\cap \partial D\neq \emptyset$,  it is enough to show
 \begin{align}\label{intersects no empty}
 \partial V\cap D\neq \emptyset,
 \end{align}
 since if so, then by noting $\partial V\cap \mathring{D}=\emptyset$,  we get the conclusion $\partial V\cap \partial D\neq \emptyset$.

 For \eqref{intersects no empty}, we argue by contradiction. Otherwise, for any $x\in \partial V$, there holds $x\notin D$. Since $\mathcal{A}$ is closed and $V\subset \mathcal{A}$, we know $x\in \partial V\subset \mathcal{A}$. As a result,
 $$x\in \mathcal{A}\backslash D\subset {\partial \mathcal{A}} \cup \mathring{\mathcal{A}}\backslash D. $$
 If $x\in \mathring{\mathcal{A}}\backslash D=\sqcup_{\alpha}V_{\alpha}$, by $x\in \partial V=\partial V_{\alpha_0}$, we know $x\notin V_{\alpha_{0}}$.  Then,  $x\in V_{\alpha}$ for some $\alpha\neq \alpha_0$, which contradicts $x\in \partial V_{\alpha_0}$, since $V_\alpha$ and $V_{\alpha_0}$ are two disjoint open sets.

 As  a result, we get $\partial V\subset \partial \mathcal{A}$  and
$ \mathring{\mathcal{A}}\backslash \partial V=\mathring{\mathcal{A}}$.
So,
 \begin{align*}
 \mathring{\mathcal{A}}\backslash \overline{V}=\mathring{\mathcal{A}}\backslash (\partial V \cup V)=(\mathring{\mathcal{A}}\backslash \partial V)\cap (\mathring{\mathcal{A}}\backslash V)=\mathring{\mathcal{A}}\backslash V.
 \end{align*}
 Thus both $\mathring{\mathcal{A}}\backslash V$ and $V$ are open sets of $\mathring{\mathcal{A}}$. Since $\mathring{\mathcal{A}}$ is connected and $V\neq \emptyset$, we know $V=\mathring{\mathcal{A}}$.  This implies
 \begin{align*}
 \mathring{\mathcal{A}}\cap D=\mathring{\mathcal{A}}\backslash (\mathring{\mathcal{A}}\backslash D)=V\backslash (\mathring{\mathcal{A}}\backslash D)\subset V\backslash V=\emptyset,
 \end{align*}
 which contradicts to  our assumption $\mathring{\mathcal{A}}\backslash D\neq \emptyset$.  Next, we show $\partial D\cap \mathring{\mathcal{A}}\neq \emptyset$ by similar contradiction argument.  More precisely, if $$\partial D\subset \partial\mathcal{A},$$ then $\mathring{\mathcal{A}}=\mathring{\mathcal{A}}\backslash \partial D$, which further implies
 \begin{align*}
 \mathring{\mathcal{A}}\backslash D=(\mathring{\mathcal{A}}\backslash \partial D)\cap (\mathring{\mathcal{A}}\backslash \mathring{D})=\mathring{\mathcal{A}}\cap (\mathring{\mathcal{A}}\backslash \mathring{D})=\mathring{\mathcal{A}}\backslash \mathring{D}.
 \end{align*}
 As a result, $\mathring{\mathcal{A}}\backslash \mathring{D}=\mathring{\mathcal{A}}\backslash D$ is an open set in $\mathring{\mathcal{A}}$. But $\mathring{D}$ is also open. So, by the connectedness of $\mathring{\mathcal{A}}$, we know
 either $\mathring{\mathcal{A}}=\mathring{D}$ or $\mathring{D}=\emptyset$. In the  case $\mathring{\mathcal{A}}=\mathring{D}$, we get $\mathring{\mathcal{A}}\backslash D=\emptyset$, which contradicts to our assumption. In the other case $\mathring{D}=\emptyset$, we know $D=\partial D\subset \partial \mathcal{A}$, which contradicts to the other assumptioin $D\cap \mathring{\mathcal{A}}\neq \emptyset$.
 \end{proof}

  \begin{proposition}\label{No nodal line implies limit exists}
  If $N_J\cap E\backslash K=\emptyset$, then $\lim_{k\to\infty}\sup_{x\in \mathcal{\hat{A}}_k}|\cos\alpha_J-\cos\overline{\alpha}_J|=0$ for some $\overline{\alpha}_J$.
 \end{proposition}

 \begin{proof}
  Since $N_J\cap E\backslash K=\emptyset$, by changing $J$ to $-J$ if necessary,  we can assume  $f:=\cos\alpha_J>0$ is a non-constant analytic function  without loss of generality.  Since $\hat{\mathcal{A}}_k$ is connected, we can assume
 \begin{align*}
 \overline{f(\hat{\mathcal{A}_k})}=[a_k,b_k].
 \end{align*}
 Then, the goal is to show
 \begin{align*}
 \mathrm{osc}_{\hat{\mathcal{A}}_k} f=b_k-a_k\to 0 \quad \text{ as } \quad k\to \infty.
 \end{align*}
Letting $c$ be the constant as in Lemma \ref{construction of annulus} and denote
\begin{align*}
S_k:=\{s\in [a_k,b_k]| \mathrm{either} \  s \text{ is a singular value of } \cos\alpha_J  \text{ or } \mathcal{H}^1(\cos\alpha_J^{-1}(s)\cap \mathcal{A}_k)\ge c R_k \}.
\end{align*}
Then, by the same argument as in the proof of Lemma \ref{almost finite length}, we know  $\mathcal{L}^1(S_k)\to 0$.  If $S_k=[a_k,b_k]$, then
$b_k-a_k=\mathcal{L}^1(S_k)\to 0$ and the work is done. So, we can assume $[a_k,b_k]\backslash S_k\neq \emptyset$.  Fix $\varepsilon_k\to 0$ and take $s_0\in [a_k,b_k]\backslash S_k$ such that $t\le s_0+\varepsilon_k$ for any other $t\in [a_k,b_k]\backslash S_k$. Then,  $\mathcal{H}^1(f^{-1}(s_0)\cap \mathcal{A}_k)<c R_k$.  Especially, by item (3) of  Lemma \ref{construction of annulus}, we know  for any $x\in f^{-1}(s_0)$,
$$\Gamma_{s_0}:=\Gamma_{f,\hat{\mathcal{A}}_k, \mathcal{A}_k}(x) \text{ is a contractible closed curve in } \mathring{\mathcal{A}}_k.$$
Thus $\Gamma_{s_0}$ bounds a disk $\Omega_{s_0}$ in $\mathring{\mathcal{A}}_k$. Taking $x_{\mathrm{max}}\in \overline{\Omega_{s_0}}$ be a local maximal point, we can assume $x_{\mathrm{max}}\in \Omega_{s_0}$ and $f(x_{\mathrm{max}})=s_{\mathrm{max}}>s_0$. Otherwise, $f_{\Omega_{s_0}}\equiv s_0$ and the unique continuation theorem implies $\cos\alpha_J\equiv s_0$ on $E$, which contradicts to the assumption in the beginning of the proof.   Now, since $f$ is analytic and non-constant, by Lemma \ref{lem-graph},  $\Gamma_{f, \mathcal{A}_k}(x_{\mathrm{max}})$ consists of a graph properly embedded graph in $\mathcal{A}_k$ without boundary in $\mathring{\mathcal{A}}_k$. Noting $\Gamma_{s_0}$ separates $x_{\max}$ and $\partial \mathcal{A}_k$ in $\mathcal{A}_k$ and  $\Gamma_{f,\mathcal{A}_k}(x_{\mathrm{max}})$ is connected, we know $\Gamma_{f,\mathcal{A}_k}(x_{\mathrm{max}})\subset \Omega_{s_0}$.  If $\Gamma_{f,\mathcal{A}_k}(x_{\mathrm{max}})$ contains a circle,  then it will bound a disk, and $f\equiv s_{\mathrm{max}}$ by the maximal principle and unique continuation theorem, again a contradiction. So, we can assume $\Gamma_{f,\mathcal{A}_k}(x_{\mathrm{max}})$ contains no circle and is a tree properly embedded in $\Omega_{s_0}$ without boundary. Since $s_{\mathrm{max}}>s_0$, we also know the tree has no boundary on $\partial \Omega_{s_0}$. As a result, the tree has no boundary, which implies
\begin{align}\label{isolated maximal point}
\Gamma_{f,\mathcal{A}_k}(x_{\mathrm{max}})\text{ is an isolated point(a trivial tree)}.
\end{align}  By Lemma \ref{isolated singular value}, there exists $\delta>0$ such that all $s\in (s_\mathrm{max}-\delta, s_{\mathrm{max}})$ are regular values. Choose a small disk $D\ni x_{\mathrm{max}}$ such that $f(\overline{D}\backslash \{x_{\mathrm{max}}\})\subset (s_\mathrm{max}-\delta, s_{\mathrm{max}})$ and denote $\hat{s}=\max_{x\in \partial D}f(x)<s_{\mathrm{max}}$. Then, for any $s\in (\hat{s}, s_{\mathrm{max}})$ and $x\in D\cap f^{-1}(s)$, $\Gamma_{f,\mathcal{A}_k}(x)$ is contained in $D$, hence $s$ is a relative regular value  of $f$ with respect to $D$ and $\Gamma_{s}:=\Gamma_{f,\mathcal{A}_k}(x)$ is a closed curve in $D$.  Denoting $\Omega_{s}$ by the disk bounded by $\Gamma_{s}$ in $D$, then
\begin{align*}
x_{\mathrm{max}}\in \Omega_{s}.
\end{align*}
In fact, there exists a local maximal point $\hat{x}_{\mathrm{max}}\in \Omega_{s}\subset D$, but $f(D\backslash\{x_{\mathrm{max}}\})\subset (s_\mathrm{max}-\delta, s_{\mathrm{max}})$ implies $x_{\mathrm{max}}$ is the only critical value in $D$. So, $\hat{x}_{\mathrm{max}}=x_{\mathrm{max}}$ and $x_{\mathrm{max}}\in \Omega_{s}$. Moreover, for any $x\in \Gamma_{s}$, Lojiasiewicz's inequality implies there exists a curve $\alpha_x(t)$ such that $\alpha_x(0)=x$, $\dot{\alpha}_x(t)=\nabla f(\alpha_x(t))$ and $\alpha_x(t)$ converges to some critical point of $f$. Noting $f(\alpha_x(t))>f(x)=s$, we know $\alpha_x(t)\subset D$. Hence the limit of $\alpha_x(t)$ is the unique critical point $x_{\mathrm{max}}$ in $D$. This way, $\alpha_x(t)$ is a curve joining $x$ and $x_{\mathrm{max}}$ and for each $\tilde{s}=f(\alpha_{x}(t))$, $\Gamma_{\tilde{s}}:=\Gamma_{f,\mathcal{A}_k}(\alpha_x(t))$ is a closed curve bounding a disk $\Omega_{\tilde{s}}\ni x_{\mathrm{max}}$. As a result,
\begin{align*}
\{\Gamma_{s}\}_{s\in (\hat{s},s_{\mathrm{max}})} \text{ shapes a foliation of the punctured  neighborhood }\\ \cup_{\hat{s}<s<s_{\mathrm{max}}}\Omega_s\backslash \{x_{\mathrm{max}}\} \text{ of }   x_{\mathrm{max}}.  \text{ See Figue \ref{fig:level set}}
\end{align*}
\begin{figure}[htbp]
\centering
\includegraphics[width=0.6\textwidth]{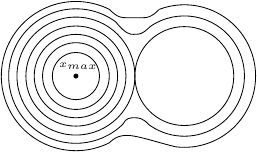}
\caption{ Level set near local maximal point}
\label{fig:level set}
\end{figure}
By taking $D$ smaller if necessary, we can assume $(\hat{s},s_{\mathrm{max}})\subset [a_k,b_k]\backslash S_k$.
Next, we want to extend these contractible circles bounding disks containing $x_{\mathrm{max}}$ bigger and bigger until the function values on these circles ``decrease"  to $a_k$.
More precisely,
define
\begin{align*}
C_k:=\{x\in \mathcal{A}_k| \text{ there exists a pseudo negative gradient  curve } \sigma  \text{ joining } x_{\mathrm{max}} \text{ and } x \}.
\end{align*}
Then, the above discussion implies
$$\cup_{s\in (\hat{s},s_{\mathrm{max}})}\Omega_s\subset C_k.$$
Moreover, Lemma \ref{osc along pseudo gradient curve} implies
\begin{align*}
\mathrm{osc}_{C_k}f\to 0.
\end{align*}
So, we only need to show $C_k$ extends large enough in the sense there exists a point $x_k\in C_k$ such that
\begin{align}\label{extend large}
f(x_k)-a_k\le \varepsilon_k\to 0.
\end{align}
Since then, by the choice of $s_0$ (that is,  $[a_k,b_k]\backslash S_k\subset [a_k, s_0+\varepsilon_k]$) and the fact $s_{\mathrm{max}}>s_0$, we know
$[a_k,b_k]\subset S_k\cup [a_k,s_0+\varepsilon_k]\subset S_k\cup [a_k,s_{\mathrm{max}}+\varepsilon_k]$, which further implies
\begin{align*}
b_k-a_k=\mathcal{L}^1([a_k,b_k])&\le \mathcal{L}^1(S_k)+s_{\mathrm{max}}+\varepsilon_k-a_k\\
&\le \mathcal{L}^1(S_k)+s_{\mathrm{max}}-f(x_k)+2\varepsilon_k\\
&\le \mathcal{L}^1(S_k)+\mathrm{osc}_{C_k}f+2\varepsilon_k\to 0.\\
\end{align*}

Define
\begin{align*}
c_k=\inf\{f(x)| x\in C_k\}.
\end{align*}
We want to show \eqref{extend large}, that is,   $c_k$ is close to $a_k$.  In fact, we will show $c_k=a_k$.

For this, we consider

\begin{align*}
D_k=\overline{C_k\cap \hat{\mathcal{A}_k}}.
\end{align*}
 Then, $f(D_k)\subset [c_k,b_k]$.

If $\mathring{\hat{\mathcal{A}}}_k\backslash D_k\subset \{f\ge c_k\}$, then $\hat{\mathcal{A}}_k\backslash D_k\subset \{f\ge c_k\}$ and
\begin{align*}
[a_k,c_k)\subset f(D_k)\subset [c_k,b_k],
\end{align*}
which implies $a_k=c_k$ and \eqref{extend large} holds trivially.
 Otherwise,  we can assume  $a_k<c_k$ and  $ \mathring{\hat{\mathcal{A}}}_k\backslash D_k\neq \emptyset$, then by Lemma \ref{boundary intersection}, we know
 $$\partial D_k\cap\mathring{\hat{\mathcal{A}}}_k\neq \emptyset.$$

   In the following, we finish the proof by  showing a contradiction to exclude this possibility.  Choose a point $x\in \partial D_k\cap \mathring{\hat{\mathcal{A}}}_k\neq \emptyset.$

 If $x$ is a regular point of $f$, then there exists an open disk $U_x\ni x$ such that $L_x:=U_x\cap \Gamma_{f,\mathcal{A}_k}(x)$ is a segment such that $U_x\backslash L_x=U_x^+\sqcup U_x^-$, where $U_x^{\pm}$ are both open half disk such that $f|_{U_x^+}>f(x)\ge c_k$ and $f|_{U_x^-}<f(x)$.  By taking $U_x$ smaller if necessary, the structure of the negative gradient flow near a regular point implies  that for any $y\in U_x$, there exists a  negative gradient curve $\sigma_{y}:[t_{y}^-,t_{y}^+]\to U_x $ such that
 \begin{align}
 \begin{cases}
 \sigma_y(t_y^-)=y  \text{ and }\sigma_y(t_y^+)\in L_x,   & \text{ if } y\in U_x^+\cup L_x,\\
 \sigma_y(t_y^-)\in L_x \text{ and } \sigma_y(t_y^+)=y &\text{ if } y\in U_x^-\cup L_x.
  \end{cases}
 \end{align}
 Define
 \begin{align*}
 \begin{cases}
 \eta_{y}(t)=\sigma_y(t_y^+-t)& \text{ if } y\in U_x^-\\
 \eta_{y_1,y_2}=\text{a curve in } U_x^+\cup L_x \text{ joining } y_1 \text{ and } y_2 &\text{ if } y_1,y_2\in U_x\cup L_x.
 \end{cases}
 \end{align*}
 Since $x\in \partial D_k\cap \mathring{\hat{\mathcal{A}}}_k$, there exists $x_i\to x$ such that $x_i\in C_k$. Then $x_i\in U_x$ for $i$ large.  Denote the pseudo negative gradient curve joining $x_{\mathrm{max}}$ to $x_i$ by $\sigma_{x_{\mathrm{max}},x_i}$.
 If $x_i\in U_x^+$, then
 \begin{align*}
 \sigma_{x_{\mathrm{max},y}}=
 \begin{cases}
 \sigma_{x_{\mathrm{max}},x_i}\cup \sigma_{x_i}\cup \eta_{\sigma_{x_i}(t_{x_i}^+),y}& \text{if } y\in U_x^+\cup L_x\\
 \sigma_{x_{\mathrm{max},x_i}}\cup \sigma_{x_i}\cup \eta_{\sigma_{x_i}(t_{x_i}^+),\sigma_y(t_y^-)}\cup \sigma_y & \text{ if } y\in U_x^-
 \end{cases}
 \end{align*}
 is a pseudo negative gradient curve joining $x_{\mathrm{max}}$ to $y$.
Similarly, if $x\in U_x^-\cup L_x$, then
\begin{align*}
 \sigma_{x_{\mathrm{max},y}}=
 \begin{cases}
 \sigma_{x_{\mathrm{max}},x_i}\cup (\eta_{x_i}\cup \eta_{\sigma_{x_i}(t_{x_i}^-),y})& \text{if } y\in U_x^+\cup L_x\\
 \sigma_{x_{\mathrm{max},x_i}}\cup (\eta_{x_i}\cup \eta_{\sigma_{x_i}(t_{x_i}^+),\sigma_y(t_y^-)})\cup \sigma_y & \text{ if } y\in U_x^-
 \end{cases}
 \end{align*}
  is a pseudo negative gradient curve joining $x_{\mathrm{max}}$ to $y$. See Case 1 of Figure \ref{fig: extending}  for the example $x_i\in U_x^+$ and $y\in U_x^-$.
 As a result, we know $U_x\subset C_k\subset D_k$, which contradicts to $x\in \partial D_k$.
\begin{figure}[htbp]
\centering
\includegraphics[width=1.0\textwidth]{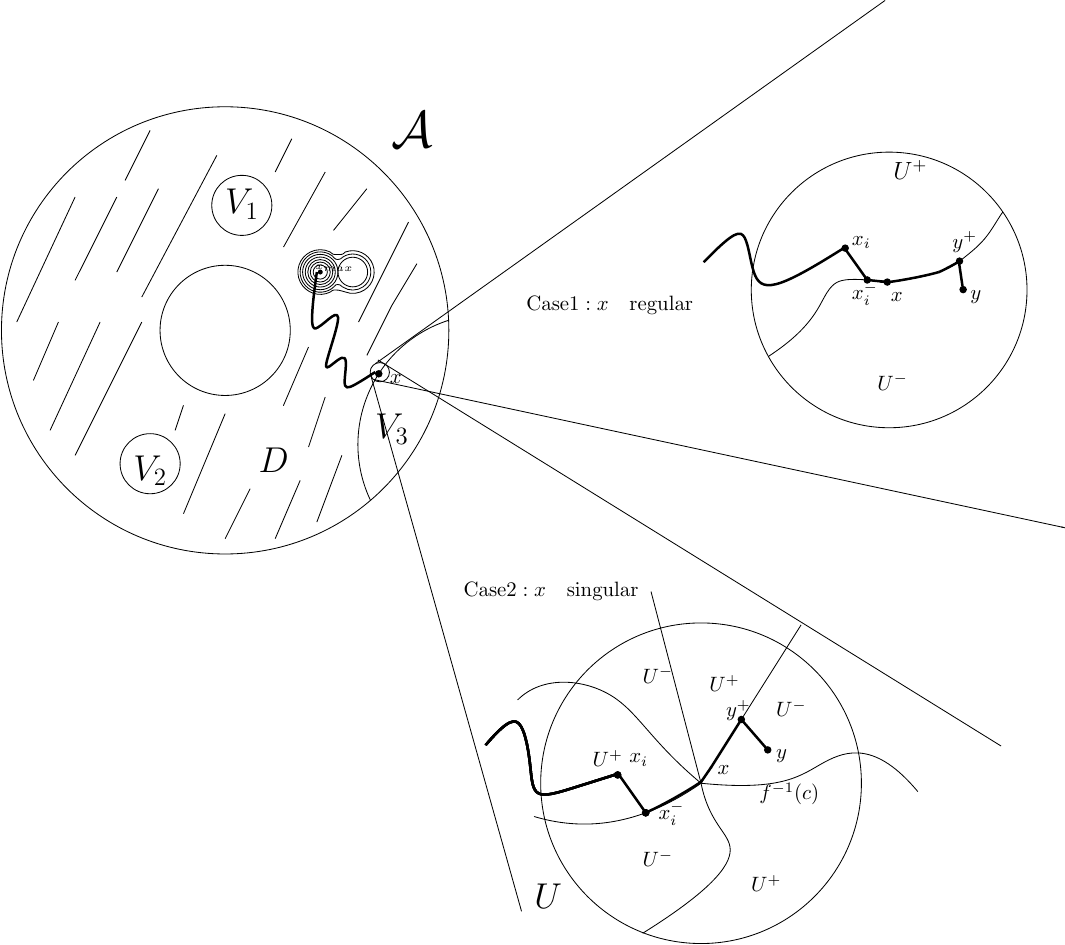}
\caption{ Extending pseudo negative gradient curve near the boundary point $x\in \partial D_k\cap \mathring{\hat{\mathcal{A}}}_k$}
\label{fig: extending}
\end{figure}

 If $x$ is a singular point of $f$, since $f$ is analytic, we know  there exists a disk $U_x\ni x$ such that  $\Gamma_x=\Gamma_{f,\mathcal{A}_k}(x)\cap U_x$
 is path connected and $U_x\backslash \Gamma_x$ consists of finite many disjoint  disks. Denote \
 $$U_x^+=\{y\in U_x\backslash \Gamma_x|f(y)>f(x)\} \quad \text{ and } \quad U_x^-=\{y\in U_x\backslash \Gamma_x|f(y)<f(x)\}.$$
 Then, by Lemma \ref{gradient flow to the level set}, there exists a smaller disk  $U_x'\subset \subset U_x$ such that for any $y\in U_x'\cap U_x^+$, there exists a negative gradient curve $\sigma_y$ joining $y$ to some point $y^-\in \Gamma_x$, and for any $y\in U_x'\cap U_x^-$, there exists a negative gradient curve $\sigma_y$ joining some  point $y^+\in  \Gamma_x$ to $y$. Similar as before, denote the inverse curve of $\sigma_y$ by $\eta_y$ for $y\in U_x'\backslash L_x$ and define $\eta_{y_1,y_2}$ by a curve joining $y_1$ to $y_2$ in $\Gamma_x$ if both $y_1,y_2\in \Gamma_x$. Then, by choosing $x_i\in U_x'\cap C_k$ and negative gradient curve $\sigma_{x_{\mathrm{max}},x_i}$ joining $x_{\mathrm{max}}$ and $x_i$, and defining
 \begin{align*}
 \sigma_{x_{\mathrm{max}},y}=
 \begin{cases}
 \sigma_{x_{\mathrm{max},x_i}}\cup (\eta_{x_i}\cup \eta_{x_i^+,y^-}\cup \eta_{y}) &\text{ if } x_i\in U_x'\cap (U_x^-\cup L_x), y\in U_x'\cap (U_x^+\cup \Gamma_x)\\
 \sigma_{x_{\mathrm{max},x_i}}\cup (\eta_{x_i}\cup \eta_{x_i^+,y^+})\cup \sigma_{y} &\text{ if } x_i\in U_x'\cap (U_x^-\cup L_x), y\in U_x'\cap U_x^-\\
 \sigma_{x_{\mathrm{max},x_i}}\cup \sigma_{x_i}\cup( \eta_{x_i^-,y^-}\cup \eta_{y}) &\text{ if } x_i\in U_x'\cap U_x^+, y\in U_x'\cap (U_x^+\cup \Gamma_x)\\
  \sigma_{x_{\mathrm{max},x_i}}\cup \sigma_{x_i}\cup \eta_{x_i^-,y^+}\cup \sigma_{y} &\text{ if } x_i\in U_x'\cap U_x^+, y\in U_x'\cap U_x^-,
 \end{cases}
 \end{align*}
 we get pseudo negative gradient curves joining $x_{\mathrm{max}}$ to $y\in U_x'$. See Case 2 of  Figure \ref{fig: extending} for the example $x_i\in U_x^+$ and $y\in U_x^-$. As a result, $x\in \mathring{D}_k$, which again contradicts to $x\in \partial D_k$.
 \end{proof}

 \begin{proposition}\label{nodal line implies zero}
 If $N_J\neq E$ and
 $N_J\cap E\backslash K\neq\emptyset$, then $\lim_{k\to\infty}\sup_{x\in \mathcal{\hat{A}}_k}|\cos\alpha_J|=0.$
 \end{proposition}

 \begin{proof}
 Similar to the proof of  Proposition \ref{No nodal line implies limit exists}. More precisely,  by Lemma \ref{lem-structure of nodal set}, we can assume $N_J\cap(E\backslash K)$ consists of  smooth curves $\{\Gamma_i\}_{i=1}^{l\le \infty}$ properly embedded in  $E\cap  K$ and $E\backslash K$ is homeomorphic to $\mathbb{D}^c$. Thus $(E\backslash K)\backslash N_J=\sqcup_{i=1}^l \mathcal{R}_i$, where each $\mathcal{R}_i$ is a contractible non-compact domain  whose boundary consists of two curves in $N_J$ and a   compact arc(may be empty set) in $\partial (E\backslash K)$. Moreover, for fix $1\le i\le l\le \infty$,  $\cos\alpha_J$ does not change sign in  $\mathcal{R}_i$. By replacing $J$ with $-J$ if necessary, we can assume $\cos\alpha_J\ge 0$ in $\mathcal{R}_i$, and it is enough to prove
 \begin{align}\label{limit between two lines}
 \lim_{k\to \infty}\sup_{x\in \mathcal{A}_k\cap \mathcal{R}_i}\cos\alpha_J=0.
 \end{align}
 Noting that $\mathcal{A}_k\cap R_i$ is connected, we can assume
 \begin{align*}
 \overline{\{\cos\alpha_J(x)|x\in \mathcal{A}_k\cap \mathcal{R}_i\}}=[a_k,b_k].
 \end{align*}
 Since $a_k=\inf_{x\in \mathcal{A}_k\cap R_i}\cos\alpha_J=0$, by the same argument as in the proof of Proposition \ref{No nodal line implies limit exists}, we know
 \begin{align}\label{osc estimate}
  \lim_{k\to \infty}\sup_{x\in \mathcal{A}_k\cap \mathcal{R}_i}\cos\alpha_J=\lim_{k\to\infty}(b_k-a_k)=\lim_{k\to\infty}\mathrm{osc}_{\mathcal{A}_k\cap R_i} \cos\alpha_J=0.
 \end{align}
\end{proof}

 \begin{corollary}
 For $J_1,J_2,J_3$, there is at least one  $i\in \{1,2,3\}$ such that $\cos\alpha_{J_i}$ does not change sign outside a compact set of $E$.
 \end{corollary}
 \begin{proof}
 We argue by contradiction. Otherwise, by Proposition \ref{nodal line implies zero}, we know
 $$\lim_{k\to \infty}\max_{\mathcal{\hat{A}}_k}|\cos\alpha_{J_i}|=0, \forall i=1,2,3.$$
 This contradicts to $\sum_{i=1}^3\cos^2\alpha_{J_i}=1$.
 \end{proof}
 \begin{corollary}
Without loss of generality, we can assume there exists $K\subset E$ such that
$
 N_{J_1}\cap (E\backslash K)=\emptyset.
$
 Replace $J_1$ by $-J_1$ if necessary,  $\cos\alpha_{J_1}>0$ on $E\backslash K$ and $\lim_{ \mathcal{\hat{A}}_k\ni x\to \infty}\cos\alpha_{J_1}$ exists.
\end{corollary}

Now we can prove Theorem \ref{prop-alpha-decay}.

\vspace{.1in}

\noindent {\it Proof of Theorem \ref{prop-alpha-decay}.} Passing to a subsequence if necessary, there exist constants  $\bar{\alpha}_{J_1,j}, \bar{\alpha}_{J_2,j}, \bar{\alpha}_{J_3,j}\in (0,\frac{\pi}{2}]$ with the following property: on the fixed end $\Sigma_j$, we can find a sequence $\mathcal{\hat{A}}_k$ diverging to  infinity such that
\begin{align}\label{e-bar-1}
    \lim_{k\to\infty}\sup_{x\in \mathcal{\hat{A}}_k}|\cos\alpha_{J_i}-\cos\overline{\alpha}_{J_{i,j}}|=0,
\end{align}
for $i=1,2,3$. Notice that we have
\begin{align}\label{e-123}
    \cos^2{\bar \alpha}_{J_1,j}+\cos^2{\bar \alpha}_{J_2,j}+\cos^2{\bar \alpha}_{J_3,j}=1.
\end{align}

For the subsequent argument, we introduce a new complex structure
		\begin{equation}\label{choose-J}
I_{j}=\cos{\bar \alpha}_{J_1,j} J_1+\cos{\bar \alpha}_{J_2,j} J_2+\cos{\bar \alpha}_{J_3,j} J_3.
		\end{equation}
Then on any surface in ${\mathbb R}^4$, we have
		\begin{equation}\label{choose-alpha-J}
\cos\alpha_{I_j}=\cos{\bar \alpha}_{J_1,j}\cos\alpha_{J_1}+\cos{\bar \alpha}_{J_2,j}\cos\alpha_{J_2}+\cos{\bar \alpha}_{J_3,j}\cos\alpha_{J_3}.
		\end{equation}
(\ref{e-bar-1}) and (\ref{e-123}) imply that
		\begin{equation}\label{e-I-j-2}
\lim_{k\to\infty}\max_{ \Sigma_j\cap \mathcal{\hat{A}}_k}|\cos\alpha_{I_j}-1|=0.
		\end{equation}
Thus
		\begin{equation}\label{e-I-j}
			\lim_{k\to\infty}\max_{ \Sigma_j\cap \mathcal{\hat{A}}_k}|\alpha_{I_j}|=0.
		\end{equation}
Set
		\begin{equation*}
			u_{j}=\sin^2\frac{\alpha_{I_j}}{2}\exp\left(\frac{x_1}{2}\right).
		\end{equation*}
By (\ref{e-I-j-2}) and Proposition \ref{nodal line implies zero}, we know that $\cos\alpha_{I_j}>0$ on $\Sigma_j\backslash B_{R_{2}}$ for some $R_{2}>R_1$, where $R_1$ is given by Lemma \ref{lem-f-2}. Using (\ref{e-sinalpha}) and (\ref{e-e-x1}) we have that on $\Sigma_j\backslash B_{R_1}$
		\begin{eqnarray*}
			\Delta u_j
			&=&\Delta\left[\sin^2\frac{\alpha_{I_j}}{2}\exp\left(\frac{x_1}{2}\right)\right]\\
			&=&\left[\Delta\sin^2\frac{\alpha_{I_j}}{2}+\langle\nabla x_1,\nabla\sin^2\frac{\alpha_{I_j}}{2}\rangle\right]\exp\left(\frac{x_1}{2}\right)+\sin^2\frac{\alpha_{I_j}}{2}\Delta\exp\left(\frac{x_1}{2}\right)\\
			&=&\frac{1}{2}|\overline{\nabla}J_{\Sigma}|^2\cos\alpha_{I_j}\exp\left(\frac{x_1}{2}\right)+\frac{|{\bf H}|^2+1}{4}u_j\\
			&\geq & \frac{|{\bf H}|^2+1}{4}u_j.
		\end{eqnarray*}
		Define $h_{\eta}(x)=\eta\exp\left(\frac{x_1}{2}\right)$ and $f_{B}(x)=B|x|^{-\beta}\exp\left(-\frac{|x|}{2}\right)$. It follows from Lemma \ref{lem-f-2} that
		\begin{equation*}
			\Delta (f_B+h_{\eta}-u_j)\leq \frac{|{\bf H}|^2+1}{4} (f_B+h_{\eta}-u_j) \ \ \text{for  all} \ |x|\geq R_2,
		\end{equation*}
		where $R_2>R_1$ is chosen large enough so that $\partial\Sigma_j\subset B_{R_2}$ and the constant $B$ is chosen so that, for every $j=1,\cdots, N$, we have
		\begin{equation*}
			x\in \Sigma_j\cap \partial B_{R_2} \Rightarrow f_B(x)>u_j(x).
		\end{equation*}

        Next, by the choice of $\mathcal{\hat{A}}_k$, we know that there exists $k_0$ such that for each $k\geq k_0$, $\mathcal{\hat{A}}_k\subset \Sigma\backslash B_{R_2}$. Thus we can find a homotopically nontrivial closed curve $\gamma_k$ in  $\mathcal{\hat{A}}_k$ and a domain $\Omega_k\subset \Sigma_j$ such that $\partial \Omega_k=(\Sigma_j\cap \partial B_{R_2})\cup \gamma_k$. Then we have from (\ref{e-I-j}) that
		\begin{equation*}
			\sup_{\Sigma_j\cap \gamma_k}\left|u_j\exp\left(-\frac{x_1}{2}\right)\right|\leq \eta,
		\end{equation*}
		for all $k\geq k_0$ for possibly larger $k_0$. Thus
		\begin{equation*}
			\inf_{\gamma_k}(f_B+h_{\mu}-u_j)>0.
		\end{equation*}
		By applying the maximum principle to $\Omega_k$ for all $k$ sufficiently large, it follows that
		\begin{equation*}
			u_j(x)\leq B|x|^{-\beta}\exp\left(-\frac{|x|}{2}\right)+\eta\exp\left(\frac{x_1}{2}\right) \ \ \text{for  all} \ |x|\geq R_2.
		\end{equation*}
		Letting $\eta\to 0$, we obtain
		\begin{equation*}
			u_j(x)\leq B|x|^{-\beta}\exp\left(-\frac{|x|}{2}\right) \ \ \text{for  all} \ |x|\geq R_2,
		\end{equation*}
		i.e.,
		\begin{equation*}
			\sin^2\frac{\alpha_{I_j}}{2}\leq B|x|^{-\beta}\exp\left(-\frac{|x|}{2}-\frac{x_1}{2}\right) \ \ \text{for  all} \ |x|\geq R_2.
		\end{equation*}
This proves the desired estimate and completes the proof of the theorem.
\hfill $\square$

	\vspace{.2in}

	\section{Proof of Theorem \ref{maintheorem}}

	\vspace{.1in}
	In this section, our aim is to prove Theorem \ref{maintheorem}.  First, we establish a decay estimate for $|\overline{\nabla}J_{\Sigma}|^2$, which, by (\ref{e-J-nablaalpha}), yields a gradient estimate for the Kähler angle.

	\begin{theorem}\label{thm-DJ}
		Suppose $\Sigma$ is a complete translating soliton in ${\mathbb R}^4$ satisfying the assumptions in Theorem \ref{maintheorem}, then there exists a constant $C$ depending on $\Lambda$ and $\delta$ such that
		\begin{equation}\label{e-DJ}
			|\overline{\nabla}J_{\Sigma}|^2(x) \leq C|x|^{-\frac{\beta}{4}},\ \ for\ \ all \ |x|\geq R_2+2,
		\end{equation}
where $\beta$ and $R_2$ are given in Theorem \ref{prop-alpha-decay}.
	\end{theorem}

	\vspace{.1in}
	
\begin{proof}
By Proposition \ref{p-vanishing of second fundamental form}, we know $|{\bf A}|\le \Lambda$ for some $\Lambda=\Lambda(\Sigma)$.  So,
by  Moser's iteration estimate for the equation
\begin{eqnarray*}
			\Delta|\overline{\nabla}J_{\Sigma}|^2+\langle \nabla x_1,\nabla|\overline{\nabla}J_{\Sigma}|^2\rangle
			&\geq& -5|\overline{\nabla}J_{\Sigma}|^2|{\bf A}|^2.
		\end{eqnarray*}
         we know that, for any $x\in\Sigma$
	\begin{equation}\label{e-DJ-2}
		|\overline{\nabla}J_{\Sigma}|^2(x) \leq \sqrt{2}\Lambda C_1\left(\int_{\Sigma\cap B_1(x)}|\overline{\nabla}J_{\Sigma}|^2\right)^{\frac{1}{2}}.
	\end{equation}
	It suffices to estimate the right hand side of (\ref{e-DJ-2}). For this purpose, note that since $|x|\geq R_2+2$, we have $x\in \cup_{j=1}^N\Sigma_j$. Suppose $x\in \Sigma_j$ and without loss of generality, we may assume that $\Sigma\cap B_2(x)\subset \Sigma_j$ (otherwise, we $\Sigma\cap B_2(x)$ by the connected component of $\Sigma\cap B_2(x)$ containing $x$). We also have from (\ref{e-cosalpha}) that
	\begin{equation}\label{e-cosalpha-i}
		\Delta\cos\alpha_{I_j}+\langle\nabla x_1,\nabla\cos\alpha_{I_j}\rangle=-|\overline{\nabla}J_{\Sigma}|^2\cos\alpha_{I_j},
	\end{equation}
where $I_j\in{\mathcal J}$ is given by Proposition \ref{prop-alpha-decay}.

	Choose a cutoff function $\phi\in C_c^{\infty}(\Sigma\cap B_2(x))$ on $\Sigma$ such that $\phi\equiv 1$ on $\Sigma\cap B_1(x)$ and $|\nabla\phi|\leq 2$. Multiplying both sides of (\ref{e-cosalpha-i}) by $\phi\cos\alpha_{I_j}$ and integrating by parts yield
	\begin{eqnarray*}
		& & \int_{\Sigma}|\overline{\nabla}J_{\Sigma}|^2\cos^2\alpha_{I_j}\phi\\
		&=& -\int_{\Sigma}\phi\cos\alpha_{I_j}\Delta\cos\alpha_{I_j}-\int_{\Sigma}\phi\cos\alpha_{I_j}\langle\nabla x_1,\nabla\cos\alpha_{I_j}\rangle\\
		&=&\int_{\Sigma}\phi|\nabla\cos\alpha_{I_j}|^2+\int_{\Sigma}\cos\alpha_{I_j}\nabla\phi\cdot\nabla\cos\alpha_{I_j}-\int_{\Sigma}\phi\cos\alpha_{I_j}\langle\nabla x_1,\nabla\cos\alpha_{I_j}\rangle\\
		&=&\int_{\Sigma}\phi\sin^2\alpha_{I_j}|\nabla\alpha_{I_j}|^2+\int_{\Sigma}\langle -\cos\alpha_{I_j}\nabla\phi+\phi\cos\alpha_{I_j}\nabla x_1,\nabla\alpha_{I_j}\rangle \sin\alpha_{I_j}.
	\end{eqnarray*}
	Notice that $|\nabla\alpha_{I_j}|\leq |\overline{\nabla}J_{\Sigma}|\leq \sqrt{2}\Lambda$ and Theorem \ref{prop-alpha-decay},
	\begin{equation}\label{e-sin-alpha}
		\sin^2\alpha_{I_j}(y)\leq 4\sin^2\frac{\alpha_{I_j}}{2}\leq 4B|y|^{-\beta}
	\end{equation}
	and
	\begin{equation*}
		\cos^2\alpha_{I_j}\geq \frac{1}{2}
	\end{equation*}
	hold for $y\in \Sigma\cap B_2(x)$ if $R_1$ is sufficiently large. Therefore, we have
	\begin{eqnarray}\label{e-DJ-integral}
		\int_{\Sigma\cap B_1(x)}|\overline{\nabla}J_{\Sigma}|^2 \leq C\int_{\Sigma\cap B_1(x)} |y|^{-\frac{\beta}{2}} \leq C(|x|-2)^{-\frac{\beta}{2}}\leq C'|x|^{-\frac{\beta}{2}},
	\end{eqnarray} since $|x|\geq R_2+2$.
	Here, we use assumption (2) in Theorem \ref{maintheorem}.
	Now the theorem follows from  (\ref{e-DJ-integral}) and (\ref{e-DJ-2}).
\end{proof}
	
	\vspace{.1in}
	\begin{corollary}\label{cor-dalpha}
		Under the assumption of Theorem \ref{thm-DJ}, there exists a constant $C>0$ depending on $\Lambda$ and $\delta$ such that for any $J\in{\mathcal J}$, the following estimate holds:
		\begin{equation}\label{e-Dalpha}
			|\nabla\sin\alpha_J|^2(x) \leq C|x|^{-\frac{\beta}{4}},\ \  \ all \ |x|\geq R_2+2.
		\end{equation}
In particular, on $\Sigma_j$, we have
		\begin{equation}\label{e-Dalpha-2}
			|\nabla\sin\alpha_{I_j}|^2(x) \leq C|x|^{-\frac{\beta}{4}},
		\end{equation}
for all $|x|\geq R_2+2$, where $I_j\in{\mathcal J}$ is given by Theorem \ref{prop-alpha-decay}.
	\end{corollary}

	\vspace{.1in}

	Now we can prove Theorem \ref{maintheorem}.

	\vspace{.1in}

\begin{proof}[ Proof of Theorem \ref{maintheorem}] By Proposition \ref{prop-A-H}, it suffices to show that ${\bf H}\equiv 0$ on $\Sigma$. Following the approach in \cite{NT}, we construct an exhaustion of $\Sigma$ by compact sets $K_n$, and show that
		\begin{equation}\label{e-Kn}
			\lim_{n\to\infty}\oint_{\partial K_n}\langle \nu,{\bf T}\rangle d\sigma=0
		\end{equation}
		where $\nu$ denotes the exterior unit normal to $\partial K_n$ in $\Sigma$. The theorem follows because
		\begin{eqnarray*}
			\oint_{\partial K_n}\langle \nu,{\bf T}\rangle d\sigma
			=\oint_{\partial K_n}\langle \nu,\nabla x_1\rangle d\sigma=\int_{K_n}\Delta x_1 d\mu=\int_{K_n}|{\bf H}|^2d\mu.
		\end{eqnarray*}

        \vspace{.1in}

By Section 3, there exists a compact set $K\subset \Sigma$ such that $\Sigma\backslash K$ decomposes into $N$ connected components $\Sigma_1,\cdots, \Sigma_N$, each of genus zero. Moreover, the proof of Theorem \ref{prop-alpha-decay} yields that for every $1\leq j\leq N$, there is a compatible complex structure $I_{j}\in {\mathcal J}$ with $\cos\alpha_{I_j}\to 1$ at infinity  of $\Sigma_j$.

After a suitable coordinate change, we may assume that $I_j$ takes the form of $J_1$ given in (\ref{e-J-10}) and
		\begin{equation*}
			P_{I_j}:=\mathrm{span}\{{\bf{T}}, {I_j}(\bf{T}) \}=\{(u,v,0,0)|(u,v)\in {\mathbb R}^2\}.
		\end{equation*}

Next, we estimate the distance of $T_x\Sigma$ and $P_{I_j}$.  The observations are the following geometric meaning of the mean curvature and the K\"ahler angle.
\begin{align}
&\text{ if } {\bf H}(x)=0, \quad \text{ then }   {\bf T^{\bot}}=0, \quad \text{ hence }  {\bf T}\in T_x\Sigma.\\
&\text{ if } \cos\alpha_J(x)=1, \quad \text{ then } \langle Je_1,e_2\rangle =1, \quad \text{ which implies } Je_1=\pm e_2\in T_x\Sigma.
\end{align}
Combining them together, we know
\begin{align}\label{determine the tangent plane}
\text{ if } {\bf H}(x)=0  \text{ and } \cos\alpha_J(x)=1, \quad \text{ then }\quad T_x\Sigma=\mathrm{span}\{{\bf{T}}, {I_j}({\bf T}) \}=P_{I_j}.
 \end{align}

 Now, for $x\in \Sigma_j$, by Proposition \ref{p-vanishing of second fundamental form} and Theorem \ref{prop-alpha-decay},  we know
 \begin{align*}
 \lim_{\Sigma_j\ni x\to \infty}|\cos\alpha_{I_j}(x)-1|+|{\bf H}(x)|=0,
 \end{align*}
which combining a perturbation of the observation \eqref{determine the tangent plane} gives
\begin{align}\label{tangent limit}
\lim_{\Sigma_j\ni x\to \infty}d_{G(2,4)}(T_x\Sigma, P_{I_j})=0.
\end{align}

Next, we are going to show  that \eqref{tangent limit} implies  there exists $R_2>0$ such that  $\Sigma_j\backslash B_{R_2}(0)$ can be written as the graph of a multi-function defined on $P_{I_j}$ with its gradient bounded by some small constant $S_j$ and tending to zero as $x\to \infty$.

In fact, by Proposition \ref{p-vanishing of second fundamental form} and \cite[Lemma 2.4]{CM}, we know there exists  $R_0>0$ and $r_0>0$ such that for any $x\in \Sigma_j\backslash B_{R_0}$, $\hat{B}_{r_0}(x)$ can be written as a graph over $T_x\Sigma$ with gradient and Hessian estimate. By \eqref{tangent limit},  we can rotate the coordinate such that the graph is over $P_{I_j}$. More precisely, if we denote $\pi: \mathbb{R}^{4}\to P_{I_j}$ be the orthogonal projection, then for any fix $\epsilon_0$,  $R_0$ and $r_0$ can be chosen such that for any $x\in \Sigma_j\backslash B_{R_0}$,  there exists $\Omega(x,r_0)\subset P_{I_j}$ and a function $u: \Omega (x,r_0)\to P_{I_j}^{\bot}$ such that
\begin{align}\label{local graph}
\hat{B}_{r_0}(x)={\rm Graph}_u|_{\Omega(x,r_0)}, \ \Omega(x,r_0)\supset D_{\frac{r_0}{2}}(
\pi(x)
) \ \text{  and  }
r_0^2|D^2u|+r_0|Du|\le \epsilon_0.
\end{align}

Now, choose $R_1\gg R_0$ and denote $X=\Sigma_j\backslash \pi^{-1}(D_{R_1})$, we are going to prove $\pi: X\to Y:=P_{I_j}\backslash D_{R_1}$ is a covering map by definition. That is,
$\pi:X\to Y$ is surjective and for any $y\in Y$, there exists $y\in U\subset Y$ such that $\pi^{-1}(U)=\cup_{\alpha\in \Lambda_y}U_\alpha$  such that $\pi:U_\alpha \to U$ is diffeomorphism.

In fact, for any $y\in \pi(X)$, we know there exists $x\in X$ such that  $\pi(x)=y$. By \eqref{local graph}, we know
$$\pi(\hat{B}_{r_0}(x))=\Omega(x,r_0)\supset D_{\frac{r_0}{2}}(y).$$
As a result, $\mathring{D}_{\frac{r_0}{2}}(y)\cap Y\subset \pi(X)$ is an open neighborhood of $y$ in $Y$, which implies $\pi(X)$ is open in $Y$.  Moreover, for any $y_i\in \pi(X)$ satisfying $y_i\to y_\infty\in Y$, by choosing $i$ large such that $|y_i-y_\infty|\le \frac{r_0}{4}$, then \eqref{local graph} again implies $y_\infty\in Y\cap D_{\frac{r_0}{4}}(y_i)\subset \pi(X)$. So, $\pi(X)$ is also closed in $Y$. Noting that $Y$ is connected, we know $\pi(X)=Y$. Thus $\pi:X\to Y$ is surjective.   Furthermore, for any $y\in Y$ and  $x\in \pi^{-1}(y)$, since $\pi:\hat{B}_{r_0}(x)\to \Omega(x,r_0)$ is diffeomorphism, we know there exists $r_1\ll r_0$ such that $\hat{B}_{r_1}(x)\cap \hat{B}_{r_1}(x')=\emptyset$ for any $x\neq x'\in \pi^{-1}(y)$, since otherwise, there exists $z\in \hat{B}_{r_1}(x)\cap \hat{B}_{r_1}(x')$, which combining with \eqref{local graph} implies $d_g(x,x')\le C\sqrt{1+|\epsilon|^2} r_1<r_0$ (by choosing $r_1=(2C)^{-1}r_0$), which contradicts to $\pi: \hat{B}_{r_0}(x)\to \Omega(x,r_0)$ is differomorphism.  So,  $\{\hat{B}_{r_1}(x)\backslash \pi^{-1}(D_{R_0})\}_{x\in \pi^{-1}(y)}$ are disjoint open sets  in $X$ and $\cap_{x\in \pi^{-1}(y)}\Omega(x,r_1)\cap Y\supset \mathring{D}(y,\frac{r_1}{2})\cap Y=:U$ is an open neighborhood of $y$ in $Y$ satisfying $\pi^{-1}(U)=\cup_{\alpha\in \Lambda_y:= \pi^{-1}(y)}U_\alpha$ such that $\pi: U_\alpha\to U$ is diffeomorphism. Here $U_\alpha=\pi^{-1}(U)\cap \hat{B}_{r_1}(x)$ for any $x\in \pi^{-1}(y)$. As a result, $\pi: X\to Y$ is a covering map.  That is, $\Sigma_j\backslash \pi^{-1}(D_{R_1})$ is a multi-graph over $Y=P_{I_j}\backslash D_{R_1}$ with gradient estimate. Moreover, the are bound assumption $\mathrm{Area}(\Sigma\cap B_R(x))\le \Lambda R^2$ implies the multiplicity $n_j$ of the graph is bounded by $C\Lambda$. Especially, the curve $\Gamma_{R_1}:=\pi^{-1}(\partial D_{R_1})$ is a compact homotopically non-trivial curve on $\Sigma_j$.  So, $\Gamma_{R_1}$ and $\partial \Sigma_j$ bounds a compact domain in $\Sigma_j$, hence $\pi^{-1}(D_{R_1})\cap \Sigma_j$ is compact. We can take $R_2\gg R_1$ so that $\Sigma_j\backslash B_{R_2}\subset \Sigma_j\backslash \pi^{-1}(D_{R_1})$ is contained in the multi-graph defined on $P_{I_j}$ with its gradient bounded by some small constant $S_j$ and tending to zero as $x\to \infty$ and multiplicity $n_j\le C\Lambda$. More precisely,
$\Sigma_j\backslash B_{R_2}(0)$ can be expressed as
\begin{align*}
\{(r\cos\theta, r\sin\theta, f(r,\theta), g(r,\theta))\quad | \quad  r\ge R_2 \quad \theta \in[0,2\pi n_j]\}.
 \end{align*}
 Without loss of generality, when a ray is removed in the domain,  we can choose a single-value branch of the form
		\begin{equation*}
			\{(u,v,f(u,v),g(u,v)) \ \ \ \ {\rm with} \  (u,v)\in  \mathbb{R}^2\backslash (D_{R_2}\cup [0,\infty)\times \{0\}\},
		\end{equation*}
		for some functions $f,g$ with $|Df|$ and $|Dg|$ uniformly bounded by $S_j$ and $\lim_{u,v\to \infty}(|Df|+|Dg|)=0$, where the coordinate $x_1$ equals $u$ and the coordinate $y_1$ equals $v$, and in this coordinate, $I_j$ takes the form $J_1$ given by (\ref{e-J-10}).

		\vspace{.1in}
		
		By a direct computation, we see that the K\"ahler angle $\alpha$ of $\Sigma={\rm Graph}_{f,g}$ is given by (see (\ref{graph-cos-1}))
		\begin{equation}\label{graph-cos}
			\cos\alpha=\frac{1+f_ug_v-f_vg_u}{\sqrt{\det G}}
		\end{equation}
		where $G$ is the induced metric on $\Sigma$ and
		\begin{equation}\label{graph-det}
			\det G=1+f_u^2+f_v^2+g_u^2+g_v^2+(f_ug_v-f_vg_u)^2.
		\end{equation}
		In particular, the plane $\{(u,v,0,0)|(u,v)\in {\mathbb R}^2\}$ has constant K\"ahler angle $\bar \alpha=0$. Therefore,
		\begin{equation}\label{graph-sin-0}
			\sin^2\alpha=1-\cos^2\alpha=\frac{(f_u-g_v)^2+(f_v+g_u)^2}{\det G}.
		\end{equation}
		Since $|Df|$ and $|Dg|$ are uniformly bounded, there exists a constant $D$ such that
		\begin{equation}\label{graph-sin}
			|f_u-g_v|+|f_v+g_u|\leq D|\sin\alpha|.
		\end{equation}

		\vspace{.1in}
		
		In each of the connected components $\Sigma_j$, denote by $\gamma_{j,r}$ the lift to $\Sigma_j$ of the path on $P_{I_j}$ given by
		\begin{equation*}
			c_j(t)=(r\cos t,r\sin t), \ \ 0\leq t\leq 2n_j\pi.
		\end{equation*}
		There is $t_0=t_0(r)$ such that for every $t_1\leq 2n_j\pi$, we can find functions $f$ and $g$ for which
		\begin{equation*}
			\gamma_{j,r}(t)=(r\cos t,r\sin t,f,g), \ \ t_1-t_0\leq t\leq t_1+t_0.
		\end{equation*}

		Denote $M=\left(\begin{array}{cc}
			f_u & f_v \\
			g_u & g_v
		\end{array}\right)$. Then
		\begin{eqnarray*}
			\gamma_{j,r}'(t)
			&=& \left(-r\sin t,r\cos t,f_u(-r\sin t)+f_vr\cos t,g_u(-r\sin t)+g_vr\cos t\right)\\
			&=& r\left(-\sin t,\cos t,-f_u\sin t+f_v\cos t,-g_u\sin t+g_v\cos t\right)\\
			&=&r\left(\partial_t,M(\partial_t)\right),
		\end{eqnarray*}
		where $\partial_t=(-\sin t,\cos t)$.

		Furthermore, denote $\bar{\nu}={\rm Proj}_{P_{I_j}}\nu$. Since
		\begin{equation*}
			\nu\in T\Sigma={\rm span}\{(1,0,f_u,g_u), (0,1,f_v,g_v)\},
		\end{equation*}
		there exists $a,b\in {\mathbb R}$, such that
		\begin{eqnarray*}
			\nu
			&=& a(1,0,f_u,g_u)+b(0,1,f_v,g_v)=(a,b,f_ua+f_vb,g_ua+g_vb)\\
			&=&(\bar\nu,M(\bar \nu)).
		\end{eqnarray*}
		
		Denote $\partial_r=(\cos t,\sin t)$ and assume that
		\begin{eqnarray*}
			\bar\nu&=&A_0\partial_r+B_0\partial_t=A_0(\cos t,\sin t)+B_0(-\sin t,\cos t)\\
			&=&(A_0\cos t-B_0\sin t, A_0\sin t+B_0\cos t).
		\end{eqnarray*}
		Then \begin{eqnarray*}
			\nu
			&=&(A_0\cos t-B_0\sin t, A_0\sin t+B_0\cos t,\\
			& & \ f_u(A_0\cos t-B_0\sin t)+f_v(A_0\sin t+B_0\cos t)\\
			& &  g_u(A_0\cos t-B_0\sin t)+g_v(A_0\sin t+B_0\cos t)).
		\end{eqnarray*}
		It is easy to check that
		\begin{eqnarray*}
			0
			&=&r^{-1}\langle\gamma'_{j,r}(t),\nu\rangle\\
			&=& -\sin t(A_0\cos t-B_0\sin t)+\cos t(A_0\sin t+B_0\cos t)\\
			& & +\left(-f_u\sin t+f_v\cos t\right)\left[f_u(A_0\cos t-B_0\sin t)+f_v(A_0\sin t+B_0\cos t)\right]\\
			& & +\left(-g_u\sin t+g_v\cos t\right)\left[g_u(A_0\cos t-B_0\sin t)+g_v(A_0\sin t+B_0\cos t)\right]\\
			&=& B_0\left[1+(f_u^2+g_u^2)\sin^2t+(f_v^2+g_v^2)\cos^2t-2(f_uf_v+g_ug_v)\sin t\cos t\right]\\
			& & +A_0\left[(f_v^2-f_u^2+g_v^2-g_u^2)\sin t\cos t+(f_uf_v+g_ug_v)(\cos^2t-\sin^2t)\right]
		\end{eqnarray*}
		and
		\begin{eqnarray*}
			1
			&=&|\nu|^2\\
			&=& (A_0\cos t-B_0\sin t)^2+(A_0\sin t+B_0\cos t)^2\\
			& & +\left[f_u(A_0\cos t-B_0\sin t)+f_v(A_0\sin t+B_0\cos t)\right]^2\\
			& & +\left[g_u(A_0\cos t-B_0\sin t)+g_v(A_0\sin t+B_0\cos t)\right]^2\\
			&=& A_0^2+B_0^2+\left[(f_u\cos t+f_v\sin t)^2+(g_u\cos t+g_v\sin t)^2\right]A_0^2\\
			& & +\left[(-f_u\sin t+f_v\cos t)^2+(-g_u\sin t+g_v\cos t)^2\right]B_0^2\\
			& &+2\left[(f_u\cos t+f_v\sin t)(-f_u\sin t+f_v\cos t)\right.\\
			& & \left.+(g_u\cos t+g_v\sin t)(-g_u\sin t+g_v\cos t)\right]A_0B_0.
		\end{eqnarray*}
		 By a direct computation,  we obtain that
		\begin{eqnarray*}
			& &A_0\cos t-B_0\sin t\\
			&=&\frac{(1+f_v^2+g_v^2)\cos t-(f_uf_v+g_ug_v)\sin t}{\sqrt{\det G}\sqrt{1+(f_u^2+g_u^2)\sin^2t+(f_v^2+g_v^2)\cos^2t-2(f_uf_v+g_ug_v)\sin t\cos t}}.
		\end{eqnarray*}
		Therefore,
		\begin{eqnarray*}
			\langle \nu, {\bf T}\rangle |\gamma'_{j,r}(t)|
			&=&(A_0\cos t-B_0\sin t)r\\
			& &  \cdot \sqrt{1+(f_u^2+g_u^2)\sin^2t+(f_v^2+g_v^2)\cos^2t-2(f_uf_v+g_ug_v)\sin t\cos t}\\
			&=&r\frac{(1+f_v^2+g_v^2)\cos t-(f_uf_v+g_ug_v)\sin t}{\sqrt{\det G}}\\
			&=&\frac{1+f_v^2+g_v^2}{\sqrt{\det G}}r\cos t-\frac{f_uf_v+g_ug_v}{\sqrt{\det G}}r\sin t\\
			&=& r\cos t+R,
		\end{eqnarray*}
		where
		\begin{equation*}
			R=\frac{1+f_v^2+g_v^2-\sqrt{\det G}}{\sqrt{\det G}}r\cos t-\frac{f_uf_v+g_ug_v}{\sqrt{\det G}}r\sin t
		\end{equation*}
		and $\det G$ is given by (\ref{graph-det}). Therefore, we have
		\begin{eqnarray}\label{e-gammajr}
			\left|\oint_{\gamma_{j,r}}\langle \nu,e_1\rangle d\sigma\right|
			&=&\left|\int_0^{2\pi}\langle \nu,e_1\rangle  |\gamma'_{j,r}(t)|dt\right|\nonumber\\
			&=&\left|\int_0^{2\pi}Rdt\right|\leq \int_0^{2\pi}|R|dt.
		\end{eqnarray}

		\vspace{.1in}

		Now we are going to estimate $\int_0^{2\pi}|R|dt$. Notice that we can rewrite
		\begin{eqnarray*}
			f_uf_v+g_ug_v=f_v(f_u-g_v)+g_v(f_v+g_u)
		\end{eqnarray*}
		and
		\begin{eqnarray*}
			&&(1+f_v^2+g_v^2-\sqrt{\det G})(1+f_v^2+g_v^2+\sqrt{\det G})\\
			&=&(1+f_v^2+g_v^2)^2-\det G\\
			&=&(f_v-g_u)(f_v+g_u)-(f_u+g_v)(f_u-g_v)\\
			& & +[f_v(f_v-g_u)+g_v(f_u+g_v)][f_v(f_v+g_u)-g_v(f_u-g_v)].
		\end{eqnarray*}
		Combining (\ref{graph-sin}) with the fact that $|Df|$ and $|Dg|$ are uniformly bounded, we see that
		\begin{eqnarray*}
			|R|\leq Cr|\sin\alpha|(|\cos t|+|\sin t|)
		\end{eqnarray*}
		for some constant $C$. Hence
		\begin{equation}\label{e-R}
			\int_0^{2\pi}|R|dt\leq C\int_{0}^{2\pi}r|\sin\alpha||\cos t|dt+C\int_{0}^{2\pi}r|\sin\alpha||\sin t|dt:=C(I+II),
		\end{equation}
		where
		\begin{equation*}
			I=\int_{0}^{2\pi}r|\sin\alpha||\cos t|dt, \ II=\int_{0}^{2\pi}r|\sin\alpha||\sin t|dt.
		\end{equation*}
		
		\vspace{.1in}
		
		We first estimate $II$. By Theorem \ref{prop-alpha-decay}, we have for $r\geq R_0$ that
		\begin{equation*}
			\sin^2\alpha(x)\leq 4\sin^2\left(\frac{\alpha(x)}{2}\right)\leq 4B|x|^{-\beta}\exp(-\frac{|x|}{2}-\frac{x_1}{2}).
		\end{equation*}
		Then we compute for every $\delta_1>0$ small
		\begin{eqnarray*}
			II
			&=&\int_{0}^{2\pi}r|\sin\alpha||\sin t|dt\\
			&=&\int_{0}^{\pi-\delta_1}r|\sin\alpha||\sin t|dt+\int_{\pi+\delta_1}^{2\pi}r|\sin\alpha||\sin t|dt+\int_{\pi-\delta_1}^{\pi+\delta_1}r|\sin\alpha||\sin t|dt\\
			&\leq& \int_{0}^{\pi-\delta_1}2B^{\frac{1}{2}}rr^{-\frac{\beta}{2}}e^{-\frac{r}{4}(1+\cos t)}dt+\int_{\pi+\delta_1}^{2\pi}2B^{\frac{1}{2}}rr^{-\frac{\beta}{2}}e^{-\frac{r}{4}(1+\cos t)}dt\\
			& & +\int_{\pi-\delta_1}^{\pi+\delta_1}2B^{\frac{1}{2}}rr^{-\frac{\beta}{2}} e^{-\frac{r}{4}(1+\cos t)}|\sin t|dt\\
			&\leq & 4\pi B^{\frac{1}{2}}r^{1-\frac{\beta}{2}}e^{-\frac{r}{4}(1-\cos \delta_1)}+III,
		\end{eqnarray*}
		where
		\begin{equation*}
			III=\int_{\pi-\delta_1}^{\pi+\delta_1}2B^{\frac{1}{2}}r^{1-\frac{\beta}{2}} e^{-\frac{r}{4}(1+\cos t)}|\sin t|dt.
		\end{equation*}
		Choose $\delta_1$ sufficiently small so that
		\begin{equation*}
			\frac{(t-\pi)^2}{4}\leq 1+\cos t \leq (t-\pi)^2
		\end{equation*}
		for $t\in [\pi-\delta_1,\pi+\delta_1]$. Then
		\begin{eqnarray*}
			III
			&\leq&\int_{\pi-\delta_1}^{\pi+\delta_1}2B^{\frac{1}{2}}r^{1-\frac{\beta}{2}} e^{-\frac{r}{4}\frac{(t-\pi)^2}{4}}|\sin t|dt\\
			&=&2B^{\frac{1}{2}}r^{1-\frac{\beta}{2}}\int_{-\delta_1}^{\delta_1}e^{-\frac{rt^2}{16}}|\sin t|dt\\
			&=&4B^{\frac{1}{2}}r^{1-\frac{\beta}{2}}\int_{0}^{\delta_1}e^{-\frac{rt^2}{16}}\sin tdt\\
			&\leq&4B^{\frac{1}{2}}r^{1-\frac{\beta}{2}}\int_{0}^{\delta_1}e^{-\frac{rt^2}{16}}tdt\\
			&=& 32B^{\frac{1}{2}}r^{-\frac{\beta}{2}}\int_{0}^{\frac{r\delta_1^2}{16}}e^{-s}ds
			\leq  32B^{\frac{1}{2}}r^{-\frac{\beta}{2}}.
		\end{eqnarray*}
		Therefore,
		\begin{equation}\label{e-II}
			II\leq 4\pi B^{\frac{1}{2}}r^{1-\frac{\beta}{2}}e^{-\frac{r}{4}(1-\cos \delta_1)}+ 32B^{\frac{1}{2}}r^{-\frac{\beta}{2}}.
		\end{equation}

		\vspace{.1in}
		
		Next, we estimate $I$. As above, we have
		\begin{eqnarray*}
			I
			&=&\int_{0}^{2\pi}r|\sin\alpha||\cos t|dt\\
			&=&\int_{0}^{\pi-\delta_2}r\sin\alpha|\cos t|dt+\int_{\pi+\delta_2}^{2\pi}r|\sin\alpha||\cos t|dt+\int_{\pi-\delta_2}^{\pi+\delta_2}r|\sin\alpha||\cos t|dt\\
			&\leq& \int_{0}^{\pi-\delta_2}2B^{\frac{1}{2}}rr^{-\frac{\beta}{2}}e^{-\frac{r}{4}(1+\cos t)}dt+\int_{\pi+\delta_2}^{2\pi}2B^{\frac{1}{2}}rr^{-\frac{\beta}{2}}e^{-\frac{r}{4}(1+\cos t)}dt\\
			& & +\int_{\pi-\delta_2}^{\pi+\delta_2}r\sin\alpha|\cos t|dt\\
			&\leq & 4\pi B^{\frac{1}{2}}r^{1-\frac{\beta}{2}}e^{-\frac{r}{4}(1-\cos \delta_2)}+IV,
		\end{eqnarray*}
		where
		$
			IV=\int_{\pi-\delta_2}^{\pi+\delta_2}r|\sin\alpha||\cos t|dt.
		$
		To proceed further, we set
		\begin{equation*}
			I_1=\{\sin\alpha\geq0\}\cap[\pi-\delta_2,\pi+\delta_2], \ \ I_2=\{\sin\alpha<0\}\cap[\pi-\delta_2,\pi+\delta_2].
		\end{equation*}
		Then we compute
		\begin{eqnarray*}
			IV
			&=&\int_{\pi-\delta_2}^{\pi+\delta_2}r|\sin\alpha||\cos t|dt=-\int_{\pi-\delta_2}^{\pi+\delta_2}r|\sin\alpha|\cos tdt\\
			&=&-\int_{I_1}r\sin\alpha\cos tdt+\int_{I_2}r\sin\alpha\cos tdt\\
			&=&-\int_{I_1}r\sin\alpha d\sin t+\int_{I_2}r\sin\alpha d\sin t\\
			&=& -r\sin\alpha \sin t |_{\partial I_1}+ r\int_{I_1} \sin t d\sin\alpha +r\sin\alpha \sin t |_{\partial I_2}- r\int_{I_2} \sin t d\sin\alpha \\
			&\leq & 4B^{\frac{1}{2}}r^{1-\frac{\beta}{2}}e^{-\frac{r}{4}(1-\cos \delta_2)}+r\int_{I_1} \sin t \langle \nabla\sin\alpha,\frac{\partial}{\partial t}\rangle dt-r\int_{I_2} \sin t \langle \nabla\sin\alpha,\frac{\partial}{\partial t}\rangle dt\\
			&\leq & 4B^{\frac{1}{2}}r^{1-\frac{\beta}{2}}e^{-\frac{r}{4}(1-\cos \delta_2)}+r\int_{\pi-\delta_2}^{\pi+\delta_2} |\sin t| |\nabla\sin\alpha|dt\\
			&= & 4B^{\frac{1}{2}}r^{1-\frac{\beta}{2}}e^{-\frac{r}{4}(1-\cos \delta_2)}+r\int_{-\delta_2}^{\delta_2} |\sin t| |\nabla\sin\alpha|dt\\
			&\leq& 4B^{\frac{1}{2}}r^{1-\frac{\beta}{2}}e^{-\frac{r}{4}(1-\cos \delta_2)}+Cr^{1-\frac{\beta}{8}}\int_{-\delta_2}^{\delta_2} |t|dt\\
			&=&4B^{\frac{1}{2}}r^{1-\frac{\beta}{2}}e^{-\frac{r}{4}(1-\cos \delta_2)}+Cr^{1-\frac{\beta}{8}}\delta_2^2.
		\end{eqnarray*}
		Here we have used Corollary \ref{cor-dalpha} and the fact that $\sin\alpha=0$ on $(\pi-\delta_2,\pi+\delta_2)\cap(\partial I_1\cup \partial I_2)$. For fixed $\sigma>0$ small enough, we choose $\delta_2$ such that $r^{1-\frac{\beta}{8}}\delta_2^2=r^{-\sigma}$, i.e., $\delta_2=r^{-\frac{1}{2}-\frac{\sigma}{2}+\frac{\beta}{16}}$. Then we have
		\begin{equation}\label{e-I}
			I\leq 4\pi B^{\frac{1}{2}}r^{1-\frac{\beta}{2}}e^{-\frac{r}{4}(1-\cos r^{-\frac{1}{2}-\frac{\sigma}{2}+\frac{\beta}{16}})}+4B^{\frac{1}{2}}r^{1-\frac{\beta}{2}}e^{-\frac{r}{4}(1-\cos r^{-\frac{1}{2}-\frac{\sigma}{2}+\frac{\beta}{16}})}+Cr^{-\sigma}.
		\end{equation}

		\vspace{.1in}

		Plugging (\ref{e-II}) and (\ref{e-I}) into (\ref{e-R}) yields
		\begin{eqnarray}\label{e-R-2}
			\int_0^{2\pi}|R|dt
			&\leq& 4\pi B^{\frac{1}{2}}Cr^{1-\frac{\beta}{2}}e^{-\frac{r}{4}(1-\cos \delta_1)}+ 32B^{\frac{1}{2}}Cr^{-\frac{\beta}{2}}\nonumber\\
			&  & + (4\pi B^{\frac{1}{2}}+4B^{\frac{1}{2}})Cr^{1-\frac{\beta}{2}}e^{-\frac{r}{4}(1-\cos r^{-\frac{1}{2}-\frac{\sigma}{2}+\frac{\beta}{16}})}+Cr^{-\sigma}.
		\end{eqnarray}

		\vspace{.1in}
		
		\noindent \textbf{Claim:} if $\sigma<\frac{\beta}{8}$, then we have
		\begin{equation*}
			\lim_{r\to\infty}r^{1-\frac{\beta}{2}}e^{-\frac{r}{4}(1-\cos r^{-\frac{1}{2}-\frac{\sigma}{2}+\frac{\beta}{16}})}=0.
		\end{equation*}
	 In fact,
	 \begin{eqnarray*}
	 1-\cos r^{-\frac{1}{2}-\frac{\sigma}{2}+\frac{\beta}{16}}\approx\frac{1}{2}r^{-1-\sigma+\frac{\beta}{8}}, \ \  as \ r \gg 1,
	 \end{eqnarray*}
 so \begin{eqnarray*}
 	e^{-\frac{r}{4}(1-\cos r^{-\frac{1}{2}-\frac{\sigma}{2}+\frac{\beta}{16}})} \approx e^{-\frac{1}{8}r^{\frac{\beta}{8}-\sigma}},  \ \  as \ r \gg 1.
 \end{eqnarray*} Since $\sigma<\frac{\beta}{8}$, it is easy to see that
\begin{eqnarray*}
	\lim_{r\to\infty}\frac{r^{1-\frac{\beta}{2}}}{e^{\frac{1}{8}r^{\frac{\beta}{8}-\sigma}}}=0.
\end{eqnarray*}		
		 This proves the claim.
		
		\vspace{.1in}

		With $\sigma$ so chosen, applying the claim to (\ref{e-R-2}) yields
		\begin{eqnarray*}
			\lim_{r\to\infty}\int_0^{2\pi}|R|dt=0.
		\end{eqnarray*}
		We therefore get from (\ref{e-gammajr}) that
		\begin{equation*}
			\lim_{r\to\infty}\left|\oint_{\gamma_{j,r}}\langle \nu,e_1\rangle d\sigma\right|=0.
		\end{equation*}
		As each $\Sigma_j\backslash B_{R_2}(0)$ is a graph over $P_{I_j}$ with derivatives bounded by $S_{j}$, there exists a sequence of compact sets $K_n$ forming an exhaustion of $\Sigma$ such that
		\begin{equation*}
			\partial K_n=\gamma_{1,r_n}\cup\cdots\cup \gamma_{N,r_n},
		\end{equation*}
		where $\{r_n\}_{n\in {\mathbb N}}$ is a sequence that converges to infinity. This finishes the proof of the theorem.
\end{proof}

	\vspace{.2in}

	\section{Example of Symplectic Translating Solitons}

	\vspace{.1in}
	
In \cite{JLT}, Joyce-Lee-Tsui constructed  very interesting examples for Lagrangian translating solitons with the oscillation of the Lagrangian angle arbitrarily small. Their construction is as follows (Corollary I of \cite{JLT}):

For given constants $\beta>0$ and $a>0$, define
\begin{equation}\label{e-phi}
\phi(y)=\int_0^y\frac{dt}{\left(\frac{1}{a}+t^2\right)\sqrt{P(t)}},
\end{equation}
where
		\begin{equation*}
P(t)=\frac{1}{t^2}\left[(1+at^2)e^{\beta t^2}-1\right].
		\end{equation*}
Set
\begin{equation}\label{e-Sigma}
\Sigma=\left\{\left(x\sqrt{\frac{1}{a}+y^2}e^{i\phi(y)}, \frac{y^2-x^2}{2}-\frac{i}{\beta}\left[\phi(y)+\arg\left(y+iP(y)^{-\frac{1}{2}}\right)\right]\right): x,y\in {\mathbb R}\right\}.
\end{equation}
Joyce-Lee-Tsui proved that $\Sigma$ is an embedded Lagrangian translating soliton in ${\mathbb C}^2$ with translating vector ${\bf T}=(0,0,\beta,0)\in {\mathbb C}^2={\mathbb R}^4$, which is diffeomorphic to ${\mathbb R}^2$.
Here $\Sigma$ is Lagrangian with respect to the standard complex structure
\begin{equation}\label{e-J-0}
J_0=\left(\begin{array}{cccc}
    0 & -1 & 0 & 0 \\
    1 & 0  & 0 & 0 \\
    0 & 0  & 0 & -1 \\
    0 & 0  & 1 & 0 \\
\end{array}
\right)
\end{equation}
Furthermore, since $\phi$ is monotonically increasing in $y$, there exists $\bar\phi\in (0,\frac{\pi}{2})$ such that $\phi(y)\to\bar\phi$ as $y\to\infty$ and $\phi(y)\to-\bar\phi$ as $y\to-\infty$. Joyce-Lee-Tsui showed that for fixed $\beta>0$, the map $a\mapsto \bar\phi$ is a 1-1 correspondence from $(0,\infty)$ to $(0,\frac{\pi}{2})$. The Lagrangian angle of $\Sigma$ varies between $\bar\phi$ and $\pi-\bar\phi$. In particular, the oscillation of the Lagrangian angle satisfies
\begin{equation}\label{e-Langrangian-angle-osc}
{\rm osc}_{\Sigma}\theta
\leq \pi-2\bar\phi.
\end{equation}
Hence by choosing $\bar\phi$ close to $\frac{\pi}{2}$, the oscillation of the Lagrangian angle can be made arbitrarily small.

In this section, we show that by choosing a suitable compatible complex structure, $\Sigma$ can be realized as a symplectic translating soliton whose Kähler angle becomes arbitrarily close to $0$. Furthermore, the total curvature of $\Sigma$ is infinite.

\begin{theorem}\label{thm-example}
There is a compatible complex structure $J_1$ in ${\mathbb C}^2$, such that the K\"ahler angle $\alpha$ of the surface $\Sigma$ defined by (\ref{e-Sigma}) lies between $-\frac{\pi}{2}+\bar\phi$ and $\frac{\pi}{2}-\bar\phi$. In particular, by choosing $\bar\phi$ close to $\frac{\pi}{2}$, we know that for any $\varepsilon >0$, there is a symplectic translating soliton $\Sigma$ with respect to $J_1$, such that $\cos\alpha\geq 1-\varepsilon$.

Furthermore, the translating soliton $\Sigma$ constructed above has bounded second fundamental, genus 0, first Betti number 0, extrinsic quadratic area growth, and
\begin{equation*}
      \lim_{r\to\infty}\int_{\Sigma\cap\{r<|{\bf x}|<2r\}}|{\bf H}|^2d\mu>0.
\end{equation*}
In particular,
 $$
  \int_{\Sigma}|{\bf H}|^2d\mu=\infty.
 $$
\end{theorem}

	\vspace{.1in}

 The theorem implies that, the assumption (3) in Theorem \ref{maintheorem} is not superfluous.

\begin{remark} From the proof, we see that we can actually choose
\begin{equation*}
J_1=\left(\begin{array}{cccc}
    0 & 0  & 0 & -1 \\
    0 & 0  & -1 & 0 \\
    0 & 1  & 0 & 0 \\
    1 & 0  & 0 & 0 \\
\end{array}
\right)
\end{equation*}
The construction also implies that  for any $J\in\mathcal{J} $,  there holds
$\lim_{\Sigma\ni p\to \infty}\cos\alpha_J(p)\neq 1.$ This means the assumption $(3)$ in Theorem \ref{maintheorem} is key  for the limit  decay estimate  of $\alpha_J$ in Theorem \ref{prop-alpha-decay}.
\end{remark}

\begin{proof}
It is known that the compactible complex structure of ${\mathbb C}^2$ is of the form (see, for example, \cite{HLY})
\begin{equation*}
J=\left(\begin{array}{cccc}
    0 & A  & B & C \\
    -A & 0  & -C & B \\
    -B & C  & 0 & -A \\
    -C & -B  & A & 0 \\
\end{array}
\right)
\end{equation*}
or
\begin{equation*}
J=\left(\begin{array}{cccc}
    0 & A  & B & C \\
    -A & 0  & C & -B \\
    -B & -C  & 0 & A \\
    -C & B  & -A & 0 \\
\end{array}
\right)
\end{equation*}
where $A^2+B^2+C^2=1$. For convience, we will consider
\begin{equation}
J:=J_{A,B,C}=\left(\begin{array}{cccc}
    0 & A  & B & C \\
    -A & 0  & C & -B \\
    -B & -C  & 0 & A \\
    -C & B  & -A & 0 \\
\end{array}
\right).
\end{equation}
It is obvious that $J_0=J_{-1,0,0}$. We will compute tha K\"ahler angle of $\Sigma$ with respect to $J$.

The parametrization of $\Sigma$ can be expressed as
\begin{equation}\label{e-Sigma-F}
F(x,y)=\left(x\sqrt{\frac{1}{a}+y^2}\cos\phi(y), x\sqrt{\frac{1}{a}+y^2}\sin\phi(y), \frac{y^2-x^2}{2}, -\frac{1}{\beta}\left[\phi(y)+\gamma(y)\right]\right),
\end{equation}
for $x,y\in {\mathbb R}$, where
\begin{equation}\label{e-gamma}
\gamma(y)=\arg\left(y+iP(y)^{-\frac{1}{2}}\right).
\end{equation}
Notice that
		\begin{equation*}
\lim_{t\to 0}P(t)=\lim_{t\to 0}\frac{1}{t^2}\left[(1+at^2)e^{\beta t^2}-1\right]=\beta+a>0.
		\end{equation*}
Hence
\begin{equation}\label{e-gamma-2}
\gamma(y)=\arg\left(y+iP(y)^{-\frac{1}{2}}\right)=
\begin{cases}
   \arctan\frac{1}{\sqrt{(1+ay^2)e^{\beta y^2}-1}}, & {\rm if} \ y>0;\\
    \frac{\pi}{2},  & {\rm if} \ y=0;\\
   -\arctan\frac{1}{\sqrt{(1+ay^2)e^{\beta y^2}-1}}+\pi, & {\rm if} \ y<0.\\
\end{cases}
\end{equation}

\vspace{.1in}

We first consider the case that $y>0$. In this case, by direct computation, we see that
		\begin{equation*}
\frac{d\gamma(y)}{dy}=-\frac{y}{\left(\frac{1}{a}+y^2\right)\sqrt{(1+ay^2)e^{\beta y^2}-1}}-\frac{\beta y}{\sqrt{(1+ay^2)e^{\beta y^2}-1}}.
		\end{equation*}
Therefore, we have
		\begin{equation}\label{e-phi-gamma}
\frac{d}{dy}\left[\phi(y)+\gamma(y)\right]=-\frac{\beta y}{\sqrt{(1+ay^2)e^{\beta y^2}-1}}.
		\end{equation}
It is clear that the tangent space of $\Sigma$ is spanned by
		\begin{equation*}
F_x=\left(\sqrt{\frac{1}{a}+y^2}\cos\phi(y), \sqrt{\frac{1}{a}+y^2}\sin\phi(y), -x, 0\right)^T
		\end{equation*}
and
\begin{equation*}
F_y=\left(\begin{array}{c}
  \frac{xy}{\sqrt{\frac{1}{a}+y^2}}\left(\cos\phi(y)-\frac{\sin\phi(y)}{\sqrt{(1+ay^2)e^{\beta y^2}-1}}\right) \\
  \frac{xy}{\sqrt{\frac{1}{a}+y^2}}\left(\sin\phi(y)+\frac{\cos\phi(y)}{\sqrt{(1+ay^2)e^{\beta y^2}-1}}\right)\\
   y\\
  \frac{y}{\sqrt{(1+ay^2)e^{\beta y^2}-1}}\\
\end{array}
\right).
\end{equation*}
The induced metric on $\Sigma$ is given by
\begin{equation*}
g_{xx}=\langle F_x,F_x\rangle=x^2+y^2+\frac{1}{a},
\end{equation*}
\begin{equation*}
g_{xy}=\langle F_x,F_y\rangle=0,
\end{equation*}
\begin{equation*}
g_{yy}=\langle F_y,F_y\rangle=\frac{\left(x^2+y^2+\frac{1}{a}\right)ay^2e^{\beta y^2}}{(1+ay^2)e^{\beta y^2}-1}.
\end{equation*}
Thus,
\begin{equation}\label{E-detg}
{\rm det}g=\frac{\left(x^2+y^2+\frac{1}{a}\right)^2ay^2e^{\beta y^2}}{(1+ay^2)e^{\beta y^2}-1}.
\end{equation}

\vspace{.1in}

For $J=J_{A,B,C}$, we have
\begin{equation*}
JF_{x}=\left(\begin{array}{c}
  A\sqrt{\frac{1}{a}+y^2}\sin\phi(y)-Bx \\
  -A\sqrt{\frac{1}{a}+y^2}\cos\phi(y)-Cx\\
  -B\sqrt{\frac{1}{a}+y^2}\cos\phi(y)-C\sqrt{\frac{1}{a}+y^2}\sin\phi(y)\\
  -C\sqrt{\frac{1}{a}+y^2}\cos\phi(y)+B\sqrt{\frac{1}{a}+y^2}\sin\phi(y)+Ax\\
\end{array}
\right).
\end{equation*}
By a direct calculation, we have
\begin{eqnarray*}
\langle JF_{x}, F_y\rangle=-\frac{y\left(x^2+y^2+\frac{1}{a}\right)\sqrt{ae^{\beta y^2}}}{\sqrt{(1+ay^2)e^{\beta y^2}-1}}\left[B\cos(\phi(y)+\gamma(y))+C\sin(\phi(y)+\gamma(y))\right].
\end{eqnarray*}
Therefore,
\begin{equation}\label{E-cosalpha}
\cos\alpha=\frac{\langle JF_{x}, F_y\rangle}{\sqrt{{\rm det}g}}=-\left[B\cos(\phi(y)+\gamma(y))+C\sin(\phi(y)+\gamma(y))\right].
\end{equation}

In particular, if $B=C=0$, then $\cos\alpha\equiv0$. That is, $\Sigma$ is Lagrangian with respect to $J_0=J_{-1,0,0}$. Now we assume that $B^2+C^2\neq 0$, then (\ref{E-cosalpha}) can be written as
\begin{eqnarray}\label{E-cosalpha-2}
\cos\alpha&=&-\sqrt{B^2+C^2}\cos(\phi(y)+\gamma(y)-\eta_0)\nonumber\\
 &=&\sqrt{B^2+C^2}\cos(\phi(y)+\gamma(y)-\eta_0+\pi),
\end{eqnarray}
with $\eta_0\in [0,2\pi]$ such that
\begin{equation}\label{E-eta}
\cos\eta_0=\frac{B}{\sqrt{B^2+C^2}}, \ \ \sin\eta_0=\frac{C}{\sqrt{B^2+C^2}}.
\end{equation}
So, for any fixed $y$, along the curve $\gamma_y(t)=F(t,y)$, the limit
\begin{align}
\lim_{t\to \infty} \cos\alpha(\gamma_y(t))\equiv\sqrt{B^2+C^2}\cos(\phi(y)+\gamma(y)-\eta_0+\pi)\neq 1 \quad \text{ for much } y.
\end{align}
Especially,  either $\lim_{\Sigma\ni p\to \infty}\cos\alpha(p)=0$ or $\lim_{\Sigma\ni p}\cos\alpha(p)$ does not exists.  In conclusion, for any $J$, $\lim_{\Sigma\ni p\to \infty}\cos\alpha_J(p)\neq 1$.

Notice that $\frac{d}{dy}\left[\phi(y)+\gamma(y)\right]<0$ for $y>0$. Also we have that
\begin{equation*}
\phi(0)=0, \lim_{y\to\infty}\phi(y)=\bar\phi, \gamma(0)=\frac{\pi}{2}, \lim_{y\to\infty}\gamma(y)=0.
\end{equation*}
Hence we have for all $y>0$
\begin{equation*}
\bar\phi\leq \phi(y)+\gamma(y)\leq \frac{\pi}{2},
\end{equation*}
so that
\begin{equation*}
\bar\phi-\eta_0+\pi\leq \phi(y)+\gamma(y)-\eta_0+\pi\leq \frac{3\pi}{2}-\eta_0.
\end{equation*}
Now we choose $A=0, B=0, C=-1$, then $\eta_0=\frac{3\pi}{2}$. With this choice of $\eta_0$ and corresponding complex structure
\begin{equation}\label{e-J-1}
J_1:=J_{0,0,-1}=\left(\begin{array}{cccc}
    0 & 0  & 0 & -1 \\
    0 & 0  & -1 & 0 \\
    0 & 1  & 0 & 0 \\
    1 & 0  & 0 & 0 \\
\end{array}
\right),
\end{equation}
we have
\begin{eqnarray}\label{e-y-p}
\cos\alpha
&=&\cos\left(\phi(y)+\gamma(y)-\frac{\pi}{2}\right)=\cos\left[\frac{\pi}{2}-(\phi(y)+\gamma(y))\right]\nonumber\\
&\geq& \cos\left[\frac{\pi}{2}-\bar\phi\right]=\sin\bar\phi.
\end{eqnarray}

Consequently, the level sets of $\cos\alpha$ are lines.

	\vspace{.1in}

For the case of $y<0$, we can proceed a similar argument to get
\begin{equation}\label{e-y-n}
\cos\alpha\geq\cos\left[\frac{\pi}{2}-\bar\phi\right]=\sin\bar\phi.
\end{equation}

\vspace{.1in}

Observing that $\phi(0)=0$ and $\gamma(0)=\frac{\pi}{2}$, the above computation yields for all $x\in{\mathbb R}$:
\begin{equation}\label{e-y-0}
\cos\alpha(x,0)=1.
\end{equation}

\vspace{.1in}

Combining (\ref{e-y-p}), (\ref{e-y-n}) and (\ref{e-y-0}) together, we finally obtain that
\begin{equation}\label{e-y-all}
\cos\alpha\geq\sin\bar\phi,
\end{equation}
on $\Sigma$.
Hence, we have
\begin{equation*}
-\frac{\pi}{2}+\bar\phi\leq \alpha\leq\frac{\pi}{2}-\bar\phi.
\end{equation*}

In particular, given any $\varepsilon >0$, after fixing $\beta > 0$ and choosing $a > 0$ such that $\bar\phi$ is sufficiently close to $\frac{\pi}{2}$, there exists a translating soliton $\Sigma$ determined by $\beta$ and $a$, which is symplectic with respect to $J_1$, satisfying
$$\cos\alpha\geq \sin\bar\phi \geq 1-\varepsilon.$$

	\vspace{.1in}

As shown by Joyce-Lee-Tsui, $\Sigma$ is diffeomorphic to ${\mathbb R}^2$ and asymptotic to a union of two planes. Hence, the genus of $\Sigma$ is zero, the second fundamental form of $\Sigma$ is uniformly bounded and $\Sigma$ has extrinsic quadratic area growth.

	\vspace{.1in}

Next we compute the second fundamental form of $\Sigma$. Without loss of generality, we  compute it at the points where $y>0$. The other case is similar. First, we choose an orthonormal frame on $\Sigma$ as follows
		\begin{equation*}
\nu_x=\frac{J_0F_x}{|F_x|}=\frac{\left(-\sqrt{\frac{1}{a}+y^2}\sin\phi(y), \sqrt{\frac{1}{a}+y^2}\cos\phi(y), 0, -x\right)^T}{\sqrt{x^2+y^2+\frac{1}{a}}}
		\end{equation*}
and
\begin{equation*}
\nu_y=\frac{J_0F_y}{|F_y|}=\frac{\left(\begin{array}{c}
  -  \frac{xy}{\sqrt{\frac{1}{a}+y^2}}\left(\sin\phi(y)+\frac{\cos\phi(y)}{\sqrt{(1+ay^2)e^{\beta y^2}-1}}\right)\\
\frac{xy}{\sqrt{\frac{1}{a}+y^2}}\left(\cos\phi(y)-\frac{\sin\phi(y)}{\sqrt{(1+ay^2)e^{\beta y^2}-1}}\right) \\
  -\frac{y}{\sqrt{(1+ay^2)e^{\beta y^2}-1}}\\
   y\\
\end{array}
\right)}{\sqrt{\frac{\left(x^2+y^2+\frac{1}{a}\right)ay^2e^{\beta y^2}}{(1+ay^2)e^{\beta y^2}-1}}}.
\end{equation*}
We also have
		\begin{equation*}
F_{xx}=\left(0, 0, -1, 0\right)^T,
		\end{equation*}
		\begin{equation*}
F_{xy}=F_{yx}=\left(\begin{array}{c}
  \frac{y}{\sqrt{\frac{1}{a}+y^2}}\left(\cos\phi(y)-\frac{\sin\phi(y)}{\sqrt{(1+ay^2)e^{\beta y^2}-1}}\right) \\
  \frac{y}{\sqrt{\frac{1}{a}+y^2}}\left(\sin\phi(y)+\frac{\cos\phi(y)}{\sqrt{(1+ay^2)e^{\beta y^2}-1}}\right)\\
   0\\
  0\\
\end{array}
\right),
\end{equation*}
and
\begin{equation*}
F_{yy}=\left(\begin{array}{c}
  \frac{x}{\sqrt{\frac{1}{a}+y^2}\left[(1+ay^2)e^{\beta y^2}-1\right]}\left\{\left(e^{\beta y^2}-1\right)\cos\phi(y)+\frac{(\beta y^2+a\beta y^4-1)e^{\beta y^2}+1}{\sqrt{(1+ay^2)e^{\beta y^2}-1}}\sin\phi(y)\right\} \\
  \frac{x}{\sqrt{\frac{1}{a}+y^2}\left[(1+ay^2)e^{\beta y^2}-1\right]}\left\{\left(e^{\beta y^2}-1\right)\sin\phi(y)-\frac{(\beta y^2+a\beta y^4-1)e^{\beta y^2}+1}{\sqrt{(1+ay^2)e^{\beta y^2}-1}}\cos\phi(y)\right\} \\
   1\\
  \frac{e^{\alpha y^2}(1-\alpha y^2-a\alpha y^4)-1}{\left[(1+ay^2)e^{\beta y^2}-1\right]^{\frac{3}{2}}}\\
\end{array}
\right).
\end{equation*}
Therefore, we have
\begin{equation*}
h_{xx}^x=\langle F_{xx},\nu_x\rangle=0,  \  h_{xx}^y=\langle F_{xx},\nu_y\rangle=\frac{1}{\sqrt{\left(x^2+y^2+\frac{1}{a}\right)ae^{\beta y^2}}},
\end{equation*}
\begin{equation*}
h_{xy}^x=h_{yx}^x=\langle F_{xy},\nu_x\rangle=\frac{y}{\sqrt{x^2+y^2+\frac{1}{a}}\sqrt{(1+ay^2)e^{\beta y^2}-1}},
\end{equation*}
\begin{equation*}
h_{xy}^y=h_{yx}^y=\langle F_{xy},\nu_y\rangle=0,
\end{equation*}
\begin{equation*}
h_{yy}^x=\langle F_{yy},\nu_x\rangle=0,
\end{equation*}
\begin{equation*}
h_{yy}^y=\langle F_{yy},\nu_y\rangle=-\frac{y^2\left[\beta\left(x^2+y^2+\frac{1}{a}\right)+1\right]\sqrt{ae^{\beta y^2}}}{\left[(1+ay^2)e^{\beta y^2}-1\right]\sqrt{x^2+y^2+\frac{1}{a}}}.
\end{equation*}
Hence, we have
\begin{equation*}
H^{x}=g^{ij}h^x_{ij}=g^{xx}h_{xx}^x+2g^{xy}h_{xy}^x+g^{yy}h_{yy}^x=0,
\end{equation*}
\begin{eqnarray*}
H^{y}
=g^{ij}h^y_{ij}=g^{xx}h_{xx}^y+2g^{xy}h_{xy}^y+g^{yy}h_{yy}^y
=-\frac{\beta}{\sqrt{\left(x^2+y^2+\frac{1}{a}\right)ae^{\beta y^2}}}.
\end{eqnarray*}
In particular, the mean curvature vector of $\Sigma$ is given by
\begin{eqnarray}\label{E-mean-curvature}
{\bf H}=-\frac{\beta}{\sqrt{\left(x^2+y^2+\frac{1}{a}\right)ae^{\beta y^2}}}\nu_y
={\bf T}^{\perp}
\end{eqnarray}
where ${\bf T}=(0,0,\beta,0)$. Hence, $\Sigma$ is a translating soliton with translating vector ${\bf T}$.

\vspace{.1in}

Next, we compute the norm of the second fundamental form of $\Sigma$. We have
\begin{eqnarray*}
|{\bf A}|^2
&=&g^{xx}g^{xx}h^x_{xx}h^x_{xx}+2g^{xx}g^{yy}h^x_{xy}h^x_{xy}+g^{yy}g^{yy}h^x_{yy}h^x_{yy}\\
& & +g^{xx}g^{xx}h^y_{xx}h^y_{xx}+2g^{xx}g^{yy}h^y_{xy}h^y_{xy}+g^{yy}g^{yy}h^y_{yy}h^y_{yy}\\
&=&2g^{xx}g^{yy}h^x_{xy}h^x_{xy}+g^{xx}g^{xx}h^y_{xx}h^y_{xx}+g^{yy}g^{yy}h^y_{yy}h^y_{yy}\\
&=& \frac{3+\left[\beta\left(x^2+y^2+\frac{1}{a}\right)+1\right]^2}{\left(x^2+y^2+\frac{1}{a}\right)^3ae^{\beta y^2}}.
\end{eqnarray*}
We also have
\begin{eqnarray*}
|\overline{\nabla}J_{\Sigma}|^2
&=& |{\bf A}|^2+2(K_{1234}-R_{1234})\\
&=&|{\bf A}|^2-\frac{2}{\sqrt{\det(g)}}g^{ij}\left(h^{3}_{1i}h^{4}_{2j}-h^{3}_{2i}h^{4}_{1j}\right)\\
&=& \frac{3+\left[\beta\left(x^2+y^2+\frac{1}{a}\right)+1\right]^2}{\left(x^2+y^2+\frac{1}{a}\right)^3ae^{\beta y^2}}-\frac{2\left[\beta\left(x^2+y^2+\frac{1}{a}\right)+2\right]}{\left(x^2+y^2+\frac{1}{a}\right)^3ae^{\beta y^2}}\\
&=& \frac{\beta^2}{ae^{\beta y^2}\left(x^2+y^2+\frac{1}{a}\right)}=|{\bf H}|^2.
\end{eqnarray*}
Furthermore, we proceed to compute $|\nabla\cos\alpha|^2$. From (\ref{e-phi-gamma}) and (\ref{e-y-p}), we compute
\begin{eqnarray*}
|\nabla\cos\alpha|^2
&=& g^{yy}\left(\frac{\partial\cos\alpha}{\partial y}\right)^2\\
&=& g^{yy}\left(\frac{d}{dy}\sin\left(\phi(y)+\gamma(y)\right)\right)^2\\
&=& g^{yy}\cos^2\left(\phi(y)+\gamma(y)\right)\left(\frac{d}{dy}\left(\phi(y)+\gamma(y)\right)\right)^2\\
&=& (1-\cos^2\alpha) \frac{(1+ay^2)e^{\beta y^2}-1}{\left(x^2+y^2+\frac{1}{a}\right)ay^2e^{\beta y^2}}\cdot \frac{\beta^2 y^2}{(1+ay^2)e^{\beta y^2}-1}\\
&=& \sin^2\alpha\frac{\beta^2}{ae^{\beta y^2}\left(x^2+y^2+\frac{1}{a}\right)}.
\end{eqnarray*}
In particular,
\begin{eqnarray}\label{e-H-J-alpha}
|\nabla\alpha|^2=\frac{\beta^2}{ae^{\beta y^2}\left(x^2+y^2+\frac{1}{a}\right)}=|{\bf H}|^2=|\overline{\nabla}J_{\Sigma}|^2.
\end{eqnarray}

By (\ref{e-Sigma-F}), we see that
\begin{eqnarray*}
r^2<|F(x,y)|^2<4r^2
\end{eqnarray*}
if and only if
\begin{eqnarray*}
r^2<\frac{(x^2+y^2)^2}{4}+\frac{1}{a}x^2+\frac{\left[\phi(y)+\gamma(y)\right]^2}{\beta^2}<4r^2.
\end{eqnarray*}
Therefore, for $r\geq 1$, there are constants $\sigma_1<\sigma_2$ such that
\begin{eqnarray*}
F(B_{\sigma_1\sqrt{r}}^{{\mathbb R}^2}(0))\subset \Sigma \cap B_{r}^{{\mathbb R}^4}(0) \subset F(B_{\sigma_2\sqrt{r}}^{{\mathbb R}^2}(0)).
\end{eqnarray*}
Therefore, using
 \begin{eqnarray*}
d\mu=\sqrt{\frac{\left(x^2+y^2+\frac{1}{a}\right)^2ay^2e^{\beta y^2}}{(1+ay^2)e^{\beta y^2}-1}}dxdy,
\end{eqnarray*}
and the fact that
$$\frac{ay^2e^{\beta y^2}}{(1+ay^2)e^{\beta y^2}-1}\leq 1,$$
we see that, for $r\geq 1$
\begin{eqnarray*}
{\rm Area}(\Sigma\cap B_r(0))
&=&\int_{\Sigma\cap B_r(0)}d\mu\\
&\leq&\int_{B_{\sigma_2\sqrt{r}}^{{\mathbb R}^2}(0)}\sqrt{\frac{\left(x^2+y^2+\frac{1}{a}\right)^2ay^2e^{\beta y^2}}{(1+ay^2)e^{\beta y^2}-1}}dxdy\\
&\leq &\int_{B_{\sigma_2\sqrt{r}}^{{\mathbb R}^2}(0)}\left(x^2+y^2+\frac{1}{a}\right)dxdy\\
&=& \pi\left(\frac{\sigma_2^4r^2}{2}+\frac{\sigma_2^2r}{a}\right)\leq \Lambda r^2
\end{eqnarray*}
for some positive constant $\Lambda$. This implies that $\Sigma$ has quadratic area growth.

Furthermore, there exist $0<\lambda_1<\lambda_2<\lambda_3<\lambda_4$ such that for $r$ large,
\begin{eqnarray*}
F(B_{\lambda_3\sqrt{r}}^{{\mathbb R}^2}(0)\backslash B_{\lambda_2\sqrt{r}}^{{\mathbb R}^2}(0))\subset \Sigma \cap (B_{2r}^{{\mathbb R}^4}(0)\backslash B_{r}^{{\mathbb R}^4}(0)) \subset F(B_{\lambda_4\sqrt{r}}^{{\mathbb R}^2}(0)\backslash B_{\lambda_1\sqrt{r}}^{{\mathbb R}^2}(0)).
\end{eqnarray*}
Therefore,
\begin{eqnarray*}
&& \int_{\Sigma\cap \{r<|{\bf x}|<2r\}}|\overline{\nabla}J_{\Sigma}|d\mu\\
&\geq &\int_{B_{\lambda_3\sqrt{r}}^{{\mathbb R}^2}(0)\backslash B_{\lambda_2\sqrt{r}}^{{\mathbb R}^2}(0)}\sqrt{\frac{\beta^2}{\left(x^2+y^2+\frac{1}{a}\right)ae^{\beta y^2}}}\sqrt{\frac{\left(x^2+y^2+\frac{1}{a}\right)^2ay^2e^{\beta y^2}}{(1+ay^2)e^{\beta y^2}-1}}dxdy\\
&=& \int_{B_{\lambda_3\sqrt{r}}^{{\mathbb R}^2}(0)\backslash B_{\lambda_2\sqrt{r}}^{{\mathbb R}^2}(0)}\sqrt{\frac{\beta^2\left(x^2+y^2+\frac{1}{a}\right)y^2}{(1+ay^2)e^{\beta y^2}-1}}dxdy\\
&\geq&\lambda_2\sqrt{r}\int_{-1}^{1}\int_{\lambda_2\sqrt{r}}^{\sqrt{\lambda_3^2r-1}}\beta\sqrt{\frac{y^2}{(1+ay^2)e^{\beta y^2}-1}}dxdy\\
&=&\lambda_2\sqrt{r}(\sqrt{\lambda_3^2r-1}-\lambda_2\sqrt{r})\int_{-1}^{1}\beta^2\sqrt{\frac{y^2}{ae^{\beta y^2}\left[(1+ay^2)e^{\beta y^2}-1\right]}}dy.
\end{eqnarray*}
Thus,
\begin{eqnarray*}
  &&    \lim_{r\to\infty}\left(r^{-1}\int_{\Sigma\cap\{r<|{\bf x}|<2r\}}|\overline{\nabla}J_{\Sigma}|d\mu\right)\\
&\geq& \lambda_2(\lambda_3-\lambda_2)\int_{-1}^{1}\beta^2\sqrt{\frac{y^2}{ae^{\beta y^2}\left[(1+ay^2)e^{\beta y^2}-1\right]}}dy>0.
\end{eqnarray*}
The H\"older inequality implies that
\begin{equation*}
      \lim_{r\to\infty}\int_{\Sigma\cap\{r<|{\bf x}|<2r\}}|\overline{\nabla}J_{\Sigma}|^2d\mu>0.
\end{equation*}
In particular, we have
\begin{eqnarray*}
\int_{\Sigma}|{\bf H}|^2d\mu=\int_{\Sigma}|\nabla\alpha|^2d\mu=\int_{\Sigma}|\overline{\nabla}J_{\Sigma}|^2d\mu=\infty.
\end{eqnarray*}
\end{proof}


\vspace{.2in}

    {\small\begin{thebibliography}{999}
			 \bibitem{Allard}
             \ \ \ William K. Allard, {\em On the first variation of a varifold}, Ann. of Math.(2){\bf 95}(1972),417-491.

			\bibitem {AB1}
			\ \  \  P. Aldrigo and Z. Balogh, {\em P\'olya-Szeg\"o inequalities on submanifolds with small total mean curvature}, J. Geom. Anal., {\bf  35} (2025), no. 12, Paper No. 394, 23 pp.
			
			
			\bibitem {BK}
			\ \ \ R. Bamler and B. Kleiner, {\em On the Multiplicity One Conjecture for Mean Curvature Flows of surfaces}, arXiv:2312.02106 [math.DG].


                 \bibitem {BS}
			\ \ \ C. Bao and Y. Shi, {\em Gauss maps of translating solitons of mean curvature flow}, Proc. Amer. Math. Soc., {\bf 142} (2014), no. 12, 4333-4339.


             \bibitem {Ber}
			\ \ \ L. Bers, {\em Local behavior of solutions of general linear elliptic equations}, Comm. Pure Appl. Math., {\bf 8} (1955), 473–496.



           \bibitem{B}
           \ \ \ S. Brendle {\em The isoperimetric inequality for a minimal submanifold in Euclidean space}, J. Amer. Math. Soc., {\bf 34} (2021), no. 2, 595–603.


         \bibitem {CHLS}		
			\ \ \ J. Chen, X. Han, J. Li and J. Sun, \emph{Tangent flows of symplectic mean curvature flows}, Journal of Mathematical Study, {\bf 59} (1) (2026), 40-59.
		
			
			\bibitem {CL1}
			\ \ \ J. Chen and J. Li, {\em Mean curvature flow of surfaces in 4-manifolds}, Adv. Math., {\bf 163} (2001), 287-309.
			
			\bibitem{ChenWen}
			\ \ \ Y. Chen and Y. Wen, {\em An $\epsilon$-regularity theorem for mean curvature flow in higher codimension}, preprint.
			
		
			\bibitem {CW}
			\ \ \ S. S. Chern and J. Wolfson, {\em Minimal surfaces by moving frames}, Amer. J. Math. {\bf 105} (1983), 59-83.
			
			\bibitem {CM}
			\ \ \ 	T. H. Colding and W. P. Minicozzi II, {A course in minimal surfaces},Grad. Stud. Math., 121, American Mathematical Society, Providence, RI, 2011. xii+313 pp.


			\bibitem {CM1}
			\ \ \ 	T. H. Colding and W. P. Minicozzi II, {\em Generic mean curvature flow I; Generic gingularities}, Ann. of Math., {\bf 175} (2012), 755-833.
			


			
             	\bibitem {CM2}
				\ \ \ 	T. H. Colding and W. P. Minicozzi II, {\em Uniqueness of blowups and Lojasiewicz inequalities}, Ann. of Math., (2) {\bf 182} (2015), no. 1, 221-285.


			
             	\bibitem {CM3}
				\ \ \ 	T. H. Colding and W. P. Minicozzi II, {\em The singular set of mean curvature flow with generic singularities}, Invent. Math., {\bf 204} (2016), no. 2, 443-471.


			
             	\bibitem {CM4}
				\ \ \ 	T. H. Colding and W. P. Minicozzi II, {\em Complexity of parabolic systems}, Publ. Math. Inst. Hautes Etudes Sci., {\bf 132} (2020), 83-135.


			
			 \bibitem {Ecker}
			\ \ \ K. Ecker,
            {\em On regularity for mean curvature flow of hypersurfaces}, Calc. Var. Partial Differential Equations, {\bf 3} (1995), no. 1, 107–126.



			
			\bibitem {GHH}
			\ \ \ C. Gui, H. Jian and H. Ju, {\em Properties of translating solutions to mean curvature flow}, Discrete Contin. Dyn. Syst., {\bf 28} (2010), no. 2, 441-453.
			
			\bibitem {Hamilton}
			\ \ \  R. Hamilton, {\em Harnack estimate for the mean curvature flow}, J. Diff. Geom., {\bf 41}(1) (1995), 215-226.
			
	
			
			
			\bibitem{HL2}
			\ \ \ X. Han and J. Li, {\em Translating solitons to symplectic and Lagrangian mean curvature flows}, Internat. J. Math., {\bf 20} (2009), no. 4, 443-458.
			
			\bibitem{HL3}
			\ \ \ X. Han and J. Li, {\em Symplectic critical surfaces in K\"ahler surfaces}, J. Eur. Math. Soc., {\bf 12} (2010), 505-527.
			
			


 			\bibitem{HLY}
			\ \ \ X. Han, J. Li and L. Yang, {\em Symplectic mean curvature flow in $\mathbb{CP}^{2}.$} Calc. Var. PDE, \textbf{48} (2013), 111-129.
			



			
			\bibitem {Huisken1}
			\ \ \  G. Huisken, {\em Flow by mean curvature of convex surfaces into spheres}, J. Diff. Geom., {\bf 20} (1984), 237-266.
			
			
			\bibitem {Huisken2}
			\ \ \  G. Huisken, {\em Asymptotic behavior for singularities of the mean curvature flow}, J. Diff. Geom., {\bf 31} (1990), 285-299.
			
			
			\bibitem{HuS1}
			\ \ \ G. Huisken and C. Sinestrari, {\em Convexity estimates for mean curvature flow and singularities of mean convex surfaces}, Acta Math. {\bf 183} (1999), no. 1, 45-70.
			
			
			
			\bibitem{HuS2}
			\ \ \ G. Huisken and C. Sinestrari, {\em Mean curvature flow singularities for mean convex surfaces}, Calc. Var. Partial Differential Equations {\bf 8} (1999), no. 1, 1-14.
			
			
			\bibitem{Il1}
			\ \ \ T. Ilmanen, {\em Singularity of mean curvature flow of surfaces}, preprint.
			
			
			
			
			\bibitem {JLC}
			\ \ \ H. Jian, Q. Liu and X. Chen, {\em Convexity and symmetry of translating solitons in mean curvature flows}, Chinese Ann. Math. Ser. B, {\bf 26} (2005), no. 3, 413-422.
			
			
			\bibitem {JJLW}
			\ \ \ H. Jian, H. Ju, Y. Liu and W. Sun, {\em Symmetry of translating solutions to mean curvature flows}, Acta Math. Sci. Ser. B (Engl. Ed.), {\bf 30} (2010), no. 6, 2006-2016.
			
			
			
			\bibitem {JLT}
			\ \ \   D. Joyce, Y.-I. Lee and M.-P. Tsui, {\em Self-similar solutions and translating solitons for Lagrangian mean curvature flow}, J. Differential Geom., {\bf 84} (2010), no. 1, 127-161.

			\bibitem {Khan}
			\ \ \   I. Khan, {\em The structure of translating surfaces with finite total curvature}, Calc. Var. Partial Differential Equations, {\bf 62} (2023), no. 3, Paper No. 104, 28 pp.

            \bibitem{KP}
            \ \ \ S. G. Krantz and H. R. Parks, {\em A primer of real analytic functions}, Second edition. Birkhäuser Advanced Texts: Basler Lehrbücher.  Birkhäuser Boston, Inc., Boston, MA, 2002.

            \bibitem{Langer}
            \ \ \ J. Langer  {\em A compactness theorem for surfaces with $L^p$-bounded second fundamental form},  Math. Ann., {\bf 270}(2), 223–234 (1985)

           \bibitem {Li-Schoen}
			\ \ \
            P. Li and R. Schoen, {\em $L^p$ and mean value properties of subharmonic functions on Riemannian manifolds}, Acta Math. {\bf 153} (1984), no. 3-4, 279–301
            \bibitem{Lo}
          \ \ \  S. Lojasiewicz,  {\em Ensembles Semi-analytiques}, IHES lecture notes (1965)

			\bibitem {LT}
			\ \ \ S. Lynch and G. Tinaglia, {\em Translators asymptotic to planes}, J. Geom. Anal., {\bf 36} (2026), no. 1, Paper No. 27, 7 pp.
			
			\bibitem {MM}
			\ \ L. Ma and V. Miquel, {\em Bernstein theorem for translating solitons of hypersurfaces}, Manuscripta Math., {\bf 162} (2020), no. 1-2, 115-132.
			
			\bibitem{Morrey}
            \ \ \ C. B. Morrey, {\em Multiple integrals in the calculus of variations}, Reprint of the 1966 edition. Classics in Mathematics. Springer-Verlag, Berlin, 2008. x+506 pp.

			\bibitem {MS}
			\ \ \ J. Michael and L. Simon, {\em Sobolev and mean value inequalities on generalized submanifolds of ${\mathbb R}^n$}, Commun. Pure Appl. Math., {\bf 26} (1973), 361-379.

            \bibitem{MonSe}
			\ \ \ A.  Mondino and D.  Semola, {\em  Polya-Szego inequality and Dirichlet p-spectral gap for non-smooth spaces with Ricci curvature bounded below} J. Math. Pures Appl. (9){\bf 137}  (2020), 238–274.

            \bibitem{MS1}
            \ \ \ S. M\"uller and V. \v{S}ver\'{a}k, {\em On surfaces of finite total curvature}, J. Differential Geom. {\bf 42}(2): (1995) 229-258.


			
			\bibitem {NT}
			\ \ \ A. Neves and G, Tian, {\em Translating solutions to Lagrangian mean curvature flow}, Trans. AMS, {\bf 365} (2013), no. 11, 5655-5680.
			
			

			
			
			\bibitem {SY}
			\ \ \ R. Schoen and S. T. Yau, {\em Lectures on differential geometry}, Conf. Proc. Lecture Notes Geom. Topology, I, International Press, Cambridge, MA, 1994, v+235 pp.
			
			\bibitem{Simon}
            \ \ \ L. Simon, {\em Theorems on regularity and singularity of energy minimizing maps.} Based on lecture notes by Norbert Hungerbühler. Lectures in Mathematics ETH Zürich. Birkhäuser Verlag, Basel, 1996.
		
			\bibitem{SZ}
            \ \ \ J. Sun and J. Zhou, {\em Compactness of surfaces in $\mathbb{R}^n$ with small total curvature.} J. Geom. Anal. {\bf 31} (2021), no. 8, 8238-8270.
			
			
			
		
			
			\bibitem {WangXJ}
			\ \ \ X. Wang, {\em Convex solutions to the mean curvature flow}, Ann. of Math., (2) {\bf 173} (2011), no.3, 1185-1239.
			
			

			
			
			\bibitem{Xin}
			\ \ \ Y. Xin, {\em Translating solitons of the mean curvature flow}, Calc. Var. Partial Differential Equations {\bf 54} (2015), no. 2, 1995-2016.
			
	\end{thebibliography}}
\end{document}